\documentclass[twoside,a4paper]{article}
\usepackage[T1]{fontenc}
\usepackage[utf8]{inputenc}

\usepackage{amssymb}
\usepackage{graphicx}
\usepackage[cmtip,all]{xy}
\usepackage[dvipsnames]{xcolor}
\usepackage{amsmath,bm}
\usepackage{mathtools}
\usepackage[makeroom]{cancel}
\usepackage{fancyhdr}
\usepackage[colorlinks,allcolors=blue]{hyperref}
\usepackage[noabbrev,capitalize,nameinlink]{cleveref}
\usepackage[frak=esstix]{mathalpha}
\usepackage[margin=18.5mm]{geometry}
\usepackage{mathrsfs}
\usepackage{amsthm}

\usepackage{fancyhdr}
\newtheorem{theorem}{Theorem}[section]
\newtheorem{lemma}[theorem]{Lemma}

\newtheorem{proposition}{Proposition}

\numberwithin{equation}{section}

\DeclareRobustCommand{\rchi}{{\mathpalette\irchi\relax}}
\newcommand{\irchi}[2]{\raisebox{\depth}{$#1\chi$}}

\title{Heat Generation Using Lorentzian Nanoparticles. \\ The Full Maxwell System}
\author{Arpan Mukherjee\footnote{Radon Institute (RICAM), Austrian Academy of
Sciences, Altenbergerstrasse 69, A-4040, Linz, Austria (arpan.mukherjee@oeaw.ac.at). This author is supported by the Austrian Science Fund (FWF): P32660.} \ and Mourad Sini\footnote{Radon Institute (RICAM), Austrian Academy of
Sciences, Altenbergerstrasse 69, A-4040, Linz, Austria (mourad.sini@oeaw.ac.at). This author is partially supported by the Austrian Science Fund (FWF): P32660.}}

\begin{document}
\maketitle
\begin{abstract}
    We analyse and quantify the amount of heat generated by a nanoparticle, injected in a background medium, while excited by incident electromagnetic waves. These nanoparticles are dispersive with electric permittivity following the Lorentz model. The purpose is to determine the quantity of heat generated extremely close to the nanoparticle (at a distance proportional to the radius of the nanoparticle). This study extends our previous results, derived in the 2D TM and TE regimes, to the full Maxwell system. We show that by exciting the medium with incident frequencies close to the Plasmonic or Dielectric resonant frequencies, we can generate any desired amount of heat close to the injected nanoparticle while the amount of heat decreases away from it. These results offer a wide range of potential applications in the areas of photo-thermal therapy, drug delivery, and material science, to cite a few.
    \\
    
    To do so, we employ time-domain integral equations and asymptotic analysis techniques to study the corresponding mathematical model for heat generation. This model is given by the heat equation where the body source term comes from the modulus of the electric field generated by the used incident electromagnetic field. Therefore, we first analyse the dominant term of this electric field by studying the full Maxwell scattering problem in the presence of Plasmonic or All-dielectric nanoparticles. As a second step, we analyse the propagation of this dominant electric field in the estimation of the heat potential. For both the electromagnetic and parabolic models, the presence of the nanoparticles is translated into the appearance of large scales in the contrasts for the heat-conductivity (for the parabolic model) and the permittivity (for the full Maxwell system) between the nanoparticle and its surrounding.    
    
\end{abstract}
\section{Introduction and statement of the results}
\subsection{Motivation}
 It is well known that an electric laser field stimulates surface plasmons at optical frequencies on metallic nanoparticles. In turn, these plasmons produce heat from the absorbed energy that diffuses away from the nanoparticles to raise the temperature of the surrounding medium. In addition to being helpful for analyzing the principles of nanoscale heat transport, the ability to produce point-like heat sources has the potential for various significant uses, including medical therapy, thermal lithography, heat-assisted magnetic recording, \cite{intro2, intro3, intro4, intro1}. Over the years, this phenomenon has been well investigated for plasmonic nanoparticles within the framework of thermo-plasmonics, which has few practical restrictions. All-dielectric resonant nanophotonics is a new discipline of nanophotonics that uses optically generated dielectric resonances to get over those restrictions, \cite{dn2, dn1}. In the current work, we consider both types of nanoparticles based on the Lorentz model and attempt to leverage the optical characteristics of the nanoparticles to generate the desired amount of heat around a nanoparticle. Based on the Lorentz model, the same nanoparticle can have different properties while excited with different regimes of incident frequencies. Let us assume that the nanoparticle is nonmagnetic, meaning that its permeability is non-dispersive and matches with the one of a vacuum, however, its permittivity $\varepsilon_{\mathrm{p}}$ is given by the Lorentz model which can be described as follows
\begin{align}\label{lorenz-model}
    \varepsilon_\mathrm{p}(\omega) = \varepsilon_\infty\Big[1+\dfrac{\omega_\mathrm{p}^2}{\omega_0^2-\omega^2 - i\zeta\omega} \Big]
\end{align}
where $\omega_\mathrm{p}$ is the electric plasma frequency, $\omega_0$ is the undamped resonance frequency, $\zeta$ is the electric damping parameter and $\varepsilon_\infty$ is the electric permittivity of the free space. With such a model, we have the following characterization. If the used incident frequencies are in the band $(\omega_0, \sqrt{\omega^2_0+\omega_p^2})$, then the nanoparticle behaves as a Plasmonic one enjoying a proper sequence of Plasmonic resonant frequencies. But if it is excited with incident frequencies in the band $(0, \omega_0)$, then it behaves as Dielectric nanoparticle which enjoy a proper sequence of Dielectric resonant frequencies.  Such a characterization is shown in Section \ref{the scales}. In addition, with such choices of the incident frequencies, we show that the the quality
factor that we define as $Q:=\frac{\Re{(\varepsilon_\mathrm{p})}}{\Im(\varepsilon_\mathrm{p})}$ has large values for both the plasmonic and dielectric nanoparticles. This $Q$-factor is proportional to the ratio between the oscillation period of
the light and its life time. This indicates how absorbing/diffusing the nanoparticle is. In the sequel, we will choose the incident frequencies so that this $Q$-factor is large but not too large so that it allows the nanoparticle to resonate at certain particular frequencies.  Therefore, with such choices of incident frequencies, the nanoparticle will enhance the exciting incident field in a similar way being it plasmonic or dielectric. As a consequence, the nanoparticle will generate any desired amount of heat in its vicinity. Our goal is to justify these principles and quantify the amount of heat generated by the presence of the nanoparticles in terms of their (tunable) properties. 

\bigskip

\noindent
A first attempt to study this phenomenon goes back to \cite{ProfHabib} where the model is stated in the 2D-TE regime. The authors estimated the heat on the surface of the nanoparticle using semi-formal arguments based on the Laplace transform. In \cite{AM}, we have reconsidered this problem using time-domain techniques and derived the heat generated by both plasmonic and dielectric nanoparticles (in the spirit discussed above). The present work aims to extend the conclusions of \cite{AM} by considering the full Maxwell system instead of the 2D-TM or TE regimes. The outcome is that, indeed, using plasmonic or dielectric nanoparticle, we can estimate the heat generated very close to it, i.e. at distances of the order of radius of the nanoparticle. The amplitude of the generated heat is given in terms of the properties of the used nanoparticle, which can in turn be tuned to reach any desired heat potential around it. At the mathematical analysis level, we follow the approach used in \cite{AM} using time-domain integral equation methods coupled with asymptotic analysis techniques. Compared to \cite{AM}, the challenging difficulty rises in dealing with the full-Maxwell system.

\subsection{The heat generation model using nanoparticles}
In this section, we provide with the necessary mathematical framework formulations and the detailed obtained results are stated. For a bounded domain $\Omega\subseteq \mathbb{R}^3$ of class $\mathcal{C}^2$, the heat generation process using nanoparticles is governed by the following parabolic transmission problem \cite{ProfHabib, baffou}
\begin{equation}\label{eq:heat}
\begin{cases}   \rho \mathrm{c}\frac{\partial u}{\partial t} - \nabla.\ \gamma\nabla u = \frac{\omega}{2\pi} \boldsymbol{\Im}(\varepsilon)|\mathrm{E}|^{2} \ \ \text{in} \ \ (\mathbb{R}^{3}\setminus \partial\Omega) \times (0,\mathrm{T}),\\
 \bm{\mathrm{D}^-_0}u\- - \bm{\mathrm{D}^+_0}u=0    \quad \quad \quad \quad \quad \quad \text{on} \ \partial \Omega,\\
 \gamma_{\mathrm{\mathrm{p}}}\bm{\mathrm{D}^-_\nu}u - \gamma_{\mathrm{m}} \bm{\mathrm{D}^+_\nu}u= 0 \ \ \ \ \ \ \ \ \quad \text{on}\; \partial \Omega, \\
 u(\mathrm{x},0) = 0 \ \ \ \ \ \ \ \ \ \ \ \ \ \ \ \ \ \ \ \ \ \ \ \ \ \ \ \text{for} \ \mathrm{x} \in \mathbb{R}^3,
\end{cases}
\end{equation}
where $\rho = \rho_{\mathrm{p}}\rchi_{\Omega} + \rho_{m}\rchi_{\mathbb{R}^{3}\setminus\overline{\Omega}}$ is the mass density; $\mathrm{c} = \mathrm{c}_{\mathrm{p}}\rchi_{\Omega} + \mathrm{c}_{\mathrm{m}}\rchi_{\mathbb{R}^{3}\setminus\overline{\Omega}}$ is the thermal capacity; $\gamma = \gamma_{\mathrm{p}}\rchi_{\Omega} + \gamma_{m}\rchi_{\mathbb{R}^{3}\setminus\overline{\Omega}}$ is the thermal conductivity and we recall that $\varepsilon = \varepsilon_\mathrm{p} \rchi_{\Omega} + \varepsilon_\mathrm{m}\rchi_{\mathbb{R}^{3}\setminus\overline{\Omega}}$ is the electric permittivity respectively. Here, $\mathrm{T}\in \mathbb{R}$ is the final time of measurement. Given that the host medium is non dispersive, we define $\varepsilon_\mathrm{m} = \varepsilon_\infty \varepsilon'_\mathrm{m}$  as the relative permittivity of the host medium, which is considered to be constant and independent of the incident wave's frequency $\omega$. But its permittivity $\varepsilon_{\mathrm{p}}$ is given by the Lorentz model described in (\ref{lorenz-model}). 
 Moreover, $\bm{\mathrm{D}_\nu}$ denotes the Neumann trace and we use the notation $\bm{\mathrm{D}_\nu}^{\pm}$ indicating $
    \bm{\mathrm{D}_\nu}^{\pm}u(\mathrm{x},\mathrm{t}) = \lim_{\mathrm{h}\to 0}\nabla u(\mathrm{x}\pm \mathrm{h}\nu_\mathrm{x},\mathrm{t})\cdot \nu_\mathrm{x},$
where $\nu$ being the outward normal vector to $\partial\Omega$. Analogously, we indicate $\bm{\mathrm{D}^\pm_0}$ as the interior and exterior Dirichlet trace.
\\

\noindent
The source term $\mathrm{E}$ is the time-harmonic electric field solution to the problem 
\begin{align}\label{modifiedmaxwell}
    \begin{cases}
    \textbf{Curl}\; \mathrm{E} = i\omega\mu\mathrm{H} \quad \text{in}\; \mathbb{R}^3\\
    \textbf{Curl}\; \mathrm{H} = -i\omega \varepsilon\mathrm{E} \quad \text{in}\; \mathbb{R}^3.\\
    \end{cases}
\end{align}
where $\mathrm{H} $ is the related magnetic field. Moreover, we consider the magnetic permeability of the form $\mu = \mu_{\mathrm{p}}\rchi_{\Omega} + \mu_{\mathrm{m}}\rchi_{\mathbb{R}^{2}\setminus\overline{\Omega}}$. We denote by $\mu_\mathrm{m} = \mu_{\infty}\mu'_\mathrm{m}$ to be the relative permeability of the host medium, which is assumed to be constant and independent of the frequency $\omega$ of the incident wave and  $\mu_{\infty}$ is the magnetic permeability of the free space. Next, we assume the nanoparticle to be nonmagnetic, i.e. $\mu_{\mathrm{p}} = \mu_{\infty}\mu'_{m}$.
\\

\noindent
By dividing the first equation by $\mu$ and taking curl, we may also remove the magnetic field from the (\ref{modifiedmaxwell}), yielding the modified equation shown below
\begin{align}\label{Maxwell model}
    \textbf{Curl}\; \frac{1}{\mu}\textbf{Curl}\; \mathrm{E} - \omega^2\varepsilon \mathrm{E} = 0 \quad \text{in}\; \mathbb{R}^3
\end{align}
Moreover, for $\omega \in \mathbb{R}^+$, we say $(\mathrm{E},\mathrm{H})$ is radiating if it satisfies the well-known Silver-Müller radiation condition:
\begin{equation*}
    \lim_{|\mathrm{x}|\to +\infty} |\mathrm{x}|\cdot\big(\mathrm{H} \times \hat{\mathrm{x}}-\sqrt{\frac{\varepsilon}{\mu}}\mathrm{E}\big) = 0.
\end{equation*}
We also consider $(\mathrm{E}^\textbf{in},\mathrm{H}^\textbf{in})$ to be the incident plane wave satisfying 
\begin{equation*}
    \mathrm{E}^\textbf{in} = \mathrm{E}_0^\textbf{in} e^{i\mathrm{k}\; \vartheta\cdot\mathrm{x}} \; \text{and}\; \mathrm{H}^\textbf{in} = \vartheta \times \mathrm{E}_0^\textbf{in} e^{i\mathrm{k}\;\vartheta\cdot\mathrm{x}}, 
\end{equation*}
where the direction of wave propagation $\vartheta\in \mathbb{S}$ (unit sphere in $\mathbb{R}^3$), $\mathrm{E}_0^\textbf{in}\in \mathbb{S}$ is the polarization vector satisfying $\vartheta\cdot \mathrm{E}_0^\textbf{in} = 0$ and $\mathrm{k}= \omega \sqrt{\varepsilon\mu}$ is the wave number with the incidence frequency $\omega.$ 
\\

\noindent
Moreover, it is also assumed that the coefficients $\rho_{\mathrm{p}},\rho_{m},\mathrm{c}_{\mathrm{p}},\mathrm{c}_{\mathrm{m}},\gamma_{\mathrm{p}},\gamma_{m}$ to be piece-wise constants with one constant outside of $\Omega$. We also note that $\boldsymbol{\Im}(\varepsilon) = 0$ in $(\mathbb{R}^{3}\setminus \overline{\Omega})$. 
\\

\noindent
Furthermore, with $\mathrm{T}_{0}$ fixed and $u=0$ for $t<0$, we have $\mathrm{U} = u$ on $\mathbb{R}^3\times (-\infty,T_{0})$, thus to analyze $u\in (0,\mathrm{T}_0$), it suffices to investigate the following governing transmissions heat equations as follows:
\begin{align}\label{heattran}
\begin{cases}   \frac{\rho_{\mathrm{p}}\mathrm{c}_{\mathrm{p}}}{\gamma_{\mathrm{p}}}\frac{\partial \mathrm{U}_{\mathrm{i}}}{\partial t} - \Delta \mathrm{U}_{\mathrm{i}}= \frac{\omega}{2\pi\gamma_{\mathrm{p}}}\boldsymbol{\Im}(\varepsilon_\mathrm{p})|\mathrm{E}|^{2}\rchi_{(0,\mathrm{T_0})} \ \ in \ \ \Omega \times \mathbb{R}\\
\frac{\rho_{\mathrm{m}}\mathrm{c}_{\mathrm{m}}}{\gamma_{\mathrm{m}}}\frac{\partial \mathrm{U}_{\mathrm{e}}}{\partial t} - \Delta \mathrm{U}_{\mathrm{e}} = 0 \ \ \ \ \ \ \ \ \ \ \ \ \ \ in \ \mathbb{R}^{3}\setminus \overline\Omega \times \mathbb{R}\\
  \bm{\mathrm{D}^-_0}\mathrm{U}_{\mathrm{i}} \ - \ \bm{\mathrm{D}^+_0}\mathrm{U}_{\mathrm{e}}=0    \ \ \ \ \ \ \ \ \ \ \ \ \  on \ \partial \Omega \times \mathbb{R},\\
 \gamma_\mathrm{p}\bm{\mathrm{D}^-_\nu}\mathrm{U}_{\mathrm{i}}\ - \ \gamma_{m}\bm{\mathrm{D}^+_\nu}\mathrm{U}_{\mathrm{e}}= 0 \ \ \ \ \ \ \ on \ \partial \Omega \times \mathbb{R}
\end{cases}
\end{align}
where $\mathrm{U}_\mathrm{e}(\mathrm{x},t)$ is assumed to be uniformly bounded in both variables,  \cite{hofman}.
\\

\noindent
We set $\Phi(\mathrm{x},t;\mathrm{y},\tau)$ equal to fundamental solution to the heat operator $\alpha\partial_\mathrm{t}-\Delta$ in three dimensional spatial variables as follows:
\begin{align}
        \Phi(\mathrm{x},t;\mathrm{y},\tau):=\  \begin{cases}
    \Big(\frac{\alpha}{4\pi(t-\tau)}\Big)^\frac{3}{2}\textbf{exp}\big(-\frac{\alpha |\mathrm{x}-\mathrm{y}|^2}{4(t-\tau)}\big), \ \ \ t > \tau \\
    0 ,\quad \text{otherwise}
    \end{cases}.
\end{align}
The fundamental solutions for the interior and exterior heat equation (\ref{heattran}) are $\Phi^\mathrm{i}(\mathrm{x},t;\mathrm{y},\tau)$ and $\Phi^\mathrm{e}(\mathrm{x},t;\mathrm{y},\tau)$ respectively, which depend on the variables $\alpha_{\mathrm{p}}:= \frac{\rho_{\mathrm{p}}\mathrm{c}_{\mathrm{p}}}{\gamma_{\mathrm{p}}} $ and $\alpha_\mathrm{m}:= \frac{\rho_{\mathrm{m}}\mathrm{c}_{\mathrm{m}}}{\gamma_{\mathrm{m}}}$.
\\

\noindent
\subsection{The related regimes}\label{the scales}
To describe correctly the scales needed in the mathematical analysis, we consider the nanoparticle to be of the form $\Omega = \delta \mathrm{B} + \mathrm{z}$, where $\delta$ defines the size of the nanoparticle, $\mathrm{B}$ is centered at origin and $\mathrm{z}$ specifies the position of the nanoparticle and $|\mathrm{B}|\sim 1.$ 
We also assume that the nanoparticle has the following scales regarding the heat-related coefficients
\begin{equation}\label{Assumption-heat-coefficients-section 3D}
\gamma_p = \overline{\gamma}_p\; \delta^{-2}, ~~ \rho_{\mathrm{p}}\mathrm{c}_{\mathrm{p}} \sim 1, ~~\mbox{ and }~~ \alpha_\mathrm{m} \sim 1, ~~ \delta \ll 1.
\end{equation}

\noindent
The next important step is to identify suitable Hilbert spaces, which in particular incorporate the Lippmann-Schwinger equation corresponding to (\ref{Maxwell model}) and allows us to do the needed analysis. For this, we introduce the following function spaces: 
\begin{align}
\begin{cases}
        \mathbb{H}(\textbf{div},\Omega):= \Big\{ u\in \big(\mathbb{L}^{2}(\Omega)\big)^3 :\; \textbf{div}\; u \in \mathbb{L}^{2}(\Omega)\Big\} \; \text{and}\\
   \mathbb{H}(\textbf{curl},\Omega):= \Big\{ u\in \big(\mathbb{L}^{2}(\Omega)\big)^3 :\; \textbf{curl}\; u \in \big(\mathbb{L}^{2}(\Omega)\big)^3\Big\}
\end{cases}
\end{align}
and recall the decomposition
\begin{align}
\big(\mathbb{L}^{2}(\Omega)\big)^3 = \mathbb{H}_{0}(\textbf{div},0) \oplus\mathbb{H}_{0}(\textbf{curl},0)\oplus \nabla \mathbb{H}_{\textbf{arm}},
\end{align}
where 
\begin{align}
    \begin{cases}
        \mathbb{H}_{0}(\textbf{div},0) = \Big\{ u\in \mathbb{H}(\textbf{div},\Omega): \textbf{div}\;u =0\;\text{in}\; \Omega \; \text{and}\;   u\cdot \nu = 0 \; \text{on}\; \partial\Omega\Big\}, \\
        \mathbb{H}_{0}(\textbf{curl},0) = \Big\{ u\in \mathbb{H}(\textbf{curl},\Omega): \textbf{curl}\;u =0\;\text{in}\; \Omega \; \text{and}\; u\times \nu = 0 \; \text{on}\; \partial\Omega\Big\}, \; \text{and} \\
        \nabla \mathbb{H}_{\textbf{arm}} = \Big\{ u \in \big(\mathbb{L}^{2}(\Omega)\big)^3: \exists \  \varphi \ \text{s.t.} \ u = \nabla \varphi,\; \varphi\in \mathbb{H}^1(\Omega)\; \text{and} \ \Delta \varphi = 0 \Big\}.
    \end{cases}
\end{align}

\noindent
 Let us also recall the fundamental solution of the Helmholtz propagator $\Delta + \mathrm{k}^2$ satisfying the outgoing Sommerfeld radiation condition at infinity $\mathbb{G}^{(\mathrm{k})}(\cdot,\mathrm{k})$, which is defined as
\begin{align}
    \mathbb{G}^{(\mathrm{k})}(\mathrm{x},\mathrm{y}, \mathrm{k}) := \frac{e^{i\mathrm{k}|\mathrm{x}-\mathrm{y}|}}{4\pi|\mathrm{x}-\mathrm{y}|},~~ x \neq y.
\end{align}
The Magnetization operator $\mathbb{M}^{(\mathrm{k})}$ from $\nabla \mathbb{H}_{\textbf{arm}}$ to $\nabla \mathbb{H}_{\textbf{arm}}$ and the Newtonian operator from $\mathbb{L}^2(\Omega)$ to $\mathbb{H}^2(\Omega)$, are therefore defined as follows
\begin{align}
    \mathbb{M}_{\Omega}^{(\mathrm{k})}\big[u](\mathrm{x}) := \nabla \int_{\Omega}  \nabla\mathbb{G}^{(\mathrm{k})}(\mathrm{x},\mathrm{y})\cdot u(\mathrm{y})d\mathrm{y} \quad \text{and} \quad \mathbb{N}_{\Omega}^{(\mathrm{k})}\big[ u\big](\mathrm{x}) := \int_{\Omega}  \mathbb{G}^{(\mathrm{k})}(\mathrm{x},\mathrm{y})u(\mathrm{y})d\mathrm{y}, \; \text{respectively}.
\end{align}
In particular, we indicate $\mathbb{M}_{\Omega}^{(\mathrm{0})}$ and $\mathbb{N}_{\Omega}^{(\mathrm{0})}$ as the respective operators when $\mathrm{k}=0.$ 
\\

\noindent
Furthermore, we recall the Lippmann-Schwinger equation satisfied by the solution of (\ref{Maxwell model}) 
\begin{align}\label{Lippmann}
    \mathrm{E}(\mathrm{x}) + \varsigma \nabla \int_{\Omega}  \nabla\mathbb{G}^{(\mathrm{k})}(\mathrm{x},\mathrm{y})\cdot \mathrm{E}(\mathrm{y})d\mathrm{y} - \omega^2\mu_\mathrm{m} \varsigma \int_{\Omega}  \mathbb{G}^{(\mathrm{k})}(\mathrm{x},\mathrm{y})\mathrm{E}(\mathrm{y})d\mathrm{y} = \mathrm{E}^\textbf{in}(\mathrm{x}), \; \mathrm{x}\in \Omega,
\end{align}
where $\varsigma := \varepsilon_\mathrm{p}(\omega)-\varepsilon_\mathrm{m}$ is the contrast parameter and $\mathrm{k}= \omega \sqrt{\mu_m \epsilon_m}$ is the wave number.
\\

\noindent
Assume now that the nanoparticle is of the form $\Omega = \delta \mathrm{B}+ \mathrm{z} \subseteq \mathbb{R}^3$ which is of class $\mathcal{C}^2$. Then (\ref{Lippmann}) becomes
\begin{align}\label{Lippmann-scaled}
    \tilde{\mathrm{E}}(\tilde{\mathrm{x}}) + \varsigma \nabla \int_{\mathrm{B}}  \nabla\mathbb{G}^{(\mathrm{k} \delta)}(\tilde{\mathrm{x}},\tilde{\mathrm{y}})\cdot \tilde{\mathrm{E}}(\tilde{\mathrm{y}})d\tilde{\mathrm{y}} - \omega^2\mu_\mathrm{m} \varsigma \delta^2 \int_{\mathrm{B}}  \mathbb{G}^{(\mathrm{k}\delta)}(\tilde{\mathrm{x}},\tilde{\mathrm{y}})\tilde{\mathrm{E}}(\tilde{\mathrm{y}})d\tilde{\mathrm{y}} = \tilde{\mathrm{E}}^\textbf{in}(\tilde{\mathrm{x}}), \; \tilde{\mathrm{x}}\in \mathrm{B},
\end{align}
where $\tilde{x}:=\frac{\mathrm{x}-\mathrm{z}}{\delta}$, $\tilde{\mathrm{E}}:= \mathrm{E}(\frac{\mathrm{x}-\mathrm{z}}{\delta})$ and $\tilde{\mathrm{E}}^{\text{in}}:= \mathrm{E}^\text{in}(\frac{\mathrm{x}-\mathrm{z}}{\delta})$. In short, we write (\ref{Lippmann-scaled}) as 
\begin{align}\label{Lippmann-scaled-operator}
    \tilde{\mathrm{E}} + \varsigma \mathbb{M}_{\mathrm{B}}^{(\mathrm{k} \delta)} \tilde{\mathrm{E}} - \omega^2\mu_\mathrm{m} \varsigma \delta^2 \mathbb{N}_{\mathrm{B}}^{(\mathrm{k}\delta)}\tilde{\mathrm{E}}  = \tilde{\mathrm{E}}^\textbf{in}.
\end{align}
We are interested in the quasi-static regimes where $\mathrm{k} \delta\ll 1$ as compared to the size of $B$. Recall that $\mathbb{N}_{B}^{(0)}$ and $\mathbb{M}_{\mathrm{B}}^{(0)}$ are positive on the spaces $\mathbb{H}_{0}(\textbf{div},0)$ and $\nabla \mathbb{H}_{\textbf{arm}}$ respectively. In addition, on their respective subspace, they generate sequences of eigen-elements that we denote by $(\lambda^{(1)}_n, e^{(1}_n)$ and  $(\lambda^{(3)}_\mathrm{n}, \mathrm{e}^{(3)}_\mathrm{n})$. \footnote{The operator $\mathbb{N}_{B}^{(0)}$ also generates a sequence of eigen-elements on $\mathbb{H}_{0}(\textbf{curl},0)$ that we denote $(\lambda^{(2)}_\mathrm{n}, \mathrm{e}^{(2)}_\mathrm{n})$.} 
\bigskip

\noindent
We observe that 
\begin{enumerate}
\item If $\Re(\varsigma)<0$, then we can excite the eigenvalues of the Magnetization operator $\mathbb{M}_{\mathrm{B}}^{(0)}$ while the ones of the Newtonian operator $\mathbb{N}_{\mathrm{B}}^{(0)}$ are avoided (due to the presence of $\varsigma \delta^2,~~ \varsigma \delta^2 \ll 1$).
\item If $\Re (\varsigma)>0$ and $\Re (\varsigma)\sim \delta^{-2}$, then we can excite the eigenvalues of the Newtonian operator while the eigenvalues of the Magnetization operator $\mathbb{M}_{\mathrm{B}}^{(0)}$ are avoided (due to positivity).
\end{enumerate}
In both cases, the electric field will be enhanced.  As the permittivity $\varepsilon_p(\omega)$ follows
the Lorentz-model stated in (\ref{lorenz-model}), below, we show that we can choose the incident frequency $\omega$ and the damping frequency $\zeta$ so that $\varsigma$ behaves as in one of the situations described above. In the first case, we say that the nanoparticle behaves as a plasmonic one while in the second, it behaves as a Dielectric one. 
\begin{enumerate}
    \item If we choose the incidence frequency $\omega$ and the damping frequency $\zeta$ such that 
    \begin{align}\label{plas-frequency-incidence}
    \omega^2 = \omega_0^2 + \dfrac{\omega_\mathrm{p}^2\lambda_{\mathrm{n}_0}^{(3)}\varepsilon_\infty}{\lambda_{\mathrm{n}_0}^{(3)}(\varepsilon_\mathrm{m}-\varepsilon_\infty)-1} + \mathcal{O}(\delta^\mathrm{h}) \mbox{ and } \zeta\omega \sim \delta^\mathrm{h},
    \end{align}
    then $\Re (\varsigma)<0$. In addition, we have the following properties
    \begin{align}\label{plas}
        \Im{(\varepsilon_\mathrm{p})} \sim \delta^\mathrm{h}\; \text{and}\; |1 + \varsigma\lambda^{(3)}_{\mathrm{n}_{0}}| \sim \delta^\mathrm{h},\; \text{where},\;  \lambda^{(3)}_{\mathrm{n}_{0}}\; \text{is the eigen-value corresponding to}\; \mathrm{e}^{(3)}_{\mathrm{n}_{0}}\; \text{and}\; \mathrm{h}>0.
    \end{align}
    
    \item If the frequency of the incidence wave $\omega$ is chosen close to the undamped resonance frequency $\omega_0$ and the damping frequency $\zeta$ such that 
    \begin{align}\label{die2-incidence-frequency}
    \omega_0^2-\omega^2 \sim \delta^2\big(\overline{\lambda}_\mathrm{n_0}^{(1)}\mu_\mathrm{m}\omega_0^2\big)\big[1+\mathcal{O}( \delta^{\mathrm{h}})\big] \mbox{ and } \zeta\omega \sim \delta^{2-\mathrm{h}}\big(\overline{\lambda}_\mathrm{n_0}^{(\boldsymbol{\ell})}\mu_\mathrm{m}\omega_0^2\big)^2,
    \end{align}
    then then $\Re (\varsigma)>0$ with $\Re(\varsigma)\sim \delta^{-2}\big(\overline{\lambda}_\mathrm{n_0}^{(1)}\mu_\mathrm{m}\omega_0^2\big)^{-1}$ and $\Im(\varsigma) \sim \delta^{\mathrm{h}-2}\big(\overline{\lambda}_\mathrm{n_0}^{(1)}\mu_\mathrm{m}\omega_0^2\big)^{-1}$, where $\mathrm{h}>0$. Consequently, we have 
    \begin{align}\label{die2}
            |1 - \omega^2\mu_\mathrm{m}\varsigma\delta^2\lambda^{(1)}_{\mathrm{n}_{0}}| \sim \delta^\mathrm{h},\; \text{where},\;  \lambda^{(1)}_{\mathrm{n}_{0}}\; \text{is the eigen-value corresponding to}\; \mathrm{e}^{(1)}_{\mathrm{n}_{0}}\; \text{and}\; \mathrm{h}>0.
        \end{align}
      Since $\mathbb{H}_{0}(\textbf{div},0) = \textbf{curl}\Big(\mathbb{H}_0(\textbf{curl})\cap \mathbb{H}(\textbf{div},0)\Big),\; \text{we have}\; \mathrm{e}_\mathrm{n_0}^{(1)} = \textbf{curl}(\varphi_{\mathrm{n}_0})\; \text{with}\; \nu \times \varphi_{\mathrm{n}_0} =0\\ \text{and}\; \textbf{div}(\varphi_{\mathrm{n}_0})=0.$

\end{enumerate}

\noindent
\subsection{The results}
Now, we state the first result of this work.
\begin{theorem}\label{th13D}
Let a nanoparticle occupy a domain $\Omega = \delta \mathrm{B}+ \mathrm{z} \subseteq \mathbb{R}^3$ which is of class $\mathcal{C}^2$. 
\begin{enumerate}
    \item Plasmonic Case. If we choose the incidence frequency $\omega$ and the undamped frequency $\zeta$ satisfying (\ref{plas-frequency-incidence}), and hence (\ref{plas}),  we have the following approximation of the electric field with $\mathrm{E}$ as the solution to (\ref{Maxwell model}), as $\delta\to 0$,
    \begin{align}\label{t13d}
        \int_{\Omega}|\mathrm{E}|^2(\mathrm{y})d\mathrm{y} = \frac{1}{ |1 + \varsigma\lambda^{(3)}_{\mathrm{n}_{0}}|^2}\delta^3\Big|\mathrm{E}^{\textbf{in}}(\mathrm{z})\cdot\int_\mathrm{\mathrm{B}}\Tilde{\mathrm{e}}^{(3)}_{\mathrm{n}_{0}}(\mathrm{x})d\mathrm{x}\Big|^2 + \begin{cases}
            \mathcal{O}\big(\delta^{4-2\mathrm{h}}\big)\quad \text{for} \quad \mathrm{h}\in(0,\frac{3}{2}). \\[10pt]
            \mathcal{O}\big(\delta^{7-4\mathrm{h}}\big)\quad \text{for} \quad \mathrm{h}\in(\frac{3}{2},2).
        \end{cases} 
    \end{align}

     \item Dielectric Case. If we choose the incidence frequency $\omega$ and the undamped frequency $\zeta$ satisfying (\ref{die2-incidence-frequency}), and hence (\ref{die2}),  we have the following approximation of the electric field with $\mathrm{E}$ as the solution to (\ref{Maxwell model}), as $\delta\to 0$,
    \begin{align}\label{t23d}
    \int_{\Omega}|\mathrm{E}|^2(\mathrm{y})d\mathrm{y} = \frac{\omega^2\mu^2_\mathrm{m}}{ |1 - \omega^2\mu_\mathrm{m}\varsigma\delta^2\lambda^{(1)}_{\mathrm{n}_{0}}|^2}\delta^5\Big|\mathrm{H}^{\textbf{in}}(\mathrm{z})\cdot\int_\mathrm{\mathrm{B}}\Tilde{\varphi}_{\mathrm{n}_{0}}(\mathrm{x})d\mathrm{x}\Big|^2 +
    \begin{cases}
      \mathcal{O}\big(\delta^5\big) \quad \quad \; \text{for} \quad \mathrm{h}\in(0,1).\\[10pt] 
      \mathcal{O}\big(\delta^{9-4\mathrm{h}}\big) \quad \text{for} \quad \mathrm{h}\in(1,2).
    \end{cases} 
\end{align}
\end{enumerate}
\end{theorem}

\noindent
We now state the main result of this work.
\begin{theorem}\label{mainth}
Let a nanoparticle, occupy a domain $\Omega = \mathrm{z} + \delta \mathrm{B}\subseteq\mathbb{R}^3$ which is of class $\mathcal{C}^2$, be such that its heat coefficients $(\rho_p, \mathrm{c}_\mathrm{p}, \gamma_\mathrm{p})$ satisfy the conditions (\ref{Assumption-heat-coefficients-section 3D}) and $\gamma_m <\sqrt{\overline{\gamma}_p\; \rho_{\mathrm{p}}\mathrm{c}_{\mathrm{p}} }, ~~ \delta \ll 1$. Let $\xi \in \mathbb{R}^3\setminus\overline{\Omega}$ such that $\textbf{dist}(\xi ,\Omega)\sim \delta^\mathrm{p}$ $\big(|\xi -z|\sim \delta^\mathrm{p}+\delta\big)$, where $\mathrm{p}\in [0,1)$. 
\begin{enumerate}
    
\item Under the assumption of Theorem \ref{th13D}(1), then for $\mathrm{r}<\frac{1}{2}$, if $2\mathrm{p}(1-\mathrm{r})<1 $, the heat generated by the \underline{plasmonic nanoparticle}, as a solution to (\ref{heattran}), is given by, as $\delta\to 0$,
\begin{align*} 
     \mathrm{U}_{\mathrm{e}}(\xi,\mathrm{t}) &= \frac{\omega \cdot \boldsymbol{\Im}(\varepsilon_\mathrm{p})}{8\pi^2\gamma_{\mathrm{m}}|\xi-\mathrm{z}|}\delta^{3-2\mathrm{h}}\Big|\mathrm{E}^{\textbf{in}}(\mathrm{z})\cdot\int_\mathrm{\mathrm{B}}\Tilde{\mathrm{e}}^{(3)}_{\mathrm{n}_{0}}(\mathrm{x})d\mathrm{x}\Big|^2 + 
     \begin{cases}
       \mathcal{O}(\delta^{4-\mathrm{h}-\mathrm{p}}) +  \mathcal{O}\Big(\delta^{4-\mathrm{h}-\mathrm{p}(3-2\mathrm{r})}\sqrt{\mathcal{K}^{(\mathrm{T_0})}_{\mathrm{r}}}\Big). \\[10pt] 
       \mathcal{O}(\delta^{7-3\mathrm{h}-\mathrm{p}}) +  \mathcal{O}\Big(\delta^{4-\mathrm{h}-\mathrm{p}(3-2\mathrm{r})}\sqrt{\mathcal{K}^{(\mathrm{T_0})}_{\mathrm{r}}}\Big).
     \end{cases}
\end{align*}

\item Under the assumption of Theorem \ref{th13D}(2) and $\boldsymbol{\Im}(\varepsilon_\mathrm{p}) \sim \delta^{\mathrm{h}-2},\; \delta\ll 1.$, then for $\mathrm{r}<\frac{1}{2}$, if $2\mathrm{p}(1-\mathrm{r})<\mathrm{h}$, the heat generated by the \underline{dielectric nanoparticle}, as a solution to (\ref{heattran}), is given by, as $\delta\to 0$,
\begin{align*}
       \mathrm{U}_{\mathrm{e}}(\xi,\mathrm{t}) &= \frac{\omega^3\mu^2_\mathrm{m} \cdot \boldsymbol{\Im}(\varepsilon_\mathrm{p})}{8\pi^2\gamma_{\mathrm{m}}|\xi-\mathrm{z}|}\delta^{5-2\mathrm{h}}\Big|\mathrm{H}^{\textbf{in}}(\mathrm{z})\cdot\int_\mathrm{\mathrm{B}}\Tilde{\varphi}_{\mathrm{n}_{0}}(\mathrm{x})d\mathrm{x}\Big|^2 + 
       \begin{cases}
         \mathcal{O}(\delta^{3+\mathrm{h}-\mathrm{p}})+ \mathcal{O}\Big(\delta^{3-\mathrm{p}(3-2\mathrm{r})}\sqrt{\mathcal{K}^{(\mathrm{T_0})}_{\mathrm{r}}}\Big). \\[10pt]
         \mathcal{O}(\delta^{7-3\mathrm{h}-\mathrm{p}})+ \mathcal{O}\Big(\delta^{3-\mathrm{p}(3-2\mathrm{r})}\sqrt{\mathcal{K}^{(\mathrm{T_0})}_{\mathrm{r}}}\Big).
       \end{cases}
\end{align*}
   where, $\displaystyle\mathcal{K}^{(\mathrm{T_0})}_{\mathrm{r}}:=\sup_{\mathrm{t}\in (0, \mathrm{T}_0)}\int^{\mathrm{T}_0}_0\frac{1}{(\mathrm{t}-\tau)^{2\mathrm{r}}}d\tau$ and it makes sense if $\mathrm{r}< \frac{1}{2}$.
\end{enumerate}

\end{theorem}
\noindent
We end this section with a few comments regarding the results presented in the previous theorems. 
\begin{enumerate}

\item The expressions $\displaystyle\Big|\mathrm{E}^{\textbf{in}}(\mathrm{z})\cdot\int_\mathrm{\mathrm{B}}\Tilde{\mathrm{e}}^{(3)}_{\mathrm{n}_{0}}(\mathrm{x})d\mathrm{x}\Big|^2$ and $\displaystyle\Big|\mathrm{H}^{\textbf{in}}(\mathrm{z})\cdot\int_\mathrm{\mathrm{B}}\Tilde{\mathrm{\varphi}}_{\mathrm{n}_{0}}(\mathrm{x})d\mathrm{x}\Big|^2$ should be understood as $\displaystyle\sum_{m}\Big|\mathrm{E}^{\textbf{in}}(\mathrm{z})\cdot\int_\mathrm{\mathrm{B}}\Tilde{\mathrm{e}}^{(3)}_{\mathrm{n}_{0}, m}(\mathrm{x})d\mathrm{x}\Big|^2$ and $\displaystyle\sum_{m}\Big|\mathrm{H}^{\textbf{in}}(\mathrm{z})\cdot\int_\mathrm{\mathrm{B}}\Tilde{\mathrm{\varphi}}^{(3)}_{\mathrm{n}_{0}, \mathrm{m}}(\mathrm{x})d\mathrm{x}\Big|^2$ where, for the fixed index $\mathrm{n}_0$, $\Tilde{\mathrm{e}}^{(3)}_{\mathrm{n}_{0}, \mathrm{m}}$ and $\Tilde{\mathrm{e}}^{(1)}_{\mathrm{m}}$, with $\Tilde{\mathrm{e}}^{(1)}_{\mathrm{n}_{0}, \mathrm{m}}:=\textbf{curl} (\varphi_{\mathrm{n}_0, \mathrm{m}})$, span the eigen-space corresponding to the eigenvalues $\lambda^{(3)}_{\mathrm{n}_0}$ and $\lambda^{(1)}_{\mathrm{n}_0}$ respectively. Observe that the approximate expansions provided in the two theorems above make sense only if the terms $\displaystyle\Big|\mathrm{E}^{\textbf{in}}(\mathrm{z})\cdot\int_\mathrm{\mathrm{B}}\Tilde{\mathrm{e}}^{(3)}_{\mathrm{n}_{0}}(\mathrm{x})d\mathrm{x}\Big|^2$ and $\displaystyle\Big|\mathrm{H}^{\textbf{in}}(\mathrm{z})\cdot\int_\mathrm{\mathrm{B}}\Tilde{\mathrm{\varphi}}_{\mathrm{n}_{0}}(\mathrm{x})d\mathrm{x}\Big|^2$ are not vanishing. In \cite{Ahcene-Mourad-JDE}, it is shown that for a sphere-shaped $\mathrm{B}$, we have $\displaystyle\Big|\mathrm{E}^{\textbf{in}}(\mathrm{z})\cdot\int_\mathrm{\mathrm{B}}\Tilde{\mathrm{e}}^{(3)}_{\mathrm{n}_{0}}(\mathrm{x})d\mathrm{x}\Big|^2=\mathrm{C}^\mathrm{t}\Big|\mathrm{E}^{\textbf{in}}(\mathrm{z})\Big|^2$ with a positive constant $\mathrm{C}^\mathrm{t}$. Therefore it is not vanishing.

\item The leading order terms in Theorem \ref{mainth} are given by $\frac{\omega \cdot \boldsymbol{\Im}(\varepsilon_\mathrm{p})}{8\pi^2\gamma_{\mathrm{m}}|\xi-\mathrm{z}|}\delta^{3-2\mathrm{h}}$ and $\frac{\omega^3\mu^2_\mathrm{m} \cdot \boldsymbol{\Im}(\varepsilon_\mathrm{p})}{8\pi^2\gamma_{\mathrm{m}}|\xi-\mathrm{z}|}\delta^{5-2\mathrm{h}}$ respectively. Therefore,  by selecting $\mathrm{h}$ close 2, the generated heat can be increased to any desired amount at a distance of the order $\delta$ from the nanoparticle while it decreases away from it. We further highlight that to adjust the dominant terms to any desired temperature, one requires knowledge of both the surrounding medium and optical properties of the nanoparticle.
Such features are useful for the purpose of therapy using heat.

\end{enumerate}

\noindent
The remaining parts of the work are structured as follows.  In Sec. \ref{plasmonicnanoth} and Sec. \ref{dielectricnanoth}, the proofs of Theorem \ref{th13D}(1) and Theorem \ref{th13D}(2), which deal with the asymptotic expansions of the Electric field, used to create plasmonic as well as dielectric resonances, are provided. In Sec \ref{heatmainth}, we give the proof of Theorem \ref{mainth}, which is the main finding of this work i.e. to provide an asymptotic expansion of the generated heat. In Sec. \ref{e23D}, we provide the justifications for some claimed a priori estimates. Finally, in Appendix \ref{appen}, we present a few technical estimates used in the prior sections.
\bigskip

\noindent
Unless specified, in this paper, we indicate $'\le'$ with its right-hand side multiplied by a general positive constant by the notation $'\lesssim'$.


\section{Proof of Theorem \ref{th13D}}
The proof is based on the Lippmann-Schwinger system of equations.  First, we note that the Lippmann-Schwinger equation stated in (\ref{Lippmann}) consists of Newtonian and Magnetization operators. Second, as we are using Lorentzian nanoparticles, plasmonic and dielectric resonant frequencies enable us to perform the approximations. We show that the field corresponding to the Magnetization operator is the dominant one in equation (\ref{t13d}) when we choose the incidence frequency close to plasmonic frequency and the field corresponding to the Newtonian operator is the dominant one in equation (\ref{t23d}) when we choose the incidence frequency close to dielectric resonance. In order to avoid confusion, we separated the proofs into Theorem \ref{th13D}(1) and Theorem \ref{th13D}(2).

\subsection{Proof of Theorem \ref{th13D}(1)}\label{plasmonicnanoth}
In this section, we describe the asymptotic analysis of the solution to (\ref{Maxwell model}) as $\delta\to 0$ when a \underline{plasmonic nanoparticle} occupy the domain $\Omega = \delta\mathrm{B}+\mathrm{z}.$
\\

\noindent
We begin by stating the Lippmann-Schwinger equation, given below as the solution to the electromagnetic scattering problem, (\ref{Maxwell model})
\begin{align}
    \mathrm{E}(\mathrm{x}) - (\varepsilon_\mathrm{p}-\varepsilon_\mathrm{m})\int_\Omega \Upsilon^{(\mathrm{k})}(\mathrm{x},\mathrm{y})\cdot \mathrm{E}(\mathrm{y}) d\mathrm{y} = \mathrm{E}^\textbf{in}(\mathrm{x}), \quad \quad \mathrm{x} \in \Omega,
\end{align}
where, $\Upsilon^{(\mathrm{k})}(\mathrm{x},\mathrm{y}): = \underset{\mathrm{x}}{\textbf{Hess}}\;\mathbb{G}^{(\mathrm{k})}(\mathrm{x},\mathrm{y}) + \omega^2\mu_\mathrm{m}\mathbb{G}^{(\mathrm{k})}(\mathrm{x},\mathrm{y})\mathrm{I}$ is the corresponding dyadic Green's function and $\mathbb{G}^{(\mathrm{k})}(\mathrm{x},\mathrm{y})$ is the Green's function for the Helmholtz Operator. Let us also denote $\varsigma:= \varepsilon_\mathrm{p}-\varepsilon_\mathrm{m}.$
\bigbreak 
\noindent
From the definition of dyadic Green's function, we rewrite integral equations representation 
\begin{align} \label{Lipmann}
    \mathrm{E}(\mathrm{x}) + \varsigma \; \mathbb{M}^{(\mathrm{k})}\big[\mathrm{E}\big](\mathrm{x}) - \omega^2\mu_\mathrm{m}\varsigma \; \mathbb{N}^{(\mathrm{k})}\big[\mathrm{E}\big](\mathrm{x}) = \mathrm{E}^\textbf{in}(\mathrm{x}),
\end{align}
where, we recall the magnetization operator and the Newtonian operator
\begin{equation}\label{magnetic operator3D}
   \mathbb{M}^{(\mathrm{k})}\big[\mathrm{E}\big](\mathrm{x}) = \nabla \int_{\Omega}  \nabla\mathbb{G}^{(\mathrm{k})}(\mathrm{x},\mathrm{y})\cdot\mathrm{E}(\mathrm{y})d\mathrm{y}
\\ \quad \text{and} \quad
   \mathbb{N}^{(\mathrm{k})}\big[\mathrm{E}\big](\mathrm{x}) = \int_{\Omega}  \mathbb{G}^{(\mathrm{k})}(\mathrm{x},\mathrm{y})\mathrm{E}(\mathrm{y})d\mathrm{y}.
\end{equation}
Then, the magnetization as well as Newtonian potentials can be decomposed as follows
\begin{align}
   \mathbb{M}^{(\mathrm{k})}\big[\mathrm{E}\big](\mathrm{x}) \nonumber&=  \mathbb{M}^{(\mathrm{0})}\big[\mathrm{E}\big](\mathrm{x}) + \frac{\omega^2\mu_\mathrm{m}}{2} \mathbb{N}^{(\mathrm{0})}\big[\mathrm{E}\big](\mathrm{x}) - \frac{i\omega^3\mu^2_\mathrm{m}}{12\pi}\int_\Omega \mathrm{E}(\mathrm{y})d\mathrm{y} + \frac{\omega^2\mu_\mathrm{m}}{2} \int_\Omega \mathbb{G}^{(\mathrm{0})}(\mathrm{x},\mathrm{y})\frac{\mathrm{A}(\mathrm{x},\mathrm{y})\cdot\mathrm{E}(\mathrm{y})}{\Vert \mathrm{x}-\mathrm{y}\Vert^2}d\mathrm{y}
   \\ &- \frac{1}{4\pi}\sum_{\mathrm{j}\ge 3} \frac{(i\omega\mu^\frac{1}{2}_\mathrm{m})^{\mathrm{j}+1}}{(\mathrm{j}+1)!} \underset{\mathrm{x}}{\textbf{Hess}}(\Vert \mathrm{x}-\mathrm{y}\Vert^\mathrm{j}), 
\end{align}
where $\mathrm{A}(\mathrm{x},\mathrm{y}):= (\mathrm{x}-\mathrm{y})\otimes (\mathrm{x}-\mathrm{y})$ and
\begin{align}\label{newton3d}
    \mathbb{N}^{(\mathrm{k})}\big[\mathrm{E}\big](\mathrm{x}) &=  \mathbb{N}^{(\mathrm{0})}\big[\mathrm{E}\big](\mathrm{x}) + \frac{i\omega\mu^\frac{1}{2}_\mathrm{m}}{4\pi}\int_\Omega \mathrm{E}(\mathrm{y})d\mathrm{y} + \frac{1}{4\pi}\sum_{\mathrm{j}\ge 1} \frac{(i\omega\mu^\frac{1}{2}_\mathrm{m})^{\mathrm{j}+1}}{(\mathrm{j}+1)!} \int_\Omega \Vert \mathrm{x}-\mathrm{y}\Vert^\mathrm{j} \mathrm{E}(\mathrm{y})d\mathrm{y}.
\end{align}
Let us recall the following decomposition of the space $\big(\mathbb{L}^{2}(\Omega)\big)^3$  into the following three sub-spaces as a direct sum as follows:
\begin{equation}\label{decomposition}
   \big(\mathbb{L}^{2}(\Omega)\big)^3 = \mathbb{H}_{0}(\textbf{div},0) \oplus\mathbb{H}_{0}(\textbf{curl},0)\oplus \nabla \mathbb{H}_{\textbf{arm}},
\end{equation}
where we define these three sub-spaces as follows:
\begin{align}{\label{eq:subspaces3D}}
    \begin{cases}
        \mathbb{H}_{0}(\textbf{div},0) = \Big\{ u\in \mathbb{H}(\textbf{div},\Omega): \textbf{div}\;u =0\;\text{in}\; \Omega \; \text{and}\;   u\cdot \nu = 0 \; \text{on}\; \partial\Omega\Big\}, \\
        \mathbb{H}_{0}(\textbf{curl},0) = \Big\{ u\in \mathbb{H}(\textbf{curl},\Omega): \textbf{curl}\;u =0\;\text{in}\; \Omega \; \text{and}\; u\times \nu = 0 \; \text{on}\; \partial\Omega\Big\}, \; \text{and} \\
        \nabla \mathbb{H}_{\textbf{arm}} = \Big\{ u \in \big(\mathbb{L}^{2}(\Omega)\big)^3: \exists \  \varphi \ \text{s.t.} \ u = \nabla \varphi,\; \varphi\in \mathbb{H}^1(\Omega)\; \text{and} \ \Delta \varphi = 0 \Big\}.
    \end{cases}
\end{align}
We also know that the Magnetization operator is self-adjoint and bounded, which satisfies the followings
\begin{align}
    \mathbb{M}^{(0)}\Big|_{\mathbb{H}_{0}(\textbf{div},0)} = 0, \quad \text{and} \quad \mathbb{M}^{(0)}\Big|_{\mathbb{H}_{0}(\textbf{curl},0)} = \mathrm{I}.
\end{align}
From the decomposition (\ref{decomposition}), we define $\overset{1}{\mathbb{P}}, \overset{2}{\mathbb{P}}$ and $\overset{3}{\mathbb{P}}$ to be the natural projectors as follows
\begin{align}
\overset{1}{\mathbb{P}}:= \mathbb{L}^2 \to \mathbb{H}_{0}(\textbf{div},0), \;   \overset{2}{\mathbb{P}}:= \mathbb{L}^2 \to \mathbb{H}_{0}(\textbf{curl},0), \; \text{and}\; \overset{3}{\mathbb{P}}:= \mathbb{L}^2 \to \nabla \mathbb{H}_{\textbf{arm}}. 
\end{align}
We also know that the  Magnetization operator $\mathbb{M}^{(0)}: \nabla \mathbb{H}_{\text{arm}}\rightarrow \nabla \mathbb{H}_{\text{arm}}$ induces a complete orthonormal basis namely $\big(\lambda^{(3)}_{\mathrm{n}},\mathrm{e}^{(3)}_{\mathrm{n}}\big)_{\mathrm{n} \in \mathbb{N}}$. Also $\mathbb{N}\Big|_{\mathbb{H}_{0}(\textbf{div},0)}$ and $\mathbb{N}\Big|_{\mathbb{H}_{0}(\textbf{curl},0)}$ generate complete orthonormal bases $\big(\lambda^{(1)}_{\mathrm{n}},\mathrm{e}^{(1)}_{\mathrm{n}}\big)_{\mathrm{n} \in \mathbb{N}}$ and $\big(\lambda^{(2)}_{\mathrm{n}},\mathrm{e}^{(2)}_{\mathrm{n}}\big)_{\mathrm{n} \in \mathbb{N}}$ of $\mathbb{H}_{0}(\textbf{div},0)$ and $\mathbb{H}_{0}(\textbf{curl},0)$ respectively. Due to the scale-invariance of the magnetization operator, we rewrite the integral representation given above in the scaled domain $\mathrm{B}$ to obtain
\begin{align}\label{scaleq}
    \Tilde{\mathrm{E}}(\mathrm{\xi}) + \varsigma \; \mathbb{M}^{(\mathrm{k}\delta)}_\mathrm{B}\big[\Tilde{\mathrm{E}}\big](\mathrm{\xi}) - \omega^2\mu_\mathrm{m}\varsigma \delta^2 \; \mathbb{N}_\mathrm{B}^{(\mathrm{k}\delta)}\big[\Tilde{\mathrm{E}}\big](\mathrm{\xi}) = \Tilde{\mathrm{E}}^\textbf{in}(\mathrm{\xi}).
\end{align}
The aforementioned equation will be considered in each of the sub-spaces indicated in (\ref{eq:subspaces3D}). We start with $\mathbb{H}_{0}(\textbf{div},0)$.
\begin{enumerate}
    \item We consider the inner-product with respect to $\mathrm{e}_\mathrm{n}^{(1)}$ to obtain
    \begin{align*}
        \big\langle \Tilde{\mathrm{E}}; \tilde{\mathrm{e}}_\mathrm{n}^{(1)}\big\rangle + \varsigma \big\langle \mathbb{M}^{(\mathrm{k}\delta)}_\mathrm{B}\big[\Tilde{\mathrm{E}}\big]; \tilde{\mathrm{e}}_\mathrm{n}^{(1)}\big\rangle = \big\langle \Tilde{\mathrm{E}}^\textbf{in}; \tilde{\mathrm{e}}_\mathrm{n}^{(1)}\big\rangle + \omega^2\mu_\mathrm{m}\varsigma \delta^2 \big\langle \mathbb{N}^{(\mathrm{k}\delta)}_\mathrm{B}\big[\Tilde{\mathrm{E}}\big]; \tilde{\mathrm{e}}_\mathrm{n}^{(1)}\big\rangle
    \end{align*}
    As $\mathbb{M}^{(\mathrm{k}\delta)}_\mathrm{B}$ has vanishing property in $\mathbb{H}_{0}(\textbf{div},0)$ we obtain
    \begin{align*}
        \big\langle \Tilde{\mathrm{E}}; \tilde{\mathrm{e}}_\mathrm{n}^{(1)}\big\rangle = \big\langle \Tilde{\mathrm{E}}^\textbf{in}; \tilde{\mathrm{e}}_\mathrm{n}^{(1)}\big\rangle + \omega^2\mu_\mathrm{m}\varsigma \delta^2 \big\langle \mathbb{N}^{(\mathrm{k}\delta)}_\mathrm{B}\big[\Tilde{\mathrm{E}}\big]; \tilde{\mathrm{e}}_\mathrm{n}^{(1)}\big\rangle.
    \end{align*}
    Consequently, we derive
     \begin{align}
   \big\Vert \overset{1}{\mathbb{P}}(\mathrm{\Tilde{E}})\big\Vert^2_{\mathbb{L}^2(\mathrm{B})}  = \sum_\mathrm{n} \big|\big\langle \Tilde{\mathrm{E}}^\textbf{in}; \tilde{\mathrm{e}}_\mathrm{n}^{(1)}\big\rangle\big|^2 + (\omega^2\mu_\mathrm{m})^2\varsigma^2 \delta^4\sum_\mathrm{n}\big|\big\langle \mathbb{N}^{(\mathrm{k}\delta)}_\mathrm{B}\big[\Tilde{\mathrm{E}}\big]; \tilde{\mathrm{e}}_\mathrm{n}^{(1)}\big\rangle\big|^2
 \end{align}
    
    \item Next, we consider the sub-space $\mathbb{H}_{0}(\textbf{curl},0)$ and we take the inner-product with respect to $\mathrm{e}_\mathrm{n}^{(2)}$ to obtain 
    \begin{align*}
      \big\langle \Tilde{\mathrm{E}}; \tilde{\mathrm{e}}_\mathrm{n}^{(2)}\big\rangle + \varsigma \big\langle \mathbb{M}^{(\mathrm{0})}_\mathrm{B}\big[\Tilde{\mathrm{E}}\big]; \tilde{\mathrm{e}}_\mathrm{n}^{(2)}\big\rangle = \big\langle \Tilde{\mathrm{E}}^\textbf{in}; \tilde{\mathrm{e}}_\mathrm{n}^{(2)}\big\rangle + \omega^2\mu_\mathrm{m}\varsigma \delta^2 \big\langle \mathbb{N}^{(\mathrm{k}\delta)}_\mathrm{B}\big[\Tilde{\mathrm{E}}\big]; \tilde{\mathrm{e}}_\mathrm{n}^{(2)}\big\rangle + \varsigma \Big\langle \big(\mathbb{M}^{(\mathrm{k}\delta)}_\mathrm{B}-\mathbb{M}^{(\mathrm{0})}_\mathrm{B}\big)\big[\Tilde{\mathrm{E}}\big]; \tilde{\mathrm{e}}_\mathrm{n}^{(2)}\Big\rangle 
    \end{align*}
    As $\mathbb{M}^{(\mathrm{0})}_\mathrm{B}\Big|_{\mathbb{H}_{0}(\textbf{curl},0)} = \mathrm{I}$, we derive
    \begin{align*}
        (1+\varsigma)\big\langle \Tilde{\mathrm{E}}; \tilde{\mathrm{e}}_\mathrm{n}^{(2)}\big\rangle = \big\langle \Tilde{\mathrm{E}}^\textbf{in}; \tilde{\mathrm{e}}_\mathrm{n}^{(2)}\big\rangle + \omega^2\mu_\mathrm{m}\varsigma \delta^2 \big\langle \mathbb{N}^{(\mathrm{k}\delta)}_\mathrm{B}\big[\Tilde{\mathrm{E}}\big]; \tilde{\mathrm{e}}_\mathrm{n}^{(2)}\big\rangle +  \varsigma \Big\langle \big(\mathbb{M}^{(\mathrm{k}\delta)}_\mathrm{B}-\mathbb{M}^{(\mathrm{0})}_\mathrm{B}\big)\big[\Tilde{\mathrm{E}}\big]; \tilde{\mathrm{e}}_\mathrm{n}^{(2)}\Big\rangle.
    \end{align*}
    Moreover, we express the above equation as follows
   \begin{align}
   \big\Vert \overset{2}{\mathbb{P}}(\mathrm{\Tilde{E}})\big\Vert^2_{\mathbb{L}^2(\mathrm{B})}  = \frac{1}{|1+\varsigma|^2} \sum_\mathrm{n} \big|\big\langle \Tilde{\mathrm{E}}^\textbf{in}; \tilde{\mathrm{e}}_\mathrm{n}^{(2)}\big\rangle\big|^2 + \frac{(\omega^2\mu_\mathrm{m})^2\varsigma^2 \delta^4}{|1+\varsigma|^2}\sum_\mathrm{n}\big|\big\langle \mathbb{N}^{(\mathrm{k}\delta)}_\mathrm{B}\big[\Tilde{\mathrm{E}}\big]; \tilde{\mathrm{e}}_\mathrm{n}^{(2)}\big\rangle\big|^2 + \sum_\mathrm{n} \frac{\varsigma^2}{|1+\varsigma|^2} \big|\textbf{err.}^{(2)}_\mathrm{n}\big|^2,
 \end{align}
 where we denote by $\textbf{err.}^{(2)}_\mathrm{n} := \Big\langle \big(\mathbb{M}^{(\mathrm{k}\delta)}_\mathrm{B}-\mathbb{M}^{(\mathrm{0})}_\mathrm{B}\big)\big[\Tilde{\mathrm{E}}\big]; \tilde{\mathrm{e}}_\mathrm{n}^{(2)}\Big\rangle.$
    \item As a last step, we consider the sub-space $\nabla \mathbb{H}_{\text{arm}}$ and we take the inner-product with respect to $\mathrm{e}_\mathrm{n}^{(3)}$. We then derive
    \begin{align*}
      \big\langle \Tilde{\mathrm{E}}; \tilde{\mathrm{e}}_\mathrm{n}^{(3)}\big\rangle + \varsigma \big\langle \mathbb{M}^{(\mathrm{0})}_\mathrm{B}\big[\Tilde{\mathrm{E}}\big]; \tilde{\mathrm{e}}_\mathrm{n}^{(3)}\big\rangle = \big\langle \Tilde{\mathrm{E}}^\textbf{in}; \tilde{\mathrm{e}}_\mathrm{n}^{(3)}\big\rangle + \omega^2\mu_\mathrm{m}\varsigma \delta^2 \big\langle \mathbb{N}^{(\mathrm{k}\delta)}_\mathrm{B}\big[\Tilde{\mathrm{E}}\big]; \tilde{\mathrm{e}}_\mathrm{n}^{(3)}\big\rangle  + \varsigma \big\langle \big(\mathbb{M}^{(\mathrm{k}\delta)}_\mathrm{B}-\mathbb{M}^{(\mathrm{0})}_\mathrm{B}\big)\big[\Tilde{\mathrm{E}}\big]; \tilde{\mathrm{e}}_\mathrm{n}^{(3)}\big\rangle
    \end{align*}
     Then with the self-adjointness of the magnetic operator, we deduce
     \begin{align*}
      (1+\varsigma\lambda_\mathrm{n}^{(3)}\big\langle \Tilde{\mathrm{E}}; \tilde{\mathrm{e}}_\mathrm{n}^{(3)}\big\rangle = \big\langle \Tilde{\mathrm{E}}^\textbf{in}; \tilde{\mathrm{e}}_\mathrm{n}^{(3)}\big\rangle + \omega^2\mu_\mathrm{m}\varsigma \delta^2 \big\langle \mathbb{N}^{(\mathrm{k}\delta)}_\mathrm{B}\big[\Tilde{\mathrm{E}}\big]; \tilde{\mathrm{e}}_\mathrm{n}^{(3)}\big\rangle  + \varsigma \big\langle \big(\mathbb{M}^{(\mathrm{k}\delta)}_\mathrm{B}-\mathbb{M}^{(\mathrm{0})}_\mathrm{B}\big)\big[\Tilde{\mathrm{E}}\big]; \tilde{\mathrm{e}}_\mathrm{n}^{(3)}\big\rangle.
    \end{align*}
     Consequently,
    \begin{align}
   \big\Vert \overset{3}{\mathbb{P}}(\mathrm{\Tilde{E}})\big\Vert^2_{\mathbb{L}^2(\mathrm{B})}  =  \sum_\mathrm{n}\frac{1}{|1+\varsigma\lambda_\mathrm{n}^{(3)}|^2} \big|\big\langle \Tilde{\mathrm{E}}^\textbf{in}; \tilde{\mathrm{e}}_\mathrm{n}^{(3)}\big\rangle\big|^2 + \sum_\mathrm{n}\frac{(\omega^2\mu_\mathrm{m})^2\varsigma^2 \delta^4}{|1+\varsigma \lambda_\mathrm{n}^{(3)}|^2}\big|\big\langle \mathbb{N}^{(\mathrm{k}\delta)}_\mathrm{B}\big[\Tilde{\mathrm{E}}\big]; \tilde{\mathrm{e}}_\mathrm{n}^{(3)}\big\rangle\big|^2 + \sum_\mathrm{n} \frac{\varsigma^2}{|1+\varsigma\lambda_\mathrm{n}^{(3)}|^2} \big|\textbf{err.}^{(3)}_\mathrm{n}\big|^2,
 \end{align}
 where we denote by $\textbf{err.}^{(3)}_\mathrm{n} := \Big\langle \big(\mathbb{M}^{(\mathrm{k}\delta)}_\mathrm{B}-\mathbb{M}^{(\mathrm{0})}_\mathrm{B}\big)\big[\Tilde{\mathrm{E}}\big]; \tilde{\mathrm{e}}_\mathrm{n}^{(3)}\Big\rangle.$
\end{enumerate}
Now, we use Parseval's identity to estimate $\mathrm{E}$, i.e. we write
\begin{align}\label{parseval3D}
    \big\Vert \mathrm{\Tilde{E}}\big\Vert^2_{\mathbb{L}^2(\mathrm{B})} \nonumber &= \sum_{\mathrm{j}=1}^3 \big\Vert \overset{\mathrm{j}}{\mathbb{P}}(\mathrm{\Tilde{E}})\big\Vert^2_{\mathbb{L}^2(\mathrm{B})}
    \\ \nonumber &\lesssim \sum_\mathrm{n} \big|\big\langle \Tilde{\mathrm{E}}^\textbf{in}; \tilde{\mathrm{e}}_\mathrm{n}^{(1)}\big\rangle\big|^2 + \frac{1}{|1+\varsigma|^2} \sum_\mathrm{n} \big|\big\langle \Tilde{\mathrm{E}}^\textbf{in}; \tilde{\mathrm{e}}_\mathrm{n}^{(2)}\big\rangle\big|^2 + \sum_\mathrm{n}\frac{1}{|1+\varsigma\lambda_\mathrm{n}^{(3)}|^2} \big|\big\langle \Tilde{\mathrm{E}}^\textbf{in}; \tilde{\mathrm{e}}_\mathrm{n}^{(3)}\big\rangle\big|^2
    \\ & + \sum_\mathrm{n} \frac{\varsigma^2}{|1+\varsigma\lambda_\mathrm{n}^{(3)}|^2} \big|\textbf{err.}^{(3)}_\mathrm{n}\big|^2 + \sum_\mathrm{n} \frac{\varsigma^2}{|1+\varsigma|^2} \big|\textbf{err.}^{(2)}_\mathrm{n}\big|^2
    + \frac{(\omega^2\mu_\mathrm{m})^2\varsigma^2 \delta^4}{|1+\varsigma\lambda_\mathrm{n_0}^{(3)}|^2}\big\Vert \mathbb{N}^{(\mathrm{k}\delta)}_\mathrm{B}\big[\mathrm{\Tilde{E}}\big]\big\Vert^2_{\mathbb{L}^2(\mathrm{B})}
\end{align}
Furthermore, from the choice of the incident frequency, based on the Lorentz model, we have the following properties (\ref{plas}), with $\mathrm{h}>0$,
\begin{align}\label{condition3D}
\big|1 + \varsigma\lambda^{(3)}_{\mathrm{n}}\big| \sim 
\begin{cases}
\delta^\mathrm{h} & \mathrm{n} = \mathrm{n}_0 \\
1 & \mathrm{n} \ne \mathrm{n}_0.
\end{cases}   
\end{align}
Next, we estimate the following term
\begin{align}
    \textbf{err.}^{(3)}_\mathrm{n} \nonumber&:= \Big\langle \big(\mathbb{M}^{(\mathrm{k}\delta)}_\mathrm{B}-\mathbb{M}^{(\mathrm{0})}_\mathrm{B}\big)\big[\Tilde{\mathrm{E}}\big]; \tilde{\mathrm{e}}_\mathrm{n}^{(3)}\Big\rangle
    \\ \nonumber &= \frac{\omega^2\mu_\mathrm{m}\delta^2}{2} \big\langle\mathbb{N}_\mathrm{B}^{(\mathrm{0})}\big[\Tilde{\mathrm{E}}\big];\Tilde{\mathrm{e}}^{(3)}_{\mathrm{n}}\big\rangle - \frac{i\omega^3\mu^2_\mathrm{m}\delta^3}{12\pi}\big\langle\int_\mathrm{B} \Tilde{\mathrm{E}}(\mathrm{\eta})d\mathrm{\eta};\Tilde{\mathrm{e}}^{(3)}_{\mathrm{n}}\big\rangle \\ &+ \frac{\omega^2\mu_\mathrm{m}\delta^2}{2} \big\langle\int_\mathrm{B} \mathbb{G}^{(\mathrm{0})}(\mathrm{\xi},\mathrm{\eta})\frac{\mathrm{A}(\mathrm{\xi},\mathrm{\eta})\cdot\Tilde{\mathrm{E}}(\mathrm{\eta})}{\Vert \mathrm{\xi}-\mathrm{\eta}\Vert^2}d\mathrm{\eta};\Tilde{\mathrm{e}}^{(3)}_{\mathrm{n}}\big\rangle- \frac{1}{4\pi}\sum_{\mathrm{j}\ge 3} (i\omega\mu^\frac{1}{2}_\mathrm{m}\delta)^{\mathrm{j}+1}\Big\langle \int_\mathrm{B}\frac{\underset{\mathrm{\xi}}{\textbf{Hess}}(\Vert \mathrm{\xi}-\mathrm{\eta}\Vert^\mathrm{j})}{(\mathrm{j}+1)!}\cdot\Tilde{\mathrm{E}}(\mathrm{\eta})d\mathrm{\eta};\Tilde{\mathrm{e}}^{(3)}_{\mathrm{n}}\Big\rangle
\end{align}
Using the continuity of the Newtonian operator, squaring the preceding expression, taking the series with respect to n on both sides, we obtain
\begin{align}\label{Err3}
    \sum_\mathrm{n} \frac{\varsigma^2}{|1+\varsigma\lambda_\mathrm{n}^{(3)}|^2} \big|\textbf{err.}^{(3)}_\mathrm{n}\big|^2 \nonumber&\lesssim \delta^{-2\mathrm{h}}\Big[\frac{(\omega\mu_\mathrm{m}^\frac{1}{2}\delta)^4}{4}\Vert\Tilde{\mathrm{E}}\Vert^2_{\mathbb{L}^2(\mathrm{B})} + \frac{(\omega\mu_\mathrm{m}^\frac{1}{2}\delta)^6}{(12\pi)^2}\Vert\Tilde{\mathrm{E}}\Vert^2_{\mathbb{L}^2(\mathrm{B})} + \frac{(\omega\mu_\mathrm{m}^\frac{1}{2}\delta)^4}{4}\Vert\Tilde{\mathrm{E}}\Vert^2_{\mathbb{L}^2(\mathrm{B})} \\ &+ \frac{(\omega\mu_\mathrm{m}^\frac{1}{2}\delta)^8}{(4\pi)^2}\Vert\Tilde{\mathrm{E}}\Vert^2_{\mathbb{L}^2(\mathrm{B})} \underbrace{\sum_{\mathrm{j}\ge 3}\int_\mathrm{B}\int_\mathrm{B}\Bigg(\frac{\underset{\mathrm{x}}{\textbf{Hess}}(\Vert \mathrm{x}-\mathrm{y}\Vert^\mathrm{j})}{(\mathrm{j}+1)!}\Bigg)^2d\mathrm{y}d\mathrm{x}}_{<+\infty}\Big]
\end{align}
A similar analysis follows for the term $\displaystyle\sum_\mathrm{n} \frac{\varsigma^2}{|1+\varsigma|^2} \big|\textbf{err.}^{(2)}_\mathrm{n}\big|^2$. Consequently, after simplification, we reach the following conclusions from (\ref{parseval3D})
\begin{align}
    \big\Vert \mathrm{\Tilde{E}}\big\Vert^2_{\mathbb{L}^2(\mathrm{B})} \nonumber&= \sum_{\mathrm{j}=1}^3 \big\Vert \overset{\mathrm{j}}{\mathbb{P}}(\mathrm{\Tilde{E}})\big\Vert^2_{\mathbb{L}^2(\mathrm{B})}
    \\ \nonumber&\lesssim \sum_\mathrm{n} \frac{1}{|1+\varsigma\mathrm{\lambda}^{(3)}_{\mathrm{n}}|^2}\big\Vert \mathrm{\Tilde{E}^\textbf{in}}\big\Vert^2_{\mathbb{L}^2(\mathrm{B})}+ \sum_\mathrm{n}\frac{(\omega^2\mu_\mathrm{m})^2\varsigma^2\delta^4}{|1+\varsigma\mathrm{\lambda}^{(3)}_{\mathrm{n}}|^2} \; \big\Vert \mathbb{N}_\mathrm{B}^{(\mathrm{k}\delta)}\big[\mathrm{\Tilde{E}\big]}\big\Vert^2_{\mathbb{L}^2(\mathrm{B})} + \frac{(\omega^2\mu_\mathrm{m})^2\varsigma^2\delta^4}{|1+\varsigma\mathrm{\lambda}^{(3)}_{\mathrm{n}_0}|^2}\Vert\Tilde{\mathrm{E}}\Vert^2_{\mathbb{L}^2(\mathrm{B})}   
\end{align}
and then
\begin{align}\label{18}
    \Big(1- \frac{(\omega^2\mu_\mathrm{m})^2\varsigma^2\delta^4}{|1+\varsigma\mathrm{\lambda}^{(3)}_{\mathrm{n}_0}|^2}\Big) \big\Vert \mathrm{\Tilde{E}}\big\Vert^2_{\mathbb{L}^2(\mathrm{B})} \lesssim \frac{1}{|1+\varsigma\mathrm{\lambda}^{(3)}_{\mathrm{n}_0}|^2}\big\Vert \mathrm{\Tilde{E}^\textbf{in}}\big\Vert^2_{\mathbb{L}^2(\mathrm{B})} + \sum_\mathrm{n\ne\mathrm{n_0}} \frac{1}{|1+\varsigma\mathrm{\lambda}^{(3)}_{\mathrm{n}}|^2}\big\Vert \mathrm{\Tilde{E}^\textbf{in}}\big\Vert^2_{\mathbb{L}^2(\mathrm{B})}.
\end{align}
Thus, we deduce an a priori estimate using the identity (\ref{condition3D}) 
\begin{align}\label{aprori3D}
    \big\Vert \mathrm{{E}}\big\Vert_{\mathbb{L}^2(\mathrm{\Omega})} \sim \delta^{\frac{3}{2}-\mathrm{h}}\quad \text{for}\; \mathrm{h}<2.
\end{align}
Next, we use the above derived a priori estimate to clarify the exact dominant term of the formulation (\ref{parseval3D}).\\
\noindent
We know that the following mean vanishing integral properties are satisfied by the eigenfunctions $\mathrm{e}_\mathrm{n}^{(\mathrm{j})}$ for $\mathrm{j}=1,2$:
\begin{align}\label{meanvanish}
    \int_\mathrm{B} \mathrm{e}_\mathrm{n}^{(\mathrm{j})}(\mathrm{x}) d\mathrm{x} = 0. 
\end{align}
Thereafter, we do the following estimate using Taylor's expansion
\begin{align}\label{1.20}
     \big\Vert \overset{\mathrm{j}}{\mathbb{P}}(\mathrm{\Tilde{E}}^\textbf{in})\big\Vert^2_{\mathbb{L}^2(\mathrm{B})} \nonumber &= \sum_\mathrm{n} \big|\big\langle \Tilde{\mathrm{E}}^\textbf{in}; \tilde{\mathrm{e}}_\mathrm{n}^{(\mathrm{j})}\big\rangle\big|^2 
     \\ \nonumber &= \sum_\mathrm{n} \Big| \int_\mathrm{B} \Tilde{\mathrm{E}}^\textbf{in}(\mathrm{x})\cdot \tilde{\mathrm{e}}_\mathrm{n}^{(\mathrm{j})}(\mathrm{x})d\mathrm{x}\Big|^2
     \\ &= \sum_\mathrm{n} \Big| \Tilde{\mathrm{E}}^\textbf{in}(\mathrm{z}) \cdot \underbrace{\int_\mathrm{B} \mathrm{e}_\mathrm{n}^{(\mathrm{j})}(\mathrm{x}) d\mathrm{x}}_{=\; 0,\; (\ref{meanvanish})}\Big|^2 + \mathcal{O}(\delta^2).
\end{align}
Therefore, we have shown that $\big\Vert \overset{\mathrm{j}}{\mathbb{P}}(\mathrm{\Tilde{E}}^\textbf{in})\big\Vert_{\mathbb{L}^2(\mathrm{B})} \sim \delta$ for $\mathrm{j}=1,2.$ 
Then, based on the a priori estimate (\ref{aprori3D}), we deduce that
\begin{align}\label{1.22}
    \sum_\mathrm{n} \frac{\varsigma^2}{|1+\varsigma\lambda_\mathrm{n}^{(3)}|^2} \big|\textbf{err.}^{(3)}_\mathrm{n}\big|^2 \sim \delta^{4-4\mathrm{h}}.
\end{align}
In a similar way, we can show that
\begin{align}\label{1.23}
    \sum_\mathrm{n} \frac{\varsigma^2}{|1+\varsigma|^2} \big|\textbf{err.}^{(2)}_\mathrm{n}\big|^2 \sim \delta^{4-4\mathrm{h}}.
\end{align}
Moreover, we rewrite the expression (\ref{parseval3D}) using (\ref{1.20}), (\ref{1.22}), (\ref{1.23}) as follows:
\begin{align}\label{red}
    \int_{\mathrm{B}}|\Tilde{\mathrm{E}}|^2(\mathrm{\eta})d\mathrm{\eta} &= \frac{1}{|1+\varsigma\lambda_\mathrm{n_0}^{(3)}|^2} \big|\big\langle \Tilde{\mathrm{E}}^\textbf{in}; \tilde{\mathrm{e}}_\mathrm{n_0}^{(3)}\big\rangle\big|^2 + \sum_{\mathrm{n}\ne \mathrm{n_0}}\frac{1}{|1+\varsigma\lambda_\mathrm{n}^{(3)}|^2} \big|\big\langle \Tilde{\mathrm{E}}^\textbf{in}; \tilde{\mathrm{e}}_\mathrm{n}^{(3)}\big\rangle\big|^2 + \sum_{\mathrm{n}}\frac{1}{|1+\varsigma|^2} \big|\big\langle \Tilde{\mathrm{E}}^\textbf{in}; \tilde{\mathrm{e}}_\mathrm{n}^{(2)}\big\rangle\big|^2 + \mathcal{O}(\delta^{4-4\mathrm{h}}).
\end{align}
Now, as $\displaystyle\sum_{\mathrm{n}\ne \mathrm{n_0}}\frac{1}{|1+\varsigma\lambda_\mathrm{n}^{(3)}|^2} \big|\big\langle \Tilde{\mathrm{E}}^\textbf{in}; \tilde{\mathrm{e}}_\mathrm{n}^{(3)}\big\rangle\big|^2 \sim 1$ and $\displaystyle\big\langle \Tilde{\mathrm{E}}^\textbf{in}; \tilde{\mathrm{e}}_\mathrm{n_0}^{(3)}\big\rangle_{\mathbb{L}^2(\mathrm{B})} = \mathrm{u}^{\textbf{in}}(\mathrm{z})\cdot\int_\mathrm{\mathrm{B}}\Tilde{\mathrm{e}}^{(3)}_{\mathrm{n}_{0}}(\mathrm{x})d\mathrm{x} + \mathcal{O}(\delta)$, we deduce that
\begin{align*}
    \int_{\Omega}|\mathrm{E}|^2(\mathrm{y})d\mathrm{y} = \frac{1}{ |1 + \varsigma\lambda^{(3)}_{\mathrm{n}_{0}}|^2}\delta^3\Big|\mathrm{E}^{\textbf{in}}(\mathrm{z})\cdot\int_\mathrm{\mathrm{B}}\Tilde{\mathrm{e}}^{(3)}_{\mathrm{n}_{0}}(\mathrm{x})d\mathrm{x}\Big|^2 + \begin{cases}
        \mathcal{O}\big(\delta^{4-2\mathrm{h}}\big)\quad \text{for} \quad \mathrm{h}\in(0,\frac{3}{2}). \\[10pt]
        \mathcal{O}\big(\delta^{7-4\mathrm{h}}\big)\quad \text{for} \quad \mathrm{h}\in(\frac{3}{2},2).
        \end{cases} 
\end{align*}
The proof of Theorem \ref{th13D}(1) is completed.

\subsection{Proof of Theorem \ref{th13D}(2)}\label{dielectricnanoth}

In a similar way as for to the plasmonic case, we show the asymptotic analysis of the solution to (\ref{Maxwell model}) as $\delta\to 0$ when a \underline{dielectric nanoparticle} occupy the domain $\Omega = \delta\mathrm{B}+\mathrm{z}.$
\\

\noindent
We begin by recalling that for dielectric nanoparticle the contrast parameter $\varsigma := \varepsilon_\mathrm{p}-\varepsilon_\mathrm{m}$ behaves as $\delta^{-2},\; \delta\ll 1.$ Then, we start from Lippmann-Schwinger system of equation in the scaled domain $\mathrm{B}$
    \begin{align}\label{scaeq}
    \Tilde{\mathrm{E}}(\mathrm{\xi}) + \varsigma \; \mathbb{M}^{(\mathrm{k}\delta)}_\mathrm{B}\big[\Tilde{\mathrm{E}}\big](\mathrm{\xi}) - \omega^2\mu_\mathrm{m}\varsigma \delta^2 \; \mathbb{N}_\mathrm{B}^{(\mathrm{k}\delta)}\big[\Tilde{\mathrm{E}}\big](\mathrm{\xi}) = \Tilde{\mathrm{E}}^\textbf{in}(\mathrm{\xi}).
\end{align}
Similarly to the plasmonic situation, we project the scaled equation (\ref{scaeq}) with respect to the eigen-functions $\mathrm{e}^{(\mathrm{j})}_\mathrm{n}$ for $\mathrm{j}=1,2,3$, in each of the sub-spaces mentioned in (\ref{eq:subspaces3D}).

\begin{enumerate}
    \item As, $\mathbb{M}^{(\mathrm{k})}_\mathrm{B}$ is vanishing in $\mathbb{H}_0(\textbf{div},0)$ and $\mathbb{N}^{(\mathrm{0})}_\mathrm{B}$ induces an eigen-system $\big(\Tilde{\lambda}^{(1)}_{\mathrm{n}},\Tilde{\mathrm{e}}^{(1)}_{\mathrm{n}}\big)_{\mathrm{n} \in \mathbb{N}}$, we rewrite (\ref{scaeq}) after taking an inner product with respect to $\mathrm{e}^{(\mathrm{1})}_\mathrm{n}$ as follows:
    
    \begin{align}\label{newton13d}
        (1-\omega^2\mu_\mathrm{m}\varsigma \delta^2\Tilde{\lambda}^{(1)}_{\mathrm{n}})\langle \Tilde{\mathrm{E}}; \tilde{\mathrm{e}}_\mathrm{n}^{(1)}\big\rangle = \big\langle \Tilde{\mathrm{E}}^\textbf{in}; \tilde{\mathrm{e}}_\mathrm{n}^{(1)}\big\rangle + \omega^2\mu_\mathrm{m}\varsigma \delta^2 \big\langle \mathbb{N}^{(\mathrm{k}\delta)}_\mathrm{B}-\mathbb{N}^{(\mathrm{0})}_\mathrm{B}\big[\Tilde{\mathrm{E}}\big]; \tilde{\mathrm{e}}_\mathrm{n}^{(1)}\big\rangle.
    \end{align}
    Then, using the expression (\ref{newton3d}) in (\ref{newton13d}), we deduce the following after taking a modulus 
    \begin{align}
               |\langle \Tilde{\mathrm{E}}; \tilde{\mathrm{e}}_\mathrm{n}^{(1)}\big\rangle| = \frac{1}{\big|1-\omega^2\mu_\mathrm{m}\varsigma \delta^2\Tilde{\lambda}^{(1)}_{\mathrm{n}}\big|}\Bigg[\big|\big\langle \Tilde{\mathrm{E}}^\textbf{in}; \tilde{\mathrm{e}}_\mathrm{n}^{(1)}\big\rangle\big| &\nonumber+ \omega^2\mu_\mathrm{m}\varsigma \delta^2 \frac{i\omega\mu^\frac{1}{2}_\mathrm{m}\delta}{4\pi}\big|\Big\langle\int_\mathrm{B} \Tilde{\mathrm{E}}(\eta)d\eta; \tilde{\mathrm{e}}_\mathrm{n}^{(1)}\Big\rangle\big| \\ &+ \frac{\omega^2\mu_\mathrm{m}\varsigma \delta^2}{4\pi}\sum_{\mathrm{j}\ge 1} \frac{(i\omega\mu^\frac{1}{2}_\mathrm{m}\delta)^{\mathrm{j}+1}}{(\mathrm{j}+1)!} \big|\Big\langle\int_\mathrm{B} \Vert \cdot-\eta\Vert^\mathrm{j} \Tilde{\mathrm{E}}(\eta)d\eta; \tilde{\mathrm{e}}_\mathrm{n}^{(1)}\Big\rangle\big|\Bigg]. 
    \end{align}
Moreover, we know that $\nabla \mathbb{H}_{\textbf{arm}}$ is orthogonal to each of the subspace $\mathbb{H}_{0}(\textbf{div},0)$ and $\mathbb{H}_{0}(\textbf{curl},0)$. Additionally, we have $\displaystyle\int_\mathrm{B}\tilde{\mathrm{e}}_\mathrm{n}^{(1)}(\mathrm{x})d\mathrm{x} = 0$ as the identity matrix $\mathbb{I} \in \nabla \mathbb{H}_{\textbf{arm}}$. Consequently, we obtain
    \begin{align}\label{dieapp1}
        \big\Vert \overset{1}{\mathbb{P}}(\mathrm{\Tilde{E}})\big\Vert^2_{\mathbb{L}^2(\mathrm{B})} \nonumber&= \sum_\mathrm{n} \big|\big\langle \Tilde{\mathrm{E}}; \tilde{\mathrm{e}}_\mathrm{n}^{(1)}\big\rangle\big|^2
        \\ \nonumber&\lesssim \sum_\mathrm{n} 
               \frac{1}{\big|1-\omega^2\mu_\mathrm{m}\varsigma \delta^2\Tilde{\lambda}^{(1)}_{\mathrm{n}}\big|^2} \big|\big\langle \Tilde{\mathrm{E}}^\textbf{in}; \tilde{\mathrm{e}}_\mathrm{n}^{(1)}\big\rangle\big|^2 + \delta^{4-2\mathrm{h}} \sum_\mathrm{n} \Big|\sum_{\mathrm{j}\ge 1} \frac{1}{4\pi(\mathrm{j}+1)!} \Big\langle\int_\mathrm{B} \Vert \cdot-\eta\Vert^\mathrm{j} \Tilde{\mathrm{E}}(\eta)d\eta; \tilde{\mathrm{e}}_\mathrm{n}^{(1)}\Big\rangle\Big|^2
               \\ &\lesssim \sum_\mathrm{n} 
               \frac{1}{\big|1-\omega^2\mu_\mathrm{m}\varsigma \delta^2\Tilde{\lambda}^{(1)}_{\mathrm{n}}\big|^2} \big|\big\langle \Tilde{\mathrm{E}}^\textbf{in}; \tilde{\mathrm{e}}_\mathrm{n}^{(1)}\big\rangle\big|^2 + \delta^{4-2\mathrm{h}} \big\Vert \mathrm{\Tilde{E}}\big\Vert^2_{\mathbb{L}^2(\mathrm{B})}.
    \end{align}
    
    \item Next, we consider the equation (\ref{scaeq}) in $\mathbb{H}_{0}(\textbf{curl},0)$ and we take inner product with respect to $\Tilde{\mathrm{e}}_\mathrm{n}^{(2)}$ in the scaled domain $\mathrm{B}$ to obtain 
    \begin{align}
    \big\langle \Tilde{\mathrm{E}}; \tilde{\mathrm{e}}_\mathrm{n}^{(2)}\big\rangle + \varsigma \big\langle \mathbb{M}^{(\mathrm{k}\delta)}_\mathrm{B}\big[\Tilde{\mathrm{E}}\big]; \tilde{\mathrm{e}}_\mathrm{n}^{(2)}\big\rangle - \omega^2\mu_\mathrm{m}\varsigma \delta^2 \big\langle \mathbb{N}^{(\mathrm{k}\delta)}_\mathrm{B}\big[\Tilde{\mathrm{E}}\big]; \tilde{\mathrm{e}}_\mathrm{n}^{(2)}\big\rangle= \big\langle \Tilde{\mathrm{E}}^\textbf{in}; \tilde{\mathrm{e}}_\mathrm{n}^{(2)}\big\rangle.
    \end{align}
    The adjoint operators of the Magnetization and Newtonian potentials are then taken into account, and they are $\mathbb{M}^{(-\mathrm{k}  \delta)}_\mathrm{B}$ and $\mathbb{N}^{(-\mathrm{k} \delta)}_\mathrm{B}$ respectively. We then pass the adjoint operator with respect to $\tilde{\mathrm{e}}_\mathrm{n}^{(2)}$ to derive
    \begin{align}\label{project-en-2}
        \big\langle \Tilde{\mathrm{E}}; \tilde{\mathrm{e}}_\mathrm{n}^{(2)}\big\rangle + \varsigma \big\langle \Tilde{\mathrm{E}}; \mathbb{M}^{(-\mathrm{k}  \delta)}_\mathrm{B}\big(\tilde{\mathrm{e}}_\mathrm{n}^{(2)}\big)\big\rangle - \omega^2\mu_\mathrm{m}\varsigma \delta^2 \big\langle \Tilde{\mathrm{E}}; \mathbb{N}^{(-\mathrm{k}  \delta)}_\mathrm{B}\big(\tilde{\mathrm{e}}_\mathrm{n}^{(2)}\big)\big\rangle= \big\langle \Tilde{\mathrm{E}}^\textbf{in}; \tilde{\mathrm{e}}_\mathrm{n}^{(2)}\big\rangle.
    \end{align}
    First, we note that $\mathbb{M}^{(-\mathrm{k}  \delta)}_\mathrm{B}\big(\tilde{\mathrm{e}}_\mathrm{n}^{(2)}\big) = -\nabla\nabla\cdot\mathbb{N}^{(-\mathrm{k}  \delta)}_\mathrm{B}\big(\tilde{\mathrm{e}}_\mathrm{n}^{(2)}\big)$. Then from the identity $\nabla\nabla\cdot u = (\Delta + \textbf{curl}\;\textbf{curl}\; )u$ and as $\mathrm{e}_\mathrm{n}^{(2)} \in \mathbb{H}_{0}(\textbf{curl},0)$, we have $\mathbb{M}^{(-\mathrm{k}  \delta)}_\mathrm{B}\big(\tilde{\mathrm{e}}_\mathrm{n}^{(2)}\big) = \omega^2\mu_\mathrm{m}\delta^2 \mathbb{N}^{(-\mathrm{k}  \delta)}_\mathrm{B}\big(\tilde{\mathrm{e}}_\mathrm{n}^{(2)}\big) + \tilde{\mathrm{e}}_\mathrm{n}^{(2)}.$ Consequently, (\ref{project-en-2}) becomes
    \begin{align}
        \nonumber
        \big\langle \Tilde{\mathrm{E}}; \tilde{\mathrm{e}}_\mathrm{n}^{(2)}\big\rangle = \frac{1}{1+\varsigma} \big\langle \Tilde{\mathrm{E}}^\textbf{in}; \tilde{\mathrm{e}}_\mathrm{n}^{(2)}\big\rangle.
    \end{align}
    As, $\Tilde{\mathrm{E}}^\textbf{in} \in \mathbb{H}_{0}(\textbf{div})$ which is equal to $\mathbb{H}_{0}(\textbf{div},0) \bigoplus \nabla \mathbb{H}_{\textbf{arm}}$, orthogonal to $\mathbb{H}_{0}(\textbf{curl},0),$ we deduce
    \begin{align}\label{dieapp2}
        \big\Vert \overset{2}{\mathbb{P}}(\mathrm{\Tilde{E}})\big\Vert^2_{\mathbb{L}^2(\mathrm{B})} = \sum_\mathrm{n} \big|\big\langle \Tilde{\mathrm{E}}; \tilde{\mathrm{e}}_\mathrm{n}^{(2)}\big\rangle\big|^2 = 0.
    \end{align}

    \item We then consider the sub-space $\nabla \mathbb{H}_{\text{arm}}$ and we take the inner-product with respect to $\mathrm{e}_\mathrm{n}^{(3)}$ to write the equation (\ref{scaeq}) as follows
    \begin{align*}
      \big\langle \Tilde{\mathrm{E}}; \tilde{\mathrm{e}}_\mathrm{n}^{(3)}\big\rangle + \varsigma \big\langle \mathbb{M}^{(\mathrm{0})}_\mathrm{B}\big[\Tilde{\mathrm{E}}\big]; \tilde{\mathrm{e}}_\mathrm{n}^{(3)}\big\rangle = \big\langle \Tilde{\mathrm{E}}^\textbf{in}; \tilde{\mathrm{e}}_\mathrm{n}^{(3)}\big\rangle + \omega^2\mu_\mathrm{m}\varsigma \delta^2 \big\langle \mathbb{N}^{(\mathrm{k}\delta)}_\mathrm{B}\big[\Tilde{\mathrm{E}}\big]; \tilde{\mathrm{e}}_\mathrm{n}^{(3)}\big\rangle  + \varsigma \Big\langle \big(\mathbb{M}^{(\mathrm{k}\delta)}_\mathrm{B}-\mathbb{M}^{(\mathrm{0})}_\mathrm{B}\big)\big[\Tilde{\mathrm{E}}\big]; \tilde{\mathrm{e}}_\mathrm{n}^{(3)}\Big\rangle.
    \end{align*}
    We know that the magnetization potential is self-adjoint and  induces an eigen-system $\big(\lambda_\mathrm{n}^{(3)},\mathrm{e}_\mathrm{n}^{(3)}\big)_{\mathrm{n}\in \mathbb{N}}$ in $\nabla \mathbb{H}_\textbf{arm}.$, we deduce from the previous expression
    \begin{align}
        (1+\varsigma\lambda_\mathrm{n}^{(3)}\big\langle \Tilde{\mathrm{E}}; \tilde{\mathrm{e}}_\mathrm{n}^{(3)}\big\rangle = \big\langle \Tilde{\mathrm{E}}^\textbf{in}; \tilde{\mathrm{e}}_\mathrm{n}^{(3)}\big\rangle + \omega^2\mu_\mathrm{m}\varsigma \delta^2 \big\langle \mathbb{N}^{(\mathrm{k}\delta)}_\mathrm{B}\big[\Tilde{\mathrm{E}}\big]; \tilde{\mathrm{e}}_\mathrm{n}^{(3)}\big\rangle  + \varsigma \Big\langle \big(\mathbb{M}^{(\mathrm{k}\delta)}_\mathrm{B}-\mathbb{M}^{(\mathrm{0})}_\mathrm{B}\big)\big[\Tilde{\mathrm{E}}\big]; \tilde{\mathrm{e}}_\mathrm{n}^{(3)}\Big\rangle.
    \end{align}
    Furthermore, using the expression for $\mathbb{N}^{(\mathrm{k})}_\mathrm{B}$ in (\ref{1.20}), we deduce
    \begin{align}\label{app}
        \big\Vert \overset{3}{\mathbb{P}}(\mathrm{\Tilde{E}})\big\Vert^2_{\mathbb{L}^2(\mathrm{B})}  =  \sum_\mathrm{n}\frac{1}{|1+\varsigma\lambda_\mathrm{n}^{(3)}|^2} \big|\big\langle \Tilde{\mathrm{E}}^\textbf{in}; \tilde{\mathrm{e}}_\mathrm{n}^{(3)}\big\rangle\big|^2 + \sum_\mathrm{n}\frac{(\omega^2\mu_\mathrm{m})^2\varsigma^2 \delta^4}{|1+\varsigma \lambda_\mathrm{n}^{(3)}|^2}\big|\textbf{Error}^{(2)}_\mathrm{n}\big|^2 + \sum_\mathrm{n} \frac{\varsigma^2}{|1+\varsigma\lambda_\mathrm{n}^{(3)}|^2} \big|\textbf{Error}^{(3)}_\mathrm{n}\big|^2,
 \end{align}
 where we denote by $\textbf{Error}^{(2)}_\mathrm{n} := \big\langle \mathbb{N}^{(\mathrm{k}\delta)}_\mathrm{B}\big[\Tilde{\mathrm{E}}\big]; \tilde{\mathrm{e}}_\mathrm{n}^{(3)}\big\rangle$ and $\textbf{Error}^{(3)}_\mathrm{n} := \Big\langle \big(\mathbb{M}^{(\mathrm{k}\delta)}_\mathrm{B}-\mathbb{M}^{(\mathrm{0})}_\mathrm{B}\big)\big[\Tilde{\mathrm{E}}\big]; \tilde{\mathrm{e}}_\mathrm{n}^{(3)}\Big\rangle.$
 \\

 In a similar way as (\ref{Err3}), we obtain that
 \begin{align}\label{app1}
    \sum_\mathrm{n} \frac{\varsigma^2}{|1+\varsigma\lambda_\mathrm{n}^{(3)}|^2} \big|\textbf{Error}^{(3)}_\mathrm{n}\big|^2 \nonumber&\lesssim \frac{(\omega\mu_\mathrm{m}^\frac{1}{2}\delta)^4}{4}\Vert\Tilde{\mathrm{E}}\Vert^2_{\mathbb{L}^2(\mathrm{B})} + \frac{(\omega\mu_\mathrm{m}^\frac{1}{2}\delta)^6}{(12\pi)^2}\Vert\Tilde{\mathrm{E}}\Vert^2_{\mathbb{L}^2(\mathrm{B})} + \frac{(\omega\mu_\mathrm{m}^\frac{1}{2}\delta)^4}{4}\Vert\Tilde{\mathrm{E}}\Vert^2_{\mathbb{L}^2(\mathrm{B})} \\ &+ \frac{(\omega\mu_\mathrm{m}^\frac{1}{2}\delta)^8}{(4\pi)^2}\Vert\Tilde{\mathrm{E}}\Vert^2_{\mathbb{L}^2(\mathrm{B})} \underbrace{\sum_{\mathrm{j}\ge 3}\int_\mathrm{B}\int_\mathrm{B}\Bigg(\frac{\underset{\mathrm{x}}{\textbf{Hess}}(\Vert \mathrm{x}-\mathrm{y}\Vert^\mathrm{j})}{(\mathrm{j}+1)!}\Bigg)^2d\mathrm{y}d\mathrm{x}}_{<+\infty}.
\end{align}
Now, we estimate $\textbf{Error}^{(2)}_\mathrm{n}$ as follows
\begin{align}
    &\nonumber\textbf{Error}^{(2)}_\mathrm{n} \\ \nonumber&:= \big\langle \mathbb{N}^{(\mathrm{k}\delta)}_\mathrm{B}\big[\Tilde{\mathrm{E}}\big]; \tilde{\mathrm{e}}_\mathrm{n}^{(3)}\big\rangle
    \\ \nonumber &= \frac{\omega^2\mu_\mathrm{m}\delta^2}{2} \big\langle\mathbb{N}_\mathrm{B}^{(\mathrm{0})}\big[\Tilde{\mathrm{E}}\big];\Tilde{\mathrm{e}}^{(3)}_{\mathrm{n}}\big\rangle - \frac{i\omega^3\mu^2_\mathrm{m}\delta^3}{12\pi}\big\langle\int_\mathrm{B} \Tilde{\mathrm{E}}(\mathrm{\eta})d\mathrm{\eta};\Tilde{\mathrm{e}}^{(3)}_{\mathrm{n}}\big\rangle + \frac{1}{4\pi}\sum_{\mathrm{j}\ge 1} \frac{(i\omega\mu_\mathrm{m}^\frac{1}{2}\delta)^{\mathrm{j}+1}}{(\mathrm{j}+1)!} \big|\Big\langle\int_\mathrm{B} \Vert \cdot-\eta\Vert^\mathrm{j} \Tilde{\mathrm{E}}(\eta)d\eta; \tilde{\mathrm{e}}_\mathrm{n}^{(3)}\Big\rangle\big|.
\end{align}
We then use the Continuity of the Newtonian operator to obtain the following
\begin{align}\label{app2}
    \sum_\mathrm{n}\frac{(\omega^2\mu_\mathrm{m})^2\varsigma^2 \delta^4}{|1+\varsigma \lambda_\mathrm{n}^{(3)}|^2}\big|\textbf{Error}^{(2)}_\mathrm{n}\big|^2 &\nonumber\lesssim \frac{(\omega\mu_\mathrm{m}^\frac{1}{2}\delta)^4}{4}\Vert\Tilde{\mathrm{E}}\Vert^2_{\mathbb{L}^2(\mathrm{B})} + \frac{(\omega\mu_\mathrm{m}^\frac{1}{2}\delta)^6}{(12\pi)^2}\Vert\Tilde{\mathrm{E}}\Vert^2_{\mathbb{L}^2(\mathrm{B})} \\ &+ \frac{(\omega\mu_\mathrm{m}^\frac{1}{2}\delta)^8}{(4\pi)^2}\Vert\Tilde{\mathrm{E}}\Vert^2_{\mathbb{L}^2(\mathrm{B})} \underbrace{\sum_{\mathrm{j}\ge 1}\int_\mathrm{B}\int_\mathrm{B}\Bigg(\frac{(\Vert \mathrm{x}-\mathrm{y}\Vert^\mathrm{j})}{(\mathrm{j}+1)!}\Bigg)^2d\mathrm{y}d\mathrm{x}}_{<+\infty}
\end{align}
Consequently, inserting (\ref{app1}) and (\ref{app2}) in (\ref{app}), we deduce
\begin{align}\label{dieapp3}
  \big\Vert \overset{3}{\mathbb{P}}(\mathrm{\Tilde{E}})\big\Vert^2_{\mathbb{L}^2(\mathrm{B})}  =  \delta^4\big\Vert \overset{3}{\mathbb{P}}(\mathrm{\Tilde{E}}^\textbf{in})\big\Vert^2_{\mathbb{L}^2(\mathrm{B})} +  \delta^4\Vert\Tilde{\mathrm{E}}\Vert^2_{\mathbb{L}^2(\mathrm{B})}.
\end{align}
\end{enumerate}
Thus, using (\ref{dieapp1}), (\ref{dieapp2}) and (\ref{dieapp3}) into the following Parseval's identity, we deduce
\begin{align}\label{final}
    \big\Vert \mathrm{\Tilde{E}}\big\Vert^2_{\mathbb{L}^2(\mathrm{B})} \nonumber&= \sum_{\mathrm{j}=1}^3 \big\Vert \overset{\mathrm{j}}{\mathbb{P}}(\mathrm{\Tilde{E}})\big\Vert^2_{\mathbb{L}^2(\mathrm{B})}
    \\ \nonumber&=\frac{1}{\big|1-\omega^2\mu_\mathrm{m}\varsigma \delta^2\Tilde{\lambda}^{(1)}_{\mathrm{n_0}}\big|^2} \big|\big\langle \Tilde{\mathrm{E}}^\textbf{in}; \tilde{\mathrm{e}}_\mathrm{n_0}^{(1)}\big\rangle\big|^2 + \sum_{\mathrm{n} \ne \mathrm{n}_0}
    \frac{1}{\big|1-\omega^2\mu_\mathrm{m}\varsigma \delta^2\Tilde{\lambda}^{(1)}_{\mathrm{n}}\big|^2} \big|\big\langle \Tilde{\mathrm{E}}^\textbf{in}; \tilde{\mathrm{e}}_\mathrm{n}^{(1)}\big\rangle\big|^2 + (\omega^2\mu_\mathrm{m})^2\delta^{4-2\mathrm{h}} \big\Vert \mathrm{\Tilde{E}}\big\Vert^2_{\mathbb{L}^2(\mathrm{B})} \\ &+ \delta^4\big\Vert \overset{3}{\mathbb{P}}(\mathrm{\Tilde{E}}^\textbf{in})\big\Vert^2_{\mathbb{L}^2(\mathrm{B})}
\end{align}
or
\begin{align}
    \big(1-(\omega^2\mu_\mathrm{m})^2\delta^{4-2\mathrm{h}}\big) \big\Vert \mathrm{\Tilde{E}}\big\Vert^2_{\mathbb{L}^2(\mathrm{B})} \lesssim \frac{1}{\big|1-\omega^2\mu_\mathrm{m}\varsigma \delta^2\Tilde{\lambda}^{(1)}_{\mathrm{n_0}}\big|^2} \big|\big\langle \Tilde{\mathrm{E}}^\textbf{in}; \tilde{\mathrm{e}}_\mathrm{n_0}^{(1)}\big\rangle\big|^2 + \sum_{\mathrm{n} \ne \mathrm{n}_0}
    \frac{1}{\big|1-\omega^2\mu_\mathrm{m}\varsigma \delta^2\Tilde{\lambda}^{(1)}_{\mathrm{n}}\big|^2} \big|\big\langle \Tilde{\mathrm{E}}^\textbf{in}; \tilde{\mathrm{e}}_\mathrm{n}^{(1)}\big\rangle\big|^2.
\end{align}
We also have 
\begin{align}\label{curl}
    \big\langle \Tilde{\mathrm{E}}^\textbf{in}; \tilde{\mathrm{e}}_\mathrm{n_0}^{(1)}\big\rangle = \big\langle \Tilde{\mathrm{E}}^\textbf{in}; \textbf{curl}(\Tilde{\varphi}_\mathrm{n_0})\big\rangle = \big\langle \textbf{curl}(\Tilde{\mathrm{E}}^\textbf{in}); \varphi_\mathrm{n_0}\big\rangle = i\omega\mu_\mathrm{m}\delta \big\langle \Tilde{\mathrm{H}}^\textbf{in}; \varphi_\mathrm{n_0}\big\rangle,
\end{align}
where, we write $\tilde{\mathrm{e}}_\mathrm{n_0}^{(1)} = \textbf{curl}(\Tilde{\varphi}_\mathrm{n_0})$ as $\tilde{\mathrm{e}}_\mathrm{n_0}^{(1)} \in \mathbb{H}_{0}(\textbf{div},0) = \textbf{curl}\Big(\mathbb{H}_0(\textbf{curl})\cap \mathbb{H}(\textbf{div},0)\Big).$ 
\\
\noindent
Thus, we have the following a priori estimate for the electric field $\mathrm{E}$ when a dielectric nanoparticle occupy the domain $\Omega = \delta\mathrm{B}+\mathrm{z},\;\delta\ll1$
\begin{align}\label{aprdie}
    \big\Vert \mathrm{\Tilde{E}}\big\Vert^2_{\mathbb{L}^2(\mathrm{B})} \sim \delta^{2-2\mathrm{h}} \quad \text{for}\; \mathrm{h}<2.
\end{align}
Thus, utilizing the estimate (\ref{aprdie}) and the identity (\ref{curl}) we deduce from (\ref{final})
\begin{align*}
    \int_{\Omega}|\mathrm{E}|^2(\mathrm{y})d\mathrm{y} = \frac{\omega^2\mu^2_\mathrm{m}}{ |1 - \omega^2\mu_\mathrm{m}\varsigma\delta^2\lambda^{(1)}_{\mathrm{n}_{0}}|^2}\delta^5\Big|\mathrm{H}^{\textbf{in}}(\mathrm{z})\cdot\int_\mathrm{\mathrm{B}}\Tilde{\varphi}_{\mathrm{n}_{0}}(\mathrm{x})d\mathrm{x}\Big|^2 +
    \begin{cases}
      \mathcal{O}\big(\delta^5\big) \quad \quad \; \text{for} \quad \mathrm{h}\in(0,1).\\[10pt] 
      \mathcal{O}\big(\delta^{9-4\mathrm{h}}\big) \quad \text{for} \quad \mathrm{h}\in(1,2).
    \end{cases} 
\end{align*}
The proof of Theorem \ref{th13D}(2) is completed.

\section{Proof of Theorem \ref{mainth}}\label{heatmainth}

This section describes the asymptotic analysis of the solution to (\ref{heattran}) as $\delta\to 0$ when a Lorentzian nanoparticle occupy a bounded domain $\Omega = \mathrm{z} + \delta \mathrm{B}.$

\subsection{Mathematical Preliminaries}

We start this section by recalling the following classical singular estimates for the fundamental solution $\Phi(\mathrm{x},\mathrm{t};\mathrm{y},\tau):$
\begin{equation}\label{singularities}
\begin{cases}
    \big|\Phi(\mathrm{x},t;\mathrm{y},\tau)\big| \lesssim \frac{\alpha^\mathrm{r}}{(t-\tau)^r}\frac{1}{|\mathrm{x}-\mathrm{y}|^{3-2r}}, \ \ \ \mathrm{r}<\frac{3}{2},
\\[10pt]
    \big|\partial_{\mathrm{x}_{i}}\Phi(\mathrm{x},t;\mathrm{y},\tau)\big| \lesssim \frac{\alpha^\mathrm{r}}{(t-\tau)^\mathrm{r}}\frac{1}{|\mathrm{x}-\mathrm{y}|^{4-2\mathrm{r}}}, \ \ \ \mathrm{r}<\frac{5}{2},\ i=1,2,
\\[10pt]
    \big|\bm{\mathrm{D}_\nu}\Phi(\mathrm{x},t;\mathrm{y},\tau)\big| \lesssim \frac{\alpha^\mathrm{r}}{(t-\tau)^r}\frac{1}{|\mathrm{x}-\mathrm{y}|^{3-2r}}, \ \ \ \mathrm{r}<\frac{5}{2}, \mbox{ for } x, y \in \partial \Omega,
\\[10pt]
    \big|\partial_{t}\Phi(\mathrm{x},t;\mathrm{y},\tau)\big| \lesssim \frac{\alpha^{1-r}}{(t-\tau)^\mathrm{r}}\frac{1}{|\mathrm{x}-\mathrm{y}|^{5-2\mathrm{r}}}, \ \ \ r<\frac{5}{2},
\\[10pt]
    \big|\partial_{t}^{\frac{1}{2}}\Phi(\mathrm{x},t;\mathrm{y},\tau)\big| \lesssim \frac{\alpha^r}{(t-\tau)^\mathrm{r}}\frac{1}{|\mathrm{x}-\mathrm{y}|^{4-2\mathrm{r}}}, \ \ \ r<\frac{5}{2},
    \end{cases}
\end{equation}
for $0\le \tau\le \mathrm{t}\le \mathrm{T}$ and $\mathrm{x},\mathrm{y} \in \mathbb{R}^3$ with $\mathrm{x}\ne \mathrm{y}.$ It will be important to work in the environment of the anisotropic Sobolev spaces for our problem. We use $\mathrm{H}^{\frac{1}{2},\frac{1}{4}}$ to denote the Hilbert space
\begin{equation*}
    \mathrm{H}^{\frac{1}{2},\frac{1}{4}} := \mathrm{L}^2\big(\mathbb{R};\mathrm{H}^\frac{1}{2}(\mathbb{R}^3)\big) \cap \mathrm{H}^\frac{1}{4}\big(\mathbb{R};\mathrm{L}^2(\mathbb{R}^3)\big),
\end{equation*}
with the associated norm
\begin{equation*}
    \big\Vert u \big\Vert_{\mathrm{H}^{\frac{1}{2},\frac{1}{4}}}^2 := \big\Vert u \big\Vert_{\mathrm{L}^2\big(\mathbb{R};\mathrm{H}^\frac{1}{2}(\mathbb{R}^3)\big)}^2 + \big\Vert u \big\Vert_{\mathrm{H}^\frac{1}{4}\big(\mathbb{R};\mathrm{L}^2(\mathbb{R}^3)\big)}^2.
\end{equation*}
Analogously, we define, for $\partial\Omega \subseteq \mathbb{R}^3$, the following norm
\begin{align}
    \big\Vert u \big\Vert^2_{\mathrm{H}^{\frac{1}{2},\frac{1}{4}}\big(\partial\Omega\times\mathbb{R}\big)} = \int_\mathbb{R}\Vert u(\cdot,\mathrm{t})\Vert^2_{\mathrm{H}^\frac{1}{2}(\partial\Omega)}d\mathrm{t} + \int_\mathbb{R}\int_\mathbb{R}\dfrac{\Vert u(\cdot,t)-u(\cdot,\tau)\Vert_{\mathrm{L}^2(\partial\Omega)}^2}{|t-\tau|^\frac{3}{2}}d\mathrm{t} d\tau,
\end{align}
where
\begin{align}\nonumber
    \Vert u\Vert^2_{\mathrm{H}^{\frac{1}{2}}(\partial\Omega)} := \Vert u \Vert_{\mathrm{L}^2(\partial \Omega)}^2 + \int_{\partial\Omega}\int_{\partial\Omega}\frac{|u(\mathrm{x})-u(\mathrm{y})|^2}{|\mathrm{x}-\mathrm{y}|^{3}}d\sigma_\mathrm{x}d\sigma_\mathrm{y}.
\end{align}
We also recall some known properties of the boundary layers operators, volume and initial potentials for the heat operator. We refer to \cite{costabel, dohr, hofman, AM, lions2, Noonthe} for more details. 
\begin{lemma}
Let us consider $\Omega$ to be a bounded, open subset of $\mathbb{R}^3$ with a $\mathrm{C}^2$-boundary. Then
\begin{enumerate}
    \item The single layer heat operator $\mathcal{S}\big[u\big](\mathrm{x},t) := \frac{1}{\alpha}  \displaystyle\int_\mathbb{R}\int_{\partial\Omega} \Phi(\mathrm{x},t;\mathrm{y},\tau)\;u(y,\tau) d\sigma_{\mathrm{y}}d\tau,
$ maps $\mathrm{H}^{-\frac{1}{2},-\frac{1}{4}}\big(\partial\Omega\times\mathbb{R}\big) \to \mathrm{H}^{\frac{1}{2},\frac{1}{4}}\big(\partial\Omega\times\mathbb{R}\big)$ isomorphically.
    
    \item The following operators
    \begin{align*}
     \frac{1}{2} I + \mathcal{K} : \mathrm{H}^{\frac{1}{2},\frac{1}{4}}\big(\partial\Omega \times \mathbb{R}\big) \rightarrow \mathrm{H}^{\frac{1}{2},\frac{1}{4}}\big(\partial\Omega \times \mathbb{R}\big) \; \text{and}\; \;
     \frac{1}{2} I + \mathcal{K}^* : \mathrm{H}^{-\frac{1}{2},-\frac{1}{4}}\big(\partial\Omega \times \mathbb{R}\big) \rightarrow \mathrm{H}^{-\frac{1}{2},-\frac{1}{4}}\big(\partial\Omega \times \mathbb{R}\big)
     \end{align*}
     are invertible, where $\mathcal{K}$ and $\mathcal{K}^*$ are the double layer and adjoint double layer operator, which are defined as follows:\\
     $\mathcal{K}\big[u\big](\mathrm{x},t) := \frac{1}{\alpha}  \displaystyle\int_\mathbb{R}\int_{\partial\Omega} \bm{\mathrm{D}_{\nu_\mathrm{y}}}\Phi(\mathrm{x},t;\mathrm{y},\tau)\;u(y,\tau) d\sigma_{\mathrm{y}}d\tau$ ; $\mathcal{K}^*\big[u\big](\mathrm{x},t) := \frac{1}{\alpha}  \displaystyle\int_\mathbb{R}\int_{\partial\Omega} \bm{\mathrm{D}_{\nu_\mathrm{x}}}\Phi(\mathrm{x},t;\mathrm{y},\tau)\;u(y,\tau) d\sigma_{\mathrm{y}}d\tau,$ \\
     respectively.
     
     \item Furthermore, we refer to the Newtonian heat potential associated with the source term $f \in \mathrm{L}^2(\Omega \times \mathbb{R})$ as
    $$ \mathcal{V}\big[f\big](\mathrm{x},t) :=  \displaystyle\int_{-\infty}^{t}\int_{\Omega}\Phi(\mathrm{x},t;\mathrm{y},\tau)f(\mathrm{y},\tau) d\mathrm{y}d\tau. 
    $$
    \begin{enumerate}
        \item The Operator $\bm{\mathrm{D}^-_0}\mathcal{V}:\mathrm{L}^2(\Omega \times \mathbb{R}) \rightarrow \mathrm{H}^{\frac{1}{2},\frac{1}{4}}\big(\partial \Omega \times \mathbb{R}\big)$ defines a linear and bounded operator, where $\mathcal{V}$ is the Newtonian heat potential.
     
     \item The Operator $\bm{\mathrm{D}^-_\nu}\mathcal{V}:\mathrm{L}^2(\Omega \times \mathbb{R}) \rightarrow \mathrm{H}^{-\frac{1}{2},-\frac{1}{4}}\big(\partial \Omega \times \mathbb{R}\big)$ defines a linear and bounded operator, where $\mathcal{V}$ is the Newtonian heat potential.
    \end{enumerate}

     \item Also, let us define the initial heat potential for $f \in \mathrm{L}^2\big(\Omega \big)$ as follows
    \begin{align}
    \mathbb{I}[f](\mathrm{x},\mathrm{t}) = \int_{\Omega}\Phi(\mathrm{x},\mathrm{t};\mathrm{y})f(\mathrm{y})d\mathrm{y},
    \end{align}
    \begin{enumerate}
        \item The initial heat operator $\mathbb{I}:\mathrm{L}^2(\Omega) \rightarrow \mathrm{H}^{1,\frac{1}{2}}\big(\Omega \times \mathbb{R}\big)$ defines a linear and bounded operator.
    \end{enumerate}
     
\end{enumerate}
\end{lemma}
\noindent
In the following, we describe the integral operators' and the Sobolev spaces' essential scaling properties. The proofs of these properties can be obtained in a similar manner, with no essential difference, to those described in  \cite[Section 7.5]{AM} where they are performed in the 2-dimensional case.
\\

\noindent
We consider a nanoparticle that is located within $\Omega = \delta \mathrm{B} + \mathrm{z}\subseteq\mathbb{R}^3$, where $\mathrm{B}$ is centered at the origin and $|\mathrm{B}| \sim 1.$ Let us define the functions $\varphi$ and $\psi$ on $\partial\Omega\times \mathbb{R}$ and $\partial \mathrm{B}\times \mathbb{R}$, respectively, using the notation below
\begin{equation*}
\hat{\varphi}( \eta, \Tilde{\tau}) = \varphi^{\Lambda}( \eta, \Tilde{\tau}) := \varphi(\delta\eta + \mathrm{z}, \alpha\delta^2\Tilde{\tau}), \quad \quad \quad  \check{\psi}(\mathrm{x}, t) = \psi^\vee(\mathrm{x}, t) := \psi\big(\frac{\mathrm{x}-\mathrm{z}}{\delta}, \frac{t}{\alpha\delta^2}\big) \end{equation*}
for $(\mathrm{x},t) \in \partial\Omega \times\mathbb{R} $ and $(\eta,\Tilde{\tau}) \in \partial B\times \mathbb{R}$ respectively. Suppose $0<\delta\leq 1$ and $t := \alpha \delta^2\Tilde{t} $. Then, we assert the next two lemmas, which, respectively, correspond to the used function spaces and integral operators.
\begin{lemma}\label{l3.1}
\begin{enumerate}
\item for $\varphi \in \mathrm{H}^{\frac{1}{2},\frac{1}{4}}\big(\partial\Omega\times \mathbb{R}\big) $  and $\psi \in \mathrm{H}^{-\frac{1}{2},-\frac{1}{4}}\big(\partial\Omega\times \mathbb{R}\big)$, we have the following scales
\begin{align}
    \begin{cases}
    \alpha^\frac{1}{2} \delta^2\big\Vert\hat{\varphi}\big\Vert_{\mathrm{H}^{\frac{1}{2},\frac{1}{4}}\big(\partial \mathrm{B}\times \mathbb{R}\big)} \leq \big\Vert\varphi\big\Vert_{\mathrm{H}^{\frac{1}{2},\frac{1}{4}}\big(\partial\Omega\times \mathbb{R}\big)} \leq \alpha^{\frac{1}{4}}\delta^\frac{3}{2}\big\Vert\hat{\varphi}\big\Vert_{\mathrm{H}^{\frac{1}{2},\frac{1}{4}}\big(\partial \mathrm{B}\times \mathbb{R}\big)}
    \\ 
    \alpha^{\frac{3}{4}}\delta^\frac{5}{2}\big\Vert\hat{\psi}\big\Vert_{\mathrm{H}^{-\frac{1}{2},-\frac{1}{4}}\big(\partial \mathrm{B}\times \mathbb{R}\big) }\leq \big\Vert\varphi\big\Vert_{\mathrm{H}^{-\frac{1}{2},-\frac{1}{4}}\big(\partial\Omega\times \mathbb{R}\big)} \leq \alpha^{\frac{1}{2}}\delta^2\big\Vert\hat{\psi}\big\Vert_{\mathrm{H}^{-\frac{1}{2},-\frac{1}{4}}\big(\partial \mathrm{B}\times \mathbb{R}\big)}.
    \end{cases}
\end{align}
    \item for $\partial_t\varphi \in \mathrm{H}^{-\frac{1}{2},-\frac{1}{4}}\big(\partial \Omega\times\mathbb{R}\big)$ we have the following scales
    \begin{equation}
      \alpha^{-\frac{1}{4}}\delta^\frac{1}{2}\big\Vert \partial_{\Tilde{t}}\hat{\varphi}\big\Vert_{\mathrm{H}^{-\frac{1}{2},-\frac{1}{4}}\big(\partial \mathrm{B}\times \mathbb{R}\big)} \leq \big\Vert \partial_t\varphi\big\Vert_{\mathrm{H}^{-\frac{1}{2},-\frac{1}{4}}\big(\partial \Omega\times \mathbb{R}\big)} \leq  \alpha^{-\frac{1}{2}}\big\Vert \partial_{\Tilde{t}}\hat{\varphi}\big\Vert_{\mathrm{H}^{-\frac{1}{2},-\frac{1}{4}}\big(\partial \mathrm{B}\times \mathbb{R}\big)}.  
    \end{equation}
    \item for $\varphi \in H^{1,\frac{1}{2}}\big(\Omega \times\mathbb{R}_+\big) $ and $\psi \in H^{-1,-\frac{1}{2}}\big(\Omega\times\mathbb{R}_+\big) $ we have the following scales
    \begin{align}
        \begin{cases}
        \alpha^{\frac{1}{2}}\delta^\frac{5}{2}\big\Vert \hat{\varphi}\big\Vert_{\mathrm{H}^{1,\frac{1}{2}}\big(\mathrm{B}\times\mathbb{R}_+\big)} \leq \big\Vert \varphi\big\Vert_{\mathrm{H}^{1,\frac{1}{2}}\big(\Omega\times\mathbb{R}_+\big)} \leq \delta^\frac{3}{2}\big\Vert \hat{\varphi}\big\Vert_{\mathrm{H}^{1,\frac{1}{2}}\big( \mathrm{B}\times\mathbb{R}_+\big)}
        \\
        \alpha\delta^\frac{7}{2}\big\Vert \hat{\psi}\big\Vert_{H^{-1,-\frac{1}{2}}\big( \mathrm{B}\times\mathbb{R}_+\big)} \leq \big\Vert \psi\big\Vert_{H^{-1,-\frac{1}{2}}\big(\Omega\times\mathbb{R}_+\big)} \leq  \alpha^{\frac{1}{2}}\delta^\frac{5}{2}\big\Vert \hat{\psi}\big\Vert_{H^{-1,-\frac{1}{2}}\big(\mathrm{B}\times\mathbb{R}_+\big)}.
        \end{cases}
    \end{align}
\end{enumerate}
\end{lemma}

\begin{lemma}\label{l3.2}

\begin{enumerate}
    \item for $\varphi \in \mathrm{H}^{\frac{1}{2},\frac{1}{4}}\big(\partial \Omega\times \mathbb{R}\big)$ and $\psi \in \mathrm{H}^{-\frac{1}{2},-\frac{1}{4}}\big(\partial \Omega\times \mathbb{R}\big)$, we have the following estimate
\begin{align}
    \mathcal{S}_{\partial \Omega\times \mathbb{R}}\big[\psi\big](\mathrm{x},t) =  \delta \big(\Tilde{\mathcal{S}}_{\partial B\times \mathbb{R}}\big[\hat{\psi}\big]\big)^\vee \ \text{and} \
    \mathcal{S}^{-1}_{\partial \Omega\times \mathbb{R}}\big[\varphi\big](\mathrm{x},t) =  \delta^{-1} \big(\Tilde{\mathcal{S}}^{-1}_{\partial B\times \mathbb{R}}\big[\hat{\varphi}\big]\big)^\vee.
\end{align}
The following estimate is produced using the aforementioned estimates.
\begin{align}\label{operatoresti3D}
    \big\Vert \mathcal{S}^{-1}_{\partial \Omega\times \mathbb{R}} \big\Vert_{\mathcal{L}{\big( \mathrm{H}^{\frac{1}{2},\frac{1}{4}}\big(\partial \Omega\times \mathbb{R}\big), \mathrm{H}^{-\frac{1}{2},-\frac{1}{4}}\big(\partial \Omega\times \mathbb{R}\big)}\big) } \le \delta^{-1} \big\Vert \Tilde{\mathcal{S}}^{-1}_{\partial B\times \mathbb{R}} \big\Vert_{\mathcal{L}{\big( \mathrm{H}^{\frac{1}{2},\frac{1}{4}}\big(\partial B\times \mathbb{R}\big), \mathrm{H}^{-\frac{1}{2},-\frac{1}{4}}\big(\partial B\times \mathbb{R}}\big)\big)},
\end{align}
where, $\Tilde{\mathcal{S}}$ represents the single layer operator corresponding to the fundamental solution with $\alpha:=1$.
\item for $\psi \in \mathrm{H}^{-\frac{1}{2},-\frac{1}{4}}\big(\partial \Omega \times \mathbb{R}\big) $ we have
\begin{align}
    \alpha^\frac{1}{2} \delta^3\big\Vert \Tilde{\mathcal{S}}\big[\hat{\psi}\big]\big\Vert_{\mathrm{H}^{\frac{1}{2},\frac{1}{4}}\big(\partial B\times \mathbb{R}\big)} \leq \big\Vert \mathcal{S}\big[\psi\big]\big\Vert_{\mathrm{H}^{\frac{1}{2},\frac{1}{4}}\big(\partial \Omega\times \mathbb{R}\big)} \leq  \alpha^{\frac{1}{4}}\delta^\frac{5}{2} \big\Vert \Tilde{\mathcal{S}}\big[\hat{\psi}\big]\big\Vert_{\mathrm{H}^{\frac{1}{2},\frac{1}{4}}\big(\partial B\times \mathbb{R}\big)}.
\end{align}
\item for $\psi \in \mathrm{L}^2\big(\Omega\times\mathbb{R}\big) $ we have the following identities
\begin{align}
\begin{cases}    \mathcal{V}_{\Omega\times\mathbb{R}}\big[\psi\big](\mathrm{x},t) =  \delta^2 \big(\Tilde{\mathcal{V}}_{\mathrm{B}\times\mathbb{R}}\big[\hat{\psi}\big]\big)^\vee
\\
    \alpha^\frac{1}{2} \delta^4\big\Vert \Tilde{\mathcal{V}}\big[\hat{\psi}\big]\big\Vert_{\mathrm{H}^{\frac{1}{2},\frac{1}{4}}\big(\partial \mathrm{B}\times \mathbb{R}\big)} \leq \big\Vert \mathcal{V}\big[\psi\big]\big\Vert_{\mathrm{H}^{\frac{1}{2},\frac{1}{4}}\big(\partial \Omega\times \mathbb{R}\big)} \leq  \alpha^{\frac{1}{4}}\delta^\frac{7}{2} \big\Vert \Tilde{\mathcal{V}}\big[\hat{\psi}\big]\big\Vert_{\mathrm{H}^{\frac{1}{2},\frac{1}{4}}\big(\partial \mathrm{B}\times \mathbb{R}\big)}.
    \end{cases}
\end{align}
\item for $ \psi \in \mathrm{L}^{2}\big(\Omega \times \mathbb{R}\big) $ we have the following
\begin{align}
\begin{cases}
    \bm{\mathrm{D}_{\nu,\mathrm{x}}}\mathcal{V}_{\Omega \times\mathbb{R}}\big[\psi\big](\mathrm{x},t) =  \delta \big(\bm{\mathrm{D}_{\nu,\mathrm{\xi}}}\Tilde{\mathcal{V}}_{B\times\mathbb{R}}\big[\hat{\psi}\big]\big)^\vee
\\
    \alpha^{\frac{3}{4}} \delta^\frac{7}{2}\big\Vert \bm{\mathrm{D}_{\nu,\mathrm{\xi}}}\Tilde{\mathcal{V}}\big[\hat{\psi}\big]\big\Vert_{\mathrm{H}^{-\frac{1}{2},-\frac{1}{4}}\big(\partial \mathrm{B}\times \mathbb{R}\big)} \leq \big\Vert \bm{\mathrm{D}_{\nu,\mathrm{x}}}\mathcal{V}\big[\psi\big]\big\Vert_{\mathrm{H}^{-\frac{1}{2},-\frac{1}{4}}\big(\partial \Omega\times \mathbb{R}\big)}\leq  \alpha^{\frac{1}{^2}}\delta^3 \big\Vert \bm{\mathrm{D}_{\nu,\mathrm{\xi}}}\Tilde{\mathcal{V}}\big[\hat{\psi}\big]\big\Vert_{\mathrm{H}^{-\frac{1}{2},-\frac{1}{4}}\big(\partial \mathrm{B}\times \mathbb{R}\big)}.
    \end{cases}
\end{align}
\item for $\psi \in \mathrm{L}^2\big(\Omega\big) $ we have the following
\begin{align}
    \alpha^{\frac{1}{2}} \delta^\frac{5}{2}\big\Vert \Tilde{\mathbb{I}}\big[\hat{\psi}\big]\big\Vert_{\mathrm{H}^{1,\frac{1}{2}}\big(\mathrm{B}\times\mathbb{R}_+\big)} \leq \big\Vert \mathbb{I}\big[\psi\big]\big\Vert_{\mathrm{H}^{1,\frac{1}{2}}\big(\Omega\times\mathbb{R}_+\big)} \leq  \alpha^{\frac{1}{2}}\delta^\frac{3}{2} \big\Vert \Tilde{\mathbb{I}}\big[\hat{\psi}\big]\big\Vert_{\mathrm{H}^{1,\frac{1}{2}}\big(\mathrm{B}\times\mathbb{R}_+\big)}.
\end{align}
\end{enumerate}
\end{lemma}
\subsection{A Priori Estimate}
We begin this section by recalling the following boundary integral equations on $\partial\Omega$ as a solution to (\ref{heattran}), that are derived in \cite[Section 2]{AM}
\begin{align}\label{biequation}
 \begin{bmatrix}
\big(\frac{1}{2}I_{d} - \mathcal{K}_{\alpha_{\mathrm{p}}}^{*}\big)\big[\bm{\mathrm{D}^-_\nu}\mathrm{U}_{\mathrm{i}}\big](\mathrm{x},t) = \mathbb{H}_{\alpha_\mathrm{p}}\big[\bm{\mathrm{D}_0}\mathrm{U}_{\mathrm{i}}\big](\mathrm{x},t) + \bm{\mathrm{D}^-_\nu}\mathcal{V}\big[ f \big](\mathrm{x},t) \\
    \bm{\mathrm{D}^-_0}\mathrm{U}_{\mathrm{i}}(\mathrm{x},t) = -\frac{\gamma_{\mathrm{m}}}{\gamma_{\mathrm{p}}} \big(\frac{1}{2}I_{d} + \mathcal{K}_{\alpha_{\mathrm{p}}}\big)^{-1}\mathcal{S}_{\alpha_{\mathrm{p}}} \mathbb{A}^{\textbf{ext}}\big[\bm{\mathrm{D}^-_0}\mathrm{U}_{\mathrm{i}}\big](\mathrm{x},t) + \big(\frac{1}{2}I_{d} + \mathcal{K}_{\alpha_{\mathrm{p}}}\big)^{-1}\bm{\mathrm{D}^-_0}\mathcal{V}\big[ f\big](\mathrm{x},t)
    \end{bmatrix},
\end{align}
associated with the source term $f := \frac{\omega \cdot  \boldsymbol{\Im}(\varepsilon_\mathrm{p})}{2\pi\gamma_{\mathrm{p}}} |\mathrm{E}|^{2}\rchi_{(0,\mathrm{T_0})}$ and we denote $\mathbb{H}$ as the hyper-singular heat operator and the Steklov-Poincar\'e operator by $\mathbb{A}^{\textbf{ext}}$. Then, using the analysis described in \cite[Section 4.1]{AM}, we arrive at the following proposition, making the needed adjustments for the scaling properties indicated in Lemma \ref{l3.1} and Lemma \ref{l3.2}.

\begin{proposition}\label{heatapriori}
The a priori estimate for the solution of boundary integral equations (\ref{biequation}) is as follows:
\begin{align}
    \big\Vert \bm{\mathrm{D}^-_0}\mathrm{U}_{\mathrm{i}} \big\Vert_{\mathrm{H}^{\frac{1}{2},\frac{1}{4}}\big(\partial \Omega\times \mathbb{R}\big)} \lesssim \frac{\omega \cdot \boldsymbol{\Im}(\varepsilon_\mathrm{p})}{2\pi\gamma_{\mathrm{p}}}\alpha_{\mathrm{p}}^{\frac{1}{4}} \delta^2 \big\Vert|\mathrm{E}|^{2}\big\Vert_{\mathrm{L}^2\big(\Omega \big)}.
\end{align}
and
\begin{align}\label{gamma13D}
    \big\Vert \bm{\mathrm{D}^-_\nu}\mathrm{U}_\mathrm{i} \big\Vert_{\mathrm{H}^{-\frac{1}{2},-\frac{1}{4}}\big(\partial \Omega\times \mathbb{R}\big)} \lesssim \frac{\omega \cdot \boldsymbol{\Im}(\varepsilon_\mathrm{p})}{2\pi\gamma_{\mathrm{p}}}\big(\alpha^{\frac{1}{4}}_\mathrm{p}\delta + \alpha^{\frac{1}{2}}_\mathrm{p}\delta^{\frac{3}{2}}\big) \big\Vert|\mathrm{E}|^{2}\big\Vert_{\mathrm{L}^2\big(\Omega\big)}.
\end{align}
\end{proposition}

\subsection{Estimation of the Heat Potential's Dominating Term}
We begin by stating the integral representation formula for the exterior heat problem in (\ref{heattran}), see \cite{AM}. For $(\xi,t) \in \mathbb{R}^3\setminus\overline{\Omega} \times (0,\mathrm{T}_0)$
\begin{align}
   \mathrm{U}_{\mathrm{e}}(\xi,t) \nonumber &= -\frac{\gamma_{\mathrm{p}}}{\gamma_{\mathrm{m}}}\frac{1}{\alpha_\mathrm{m}}\int_{0}^t\int_{\partial\Omega}\Phi^{\textbf{e}}(\xi,t;\mathrm{y},\tau)\bm{\mathrm{D}^+_\nu}\mathrm{U}_{\mathrm{e}}(\mathrm{y},\tau)d\sigma_{\mathrm{y}}d\tau + \frac{1}{\alpha_\mathrm{m}}\int_{0}^t\int_{\partial\Omega}\bm{\mathrm{D}_{\nu_\mathrm{y}}}^{+}\Phi^{\textbf{e}}(\xi,t;\mathrm{y},\tau)\bm{\mathrm{D}^+_0}\mathrm{U}_{\mathrm{e}}(\mathrm{y},\tau)d\sigma_{\mathrm{y}}d\tau
   \\ \nonumber &= -\frac{\gamma_{\mathrm{p}}}{\gamma_{\mathrm{m}}}\frac{1}{\alpha_\mathrm{m}} \int_{0}^t\int_{\partial\Omega}\Phi^{\textbf{e}}(\xi,t;\mathrm{z},\tau)\bm{\mathrm{D}^-_\nu}\mathrm{U}_{\mathrm{i}}(\mathrm{y},\tau)d\sigma_\mathrm{y}d\tau 
   + \frac{1}{\alpha_\mathrm{m}}\int_{0}^t\int_{\partial\Omega}\bm{\mathrm{D}_{\nu_\mathrm{y}}}^{-}\Phi^{\textbf{e}}(\xi,t;\mathrm{y},\tau)\bm{\mathrm{D}^-_0}\mathrm{U}_{\mathrm{i}}(\mathrm{y},\tau)d\sigma_{\mathrm{y}}d\tau 
   \\ \nonumber &- \frac{\gamma_{\mathrm{p}}}{\gamma_{\mathrm{m}}}\frac{1}{\alpha_\mathrm{m}}\int_{0}^t\int_{\partial\Omega}\big[\Phi^{\textbf{e}}(\xi,t;\mathrm{z},\tau)-\Phi^{\textbf{e}}(\xi,t;\mathrm{y},\tau)\big]\bm{\mathrm{D}^-_\nu}\mathrm{U}_{\mathrm{i}}(\mathrm{y},\tau)d\sigma_{\mathrm{y}}d\tau. 
\end{align}
Then, we recall the following approximation for $\mathrm{U}_\mathrm{e}(\xi,\mathrm{t})$ obtained in \cite[Eq. 4.54]{AM}. For $(\xi,t) \in (\mathbb{R}^3\setminus \overline{\Omega})_{\mathrm{T}_0} $ and $\mathrm{z} \in \Omega$ the required formulation of the heat potential around the inserted Lorentzian nanoparticle is as follows
\begin{align}\label{mainappro}
     \mathrm{U}_{\mathrm{e}}(\xi,t) &= \frac{\gamma_{\mathrm{p}}}{\gamma_{\mathrm{m}}}\frac{1}{\alpha_\mathrm{m}}\bigg[\frac{\omega \cdot \boldsymbol{\Im}(\varepsilon_\mathrm{p})}{2\pi\gamma_{\mathrm{p}}}\int_{0}^{t} \Phi^{\textbf{e}}(\xi,t;z,\tau)d\tau\int_{\Omega}|\mathrm{E}|^{2}(\mathrm{y}) d\mathrm{y} + \textbf{err}^{(1)} + \textbf{err}^{(2)} + \textbf{err}^{(3)} + \textbf{err}^{(4)} + \textbf{err}^{(5)}\bigg],
\end{align}
where, 
\begin{align}
\begin{cases}
    \textbf{err}^{(1)} = \mathcal{O} \bigg(\delta^2\big\Vert \bm{\mathrm{D}^-_\nu}\mathrm{U}_{\mathrm{i}} \big\Vert_{\mathrm{H}^{-\frac{1}{2},-\frac{1}{4}}\big(\partial \Omega\times \mathbb{R}\big)}\big\Vert \nabla \Phi^{\textbf{e}}(\xi,t;\mathrm{z},\cdot) \big\Vert^{\frac{1}{2}}_{\mathrm{L}^2(0,\mathrm{t})}\big\Vert \nabla \Phi^{\textbf{e}}(\xi,t;\mathrm{z},\cdot) \big\Vert^{\frac{1}{2}}_{\mathrm{H}^{1,\frac{1}{2}}(0,\mathrm{t})}\bigg)
\\[5pt]
\textbf{err}^{(2)} := \mathcal{O} \bigg(\frac{\gamma_\mathrm{m}}{\gamma_\mathrm{p}}\big\Vert \bm{\mathrm{D}^-_0}\mathrm{U}_{\mathrm{i}}\big\Vert_{\mathrm{H}^{\frac{1}{2},\frac{1}{4}}\big(\partial \Omega \times \mathbb{R}\big)} \big\Vert \bm{\mathrm{D}_\nu}\Phi^{\textbf{e}}(\xi,t;\mathrm{y},\cdot)\big\Vert_{\mathrm{L}^2\big(\partial \Omega \times (0,\mathrm{t})\big)}\bigg),
\\[10pt]
    \textbf{err}^{(3)} := \mathcal{O}\bigg(\frac{\omega \cdot \boldsymbol{\Im}(\varepsilon_\mathrm{p})}{2\pi\gamma_{\mathrm{p}}}\alpha_\mathrm{p}^{\frac{1}{2}}\delta^\frac{3}{2} \big\Vert \Phi^{\textbf{e}}(\xi,t;\mathrm{z},\cdot) \big\Vert_{\mathrm{L}^2(0,\mathrm{T}_0)}\big\Vert |\mathrm{E}|^2 \big\Vert_{\mathrm{L}^2\big(\Omega\big)} \bigg)
\\[5pt]
    \displaystyle\textbf{err}^{(4)} := \mathcal{O}\bigg(\alpha_\mathrm{p}^{\frac{1}{4}} \delta^{\frac{9}{4}} \big\Vert \bm{\mathrm{D}^-_\nu}\mathrm{U}_{\mathrm{i}}\big\Vert_{\mathrm{H}^{-\frac{1}{2},-\frac{1}{4}}\big(\partial \Omega\times \mathbb{R}\big)} \big\Vert \Phi^{\textbf{e}}(\xi,t;z,\cdot)\big\Vert_{\mathrm{H}^{1,\frac{1}{2}}\big(\partial \Omega\times \mathbb{R}\big)}^\frac{1}{2}\sqrt{\int_{0}^{\mathrm{T}_0}\big\Vert \partial_\mathrm{s}\Phi^{\textbf{e}}(\xi,t;z,\cdot)\big\Vert_{\mathrm{L}^2(0,\mathrm{T}_0)}d\mathrm{t}}\bigg).
\\
    \textbf{err}^{(5)} := \mathcal{O}\bigg(\delta^2\big\Vert \bm{\mathrm{D}^-_0}\mathrm{U}_{\mathrm{i}}\big\Vert_{\mathrm{H}^{\frac{1}{2},\frac{1}{4}}\big(\partial \Omega\times \mathbb{R}\big)}\big\Vert \partial_\mathrm{t}\Phi^{\textbf{e}}(\xi,t;\mathrm{z},\cdot) \big\Vert_{\mathrm{L}^2(0,\mathrm{T}_0)}\bigg).
    \end{cases}
\end{align}
 Next, to estimate the dominant term and the error terms we modify appropriately the proofs given in \cite{AM} for the 2D case. The following lemma is the 3D version of \cite[Lemma 4.1]{AM}. We state it here and defer its proof until the appendix.
\begin{lemma}\label{hypersingularexpression}
 We set $\varphi(\mathrm{v},\mathrm{y},\mathrm{t},\tau)$ as follows:
\begin{equation}
    \varphi(\mathrm{v},\mathrm{y},\mathrm{t}, \tau) := \displaystyle\int_{0}^{\tau} \Big(\frac{\alpha_\mathrm{p}}{4\pi(\mathrm{s}-\tau)}\Big)^\frac{3}{2}\dfrac{2\pi |\mathrm{y}-\mathrm{v}| ^2}{(\mathrm{s}-\tau)}\textbf{exp}\big(-\dfrac{\alpha_\mathrm{p}|\mathrm{y}-\mathrm{v}| ^2}{4(\mathrm{s}-\tau)}\big)\Phi^{\textbf{e}}(\xi,\mathrm{t};\mathrm{z},\mathrm{s})d\mathrm{s}.
\end{equation}
Then we have
\begin{equation}
     \varphi(\mathrm{v},\mathrm{y},\mathrm{t}, \tau)-\Phi^{\textbf{e}}(\xi,\mathrm{t};\mathrm{z},\tau) = \mathcal{O}\bigg(\sqrt{\alpha_\mathrm{p}}|\mathrm{y}-\mathrm{v}| \ \Vert \partial_\mathrm{s}\Phi^{\textbf{e}}(\xi,\mathrm{t};\mathrm{z},\cdot)\Vert_{\mathrm{H}^{-\frac{1}{4}}(0,t)}\bigg),
\end{equation}
for $\mathrm{x}, \mathrm{y}$ such that $|\mathrm{v}-\mathrm{y}| \ll \mathrm{t}$ and $\mathrm{t} \in (0,\mathrm{T}]$ uniformly with respect to $\Omega$. 
\end{lemma}
\noindent
Instead of using \cite[Lemma 4.2]{AM}, we utilize its 3D version, see Lemma \ref{5.6} below, to derive the term $\textbf{err}^{(4)}$. The justification of Lemma \ref{5.6} can be found in \cite[Theorem 4.1]{graz1}.
\begin{lemma}\label{5.6}
 Let $\Gamma$ be the boundary of a bounded Lipschitz domain and $\Sigma = \partial\Omega \times (0,\mathrm{T})$. For $\varphi, \psi \in H^{\frac{1}{2},\frac{1}{4}}(\Sigma)$ there holds the integration by parts formula
\begin{align}
    \big \langle \mathbb{H}\varphi,\psi\big \rangle_{\Sigma} = \alpha^2\big \langle \textbf{curl}_{\Sigma}\psi,\mathcal{S}\big[\textbf{curl}_{\Sigma}\varphi\big]\big\rangle_{\Sigma} + \alpha \mathrm{b}(\varphi,\psi).
\end{align}
Here, the single layer boundary integral operator $\mathcal{S}$ is applied component-wise to the surface curl $\textbf{curl}_{\Sigma}\varphi:= \underline{\mathrm{n}} \times \nabla \varphi$ and the bi-linear form $b(\cdot,\cdot): H^{\frac{1}{2},\frac{1}{4}}(\Sigma) \times H^{\frac{1}{2},\frac{1}{4}}(\Sigma) \rightarrow \mathbb{R}$ is defined by 
\begin{align}
    b(\varphi,\psi) := \Big(\frac{\partial}{\partial \mathrm{t}} (\bm{\mathrm{D}^-_0})^*\Big(\mathcal{S}[\varphi\underline{\mathrm{n}}]\cdot\underline{\mathrm{n}}\Big)\Big)\big[\Tilde{\psi}\big] := -\big\langle \mathcal{S}[\varphi\underline{\mathrm{n}}], \frac{\partial}{\partial \mathrm{t}}\psi\underline{\mathrm{n}}\big\rangle_{\Sigma},
\end{align}
for $\varphi \in H^{\frac{1}{2},\frac{1}{4}}(\Sigma), \psi \in \bm{\mathrm{D}^-_0}\Big(C^{\infty}_{\mathrm{c}}(\mathbb{R}^2\times(0,\mathrm{T})\Big)$ and $\Tilde{\psi} \in C^{\infty}_{\mathrm{c}}(\mathbb{R}^2\times(0,\mathrm{T})$ such that $\psi = \Tilde{\psi}_{\Sigma}$, and as its continuous extension for general $\psi \in H^{\frac{1}{2},\frac{1}{4}}(\Sigma)$. $(\bm{\mathrm{D}^-_0})^*$ is the adjoint of the Dirichlet trace operator $\bm{\mathrm{D}^-_0}$, which needs to be understood in a distributional sense.
\end{lemma}

\noindent
Using classical singularities listed in (\ref{singularities}), we first note that
\begin{align}\label{er5}
    \big\Vert \partial_\mathrm{t}\Phi^{\textbf{e}}(\xi,t;\mathrm{z},\cdot) \big\Vert_{\mathrm{L}^2(0,\mathrm{T}_0)}^2 \nonumber &=\int_0^{\mathrm{T}_0} |\partial_\mathrm{t}\Phi^{\textbf{e}}(\xi,t;z,\tau)|^2d\tau 
    \\  &\lesssim \frac{\alpha_\mathrm{m}^{2\mathrm{r}}}{|\xi-z|^{10-4\mathrm{r}}} \int_{0}^{\mathrm{T_0}}\frac{1}{(t-\tau)^{2\mathrm{r}}}d\tau.
\end{align}
Let us now denote $\mathrm{K}^{(\mathrm{T_0})}_{\mathrm{r}}:= \sup_{\mathrm{t}\in(0,\mathrm{T}_0)}\displaystyle \int_{0}^{\mathrm{T_0}}\frac{1}{(t-\tau)^{2\mathrm{r}}}d\mathrm{t},\; \text{which only make sense for} \; \mathrm{r}<\frac{1}{2}.$ Consequently, we obtain from (\ref{er5}) and using Proposition \ref{heatapriori} that
\begin{align}\label{4.22}
    \textbf{err}^{(5)} \sim \frac{\omega \cdot \boldsymbol{\Im}(\varepsilon_\mathrm{p})}{2\pi\gamma_{\mathrm{p}}}\alpha_\mathrm{m}^\mathrm{r}\alpha_{\mathrm{p}}^{\frac{1}{4}} \delta^4 \big\Vert|\mathrm{E}|^{2}\big\Vert_{\mathrm{L}^2\big(\Omega \big)} \frac{1}{|\xi-z|^{5-2\mathrm{r}}}\sqrt{\mathrm{K}^{(\mathrm{T_0})}_{\mathrm{r}}}.
\end{align}
Again, based on the estimates (\ref{singularities}), we deduce that $\big\Vert \Phi^{\textbf{e}}(\xi,t;z,\cdot)\big\Vert_{\mathrm{H}^{1,\frac{1}{2}}\big(\partial \Omega\times \mathbb{R}\big)}^\frac{1}{2} \lesssim \frac{\alpha_\mathrm{m}^\frac{\mathrm{r}}{2}}{|\xi-z|^{2-\mathrm{r}}}\sqrt[4]{\mathrm{K}^{(\mathrm{T_0})}_{\mathrm{r}}}$ and 
\begin{align}
\nonumber
    \displaystyle\sqrt{\int_{0}^{\mathrm{T}_0}\big\Vert \partial_\mathrm{s}\Phi^{\textbf{e}}(\xi,t;z,\cdot)\big\Vert_{\mathrm{L}^2(0,\mathrm{T}_0)}d\mathrm{t}} \lesssim \frac{\alpha_\mathrm{m}^\frac{\mathrm{r}}{2}}{|\xi-z|^{\frac{5}{2}-\mathrm{r}}}\sqrt[4]{\mathrm{S}^{(\mathrm{T_0})}_{\mathrm{r}}}, \; \text{where}\; \mathrm{S}^{(\mathrm{T_0})}_{\mathrm{r}}:= \sup_{\mathrm{t}\in(0,\mathrm{T}_0)}\displaystyle \int_{0}^{\mathrm{T_0}}\int_{0}^{\mathrm{T_0}}\frac{1}{(t-\tau)^{2\mathrm{r}}}d\mathrm{t}d\tau,\; \\
    \text{which only makes sense for} \; \mathrm{r}<1.
\end{align}
Therefore, we obtain that 
\begin{align}\label{4.24}
    \textbf{err}^{(4)} \sim \frac{\omega \cdot \boldsymbol{\Im}(\varepsilon_\mathrm{p})}{2\pi\gamma_{\mathrm{p}}}\alpha_\mathrm{m}^\mathrm{r}\alpha_{\mathrm{p}}^{\frac{1}{4}} \delta^\frac{9}{4} \big\Vert|\mathrm{E}|^{2}\big\Vert_{\mathrm{L}^2\big(\Omega \big)} \frac{1}{|\xi-z|^{\frac{9}{2}-2\mathrm{r}}}\sqrt[4]{\mathrm{K}^{(\mathrm{T_0})}_{\mathrm{r}}\mathrm{S}^{(\mathrm{T_0})}_{\mathrm{r}}}.
\end{align}
In a similar way, we can show that
\begin{align}\label{4.25}
        \textbf{err}^{(3)} \sim \frac{\omega \cdot \boldsymbol{\Im}(\varepsilon_\mathrm{p})}{2\pi\gamma_{\mathrm{p}}}\alpha_\mathrm{m}^\mathrm{r}\alpha_{\mathrm{p}}^{\frac{1}{2}} \delta^\frac{3}{2} \big\Vert|\mathrm{E}|^{2}\big\Vert_{\mathrm{L}^2\big(\Omega \big)} \frac{1}{|\xi-z|^{3-2\mathrm{r}}}\sqrt{\mathrm{K}^{(\mathrm{T_0})}_{\mathrm{r}}},
\end{align}
\begin{align}\label{4.26}
        \textbf{err}^{(2)} \sim \frac{\omega \cdot \boldsymbol{\Im}(\varepsilon_\mathrm{p})}{2\pi\gamma_{\mathrm{p}}}\frac{\gamma_\mathrm{m}}{\gamma_\mathrm{p}}\alpha_\mathrm{m}^\mathrm{r}\alpha_{\mathrm{p}}^{\frac{1}{4}} \delta^3 \big\Vert|\mathrm{E}|^{2}\big\Vert_{\mathrm{L}^2\big(\Omega \big)} \frac{1}{|\xi-\cdot|^{3-2\mathrm{r}}}\sqrt{\mathrm{K}^{(\mathrm{T_0})}_{\mathrm{r}}},
\end{align}
and
\begin{align}\label{4.27}
        \textbf{err}^{(1)} \sim \frac{\omega \cdot \boldsymbol{\Im}(\varepsilon_\mathrm{p})}{2\pi\gamma_{\mathrm{p}}}\alpha_\mathrm{m}^\mathrm{r}\alpha_{\mathrm{p}}^{\frac{1}{4}} \delta^3 \big\Vert|\mathrm{E}|^{2}\big\Vert_{\mathrm{L}^2\big(\Omega \big)} \frac{1}{|\xi-\mathrm{z}|^{4-2\mathrm{r}}}\sqrt{\mathrm{K}^{(\mathrm{T_0})}_{\mathrm{r}}}. 
\end{align}
Now, we recall that we are in regimes where
\begin{equation}\label{condition}
    \gamma_p = \overline{\gamma}_p\; \delta^{-2}, ~~ \rho_{\mathrm{p}}\mathrm{c}_{\mathrm{p}} \sim 1, ~~\mbox{ and }~~ \alpha_\mathrm{m} \sim 1, ~~ \delta \ll 1.
\end{equation}
Therefore, assuming (\ref{condition}), we insert (\ref{4.22}), (\ref{4.24}), (\ref{4.25}), (\ref{4.26}), and (\ref{4.27}) in (\ref{mainappro}) to obtain the desired approximation for $\mathrm{U}_\mathrm{e}(\xi,\mathrm{t})$ with $\xi \in \mathbb{R}^3\setminus\overline{\Omega}$ and $\mathrm{z}\in \Omega$ such that $\text{dist}(\xi, \Omega)\sim \delta^p$ and therefore $\vert \xi- z\vert\sim \delta^p+\delta)$, where $\mathrm{p}\in [0,1)$
\begin{align}\label{finalappro}
     \mathrm{U}_{\mathrm{e}}(\xi,t) &= \frac{\omega \cdot \boldsymbol{\Im}(\varepsilon_\mathrm{p})}{2\pi\alpha_\mathrm{m}}\int_{0}^{t} \Phi^{\textbf{e}}(\xi,t;z,\tau)d\tau\int_{\Omega}|\mathrm{E}|^{2}(\mathrm{y}) d\mathrm{y}  + \mathcal{O}\Big(\boldsymbol{\Im}(\varepsilon_\mathrm{p})\delta^{\frac{5}{2}-\mathrm{p}(3-2\mathrm{r})} \big\Vert|\mathrm{E}|^{2}\big\Vert_{\mathrm{L}^2\big(\Omega \big)} \sqrt{\mathrm{K}^{(\mathrm{T_0})}_{\mathrm{r}}}\Big) .
\end{align}

\begin{proposition}\label{e12} 
We have the following a priori estimations of the electric fields. 
\begin{align}
\begin{cases}
        \big\Vert|\mathrm{E}|^{2}\big\Vert_{\mathrm{L}^2(\Omega)} = \mathcal{O}\big(\delta^{\frac{3}{2}-2\mathrm{h}}\big)\; 
 \text{for the plasmonic case},\\[10pt]
    \big\Vert|\mathrm{E}|^{2}\big\Vert_{\mathrm{L}^2(\Omega)} = \mathcal{O}\big(\delta^{\frac{5}{2}-\mathrm{h}}\big), \; \text{for the dielectric case},
\end{cases}
\end{align}
with $\mathrm{h}<2$.
\end{proposition}
\begin{proof}
See Section \ref{e23D} for the proof.
\end{proof}

\noindent
As a result, using this a priori knowledge, we infer the followings from (\ref{finalappro}) considering the two scenarios:
\begin{enumerate}
    \item Plasmonic case.
    \begin{align}\label{finalappro1}
     \mathrm{U}_{\mathrm{e}}(\xi,t) &= \frac{\omega \cdot \boldsymbol{\Im}(\varepsilon_\mathrm{p})}{2\pi\alpha_\mathrm{m}}\int_{0}^{t} \Phi^{\textbf{e}}(\xi,t;z,\tau)d\tau\int_{\Omega}|\mathrm{E}|^{2}(\mathrm{y}) d\mathrm{y}  + \mathcal{O}\Big(\boldsymbol{\Im}(\varepsilon_\mathrm{p})\delta^{4-2\mathrm{h}-\mathrm{p}(3-2\mathrm{r})} \sqrt{\mathrm{K}^{(\mathrm{T_0})}_{\mathrm{r}}}\Big).
\end{align}
     \item Dielectric case.
     \begin{align}\label{finalappro2}
     \mathrm{U}_{\mathrm{e}}(\xi,t) &= \frac{\omega \cdot \boldsymbol{\Im}(\varepsilon_\mathrm{p})}{2\pi\alpha_\mathrm{m}}\int_{0}^{t} \Phi^{\textbf{e}}(\xi,t;z,\tau)d\tau\int_{\Omega}|\mathrm{E}|^{2}(\mathrm{y}) d\mathrm{y}  + \mathcal{O}\Big(\boldsymbol{\Im}(\varepsilon_\mathrm{p})\delta^{5-\mathrm{h}-\mathrm{p}(3-2\mathrm{r})} \sqrt{\mathrm{K}^{(\mathrm{T_0})}_{\mathrm{r}}}\Big).
     \end{align}
\end{enumerate}
\noindent
Next, we recall the fundamental solution of the heat operator $\alpha_\mathrm{m} \frac{\partial}{\partial \mathrm{t}}-\Delta$ for the three dimensional spatial space
\begin{equation}\label{funda3D}
    \Phi^\mathrm{e}(\mathrm{x},t;\mathrm{y},\tau):=\  \begin{cases}
    \Big(\frac{\alpha_\mathrm{m}}{4\pi(t-\tau)}\Big)^\frac{3}{2}\textbf{exp}\big(-\frac{\alpha_\mathrm{m} |\mathrm{x}-\mathrm{y}|^2}{4(t-\tau)}\big), \ \ \ t > \tau \\
    0 ,\quad \text{otherwise}
    \end{cases}
\end{equation}
Let us consider the following integral
\begin{align}
    \mathrm{J}\nonumber&:= \int_{0}^\mathrm{t} \Phi^\mathrm{e}(\mathrm{\xi},t;\mathrm{z},\tau) d\tau
    \\ \nonumber &= \int_{0}^\mathrm{t} \Big(\frac{\alpha_\mathrm{m}}{4\pi(t-\tau)}\Big)^\frac{3}{2}\textbf{exp}\big(-\frac{\alpha_\mathrm{m} |\mathrm{x}-\mathrm{y}|^2}{4(t-\tau)}\big) d\tau
\end{align}
If we apply the change of variable $\mathrm{m} :=\frac{1}{2\sqrt{\mathrm{t}-\tau}}$, then it is evident that 
$
     d\tau = \mathrm{t} - \frac{1}{4\mathrm{m}^2}$ and then $d\tau = \frac{1}{2}\mathrm{m}^{-3}d\mathrm{m}.
$ 
We also have $\mathrm{m}^3 = \big(\frac{1}{4(t-\tau)}\big)^\frac{3}{2}$ . Consequently,
\begin{align}
    \nonumber&= \int_{\frac{1}{2\sqrt{t}}}^\infty e^{-\alpha_\mathrm{m} |\xi-\mathrm{z}|\mathrm{m}^2}\times \frac{1}{2}\mathrm{m}^{-3}\times {\big(\frac{\alpha_\mathrm{m}}{\pi}\big)}^\frac{3}{2}\mathrm{m}^{3}d\mathrm{m}
    \\ \nonumber &= \frac{1}{2}{\big(\frac{\alpha_\mathrm{m}}{\pi}\big)}^\frac{3}{2} \int_{\mathrm{a}}^\infty e^{-\mathrm{b}\;\mathrm{m}^2} d\mathrm{m},
\end{align}
where we define by $\mathrm{b}:= \alpha_\mathrm{m} |\xi-\mathrm{z}|^2\; \text{and}\; \mathrm{a}:= \frac{1}{2\sqrt{t}}.$ Now, we can derive that 
\begin{align}
    \int_{\mathrm{a}}^\infty e^{-\mathrm{b}\;\mathrm{m}^2} d\mathrm{m} = \frac{\sqrt{\pi}}{2\sqrt{b}}- \frac{\sqrt{\pi}}{2\sqrt{b}} \textbf{erf}(\sqrt{\mathrm{b}}\mathrm{a}),
\end{align}
where, for $|\xi-\mathrm{z}|\ll t$, we have error function's Maclaurin series as: $\textbf{erf}(\sqrt{\mathrm{b}}\mathrm{a}) = \frac{2}{\sqrt{\pi}}\displaystyle\sum_{\mathrm{j=0}}^\infty \frac{(-1)^\mathrm{j}(\sqrt{\mathrm{b}}\mathrm{a})^{2\mathrm{j}+1} }{\mathrm{j}!(2\mathrm{j}+1)}.$
Therefore, we obtain
\begin{align}\label{calculation}
    \mathrm{J}\nonumber&:= \int_{0}^\mathrm{t} \Phi^\mathrm{e}(\mathrm{\xi},t;\mathrm{z},\tau) d\tau
    \\ \nonumber &= \frac{1}{2}{\big(\frac{\alpha_\mathrm{m}}{\pi}\big)}^\frac{3}{2}\Bigg[\frac{\sqrt{\pi}}{2\sqrt{b}}- \frac{\sqrt{\pi}}{2\sqrt{b}} \textbf{erf}(\sqrt{\mathrm{b}}\mathrm{a})\Bigg]
    \\ \nonumber &= \frac{\alpha_\mathrm{m}}{4\pi|\xi-\mathrm{z}|}-  \frac{\alpha_\mathrm{m}}{4\pi|\xi-\mathrm{z}|}\times \frac{2}{\sqrt{\pi}}\displaystyle\sum_{\mathrm{j=0}}^\infty \frac{(-1)^\mathrm{j}(\sqrt{\mathrm{b}}\mathrm{a})^{2\mathrm{j}+1} }{\mathrm{j}!(2\mathrm{j}+1)}
    \\ &= \frac{\alpha_\mathrm{m}}{4\pi|\xi-\mathrm{z}|} + \mathcal{O}{(1)}.
\end{align}
We consider the next two scenarios in order to justify the final asymptotic expansion described in Theorem \ref{mainth}:
\begin{enumerate}
    \item Plasmonic Case. For the plasmonic nanoparticle, we have $\boldsymbol{\Im}(\varepsilon_\mathrm{p}) \sim \delta^\mathrm{h}$, $\mathrm{h}<2$. Then, considering the approximation formula for $\displaystyle\int_\Omega |\mathrm{E}|^2(\mathrm{y})d\mathrm{y}$, derived in Theorem \ref{th13D}(1) and using (\ref{calculation}) with the condition $\mathrm{r}<\frac{1}{2}$ and $2\mathrm{p}(1-\mathrm{r})<1 $ we obtain the following from (\ref{finalappro1}):
\begin{align*}
     \mathrm{U}_{\mathrm{e}}(\xi,\mathrm{t}) &= \frac{\omega \cdot \boldsymbol{\Im}(\varepsilon_\mathrm{p})}{8\pi^2\gamma_{\mathrm{m}}|\xi-\mathrm{z}|}\delta^{3-2\mathrm{h}}\Big|\mathrm{E}^{\textbf{in}}(\mathrm{z})\cdot\int_\mathrm{\mathrm{B}}\Tilde{\mathrm{e}}^{(3)}_{\mathrm{n}_{0}}(\mathrm{x})d\mathrm{x}\Big|^2 + 
     \begin{cases}
       \mathcal{O}(\delta^{4-\mathrm{h}-\mathrm{p}}) +  \mathcal{O}\Big(\delta^{4-\mathrm{h}-\mathrm{p}(3-2\mathrm{r})}\sqrt{\mathcal{K}^{(\mathrm{T_0})}_{\mathrm{r}}}\Big). \\[10pt] 
       \mathcal{O}(\delta^{7-3\mathrm{h}-\mathrm{p}}) +  \mathcal{O}\Big(\delta^{4-\mathrm{h}-\mathrm{p}(3-2\mathrm{r})}\sqrt{\mathcal{K}^{(\mathrm{T_0})}_{\mathrm{r}}}\Big).
     \end{cases}
\end{align*}
     \item Dielectric Case. We assume that $\boldsymbol{\Im}(\varepsilon_\mathrm{p}) \sim \delta^{\mathrm{h}-2},\; \delta\ll 1$ and $\mathrm{h}<2$. Considering the derived approximation formula for $\displaystyle\int_\Omega |\mathrm{E}|^2(\mathrm{y})d\mathrm{y}$ in Theorem \ref{th13D}(2), estimate in (\ref{calculation}), with the condition $\mathrm{r}<\frac{1}{2}$ and $2\mathrm{p}(1-\mathrm{r})<\mathrm{h}$, we obtain from (\ref{finalappro2}) the following:
\begin{align*}
       \mathrm{U}_{\mathrm{e}}(\xi,\mathrm{t}) &= \frac{\omega^3\mu^2_\mathrm{m} \cdot \boldsymbol{\Im}(\varepsilon_\mathrm{p})}{8\pi^2\gamma_{\mathrm{m}}|\xi-\mathrm{z}|}\delta^{5-2\mathrm{h}}\Big|\mathrm{H}^{\textbf{in}}(\mathrm{z})\cdot\int_\mathrm{\mathrm{B}}\Tilde{\varphi}_{\mathrm{n}_{0}}(\mathrm{x})d\mathrm{x}\Big|^2 + 
       \begin{cases}
         \mathcal{O}(\delta^{3+\mathrm{h}-\mathrm{p}})+ \mathcal{O}\Big(\delta^{3-\mathrm{p}(3-2\mathrm{r})}\sqrt{\mathcal{K}^{(\mathrm{T_0})}_{\mathrm{r}}}\Big). \\[10pt]
         \mathcal{O}(\delta^{7-3\mathrm{h}-\mathrm{p}})+ \mathcal{O}\Big(\delta^{3-\mathrm{p}(3-2\mathrm{r})}\sqrt{\mathcal{K}^{(\mathrm{T_0})}_{\mathrm{r}}}\Big).
       \end{cases}
\end{align*}
\end{enumerate}
Therefore it completes the proof of Theorem \ref{mainth}.


\section{Proof of Proposition \ref{e12}}\label{e23D}
In this section, we provide an a priori estimate for $\big\Vert |\mathrm{{E}}|^2\big\Vert_{\mathbb{L}^2(\mathrm{\Omega})}$ when we consider the plasmonic nanoparticle or the dielectric nanoparticle as the source of heat. This requires that we should have the $\mathbb{L}^2(\mathbb{R}^3)$-regularity of the given source term $\frac{\omega}{2\pi\gamma_{\mathrm{p}}}\boldsymbol{\Im}(\varepsilon_\mathrm{p})|\mathrm{E}|^{2}.$ As $\Im(\varepsilon)= 0\; \text{in}\; \mathbb{R}^3\setminus\overline{\Omega}$, we only require the $\mathbb{L}^4(\mathrm{\Omega})$-regularity of the electric field $\mathrm{E}.$ We refer to \cite[Section 3.4]{mourad} for more information regarding this needed regularity condition. We then consider the two cases related to the plasmonic and dielectric nanoparticles respectively.

\subsection{Plasmonic Case}
To begin, let us recall the following Lippmann-Schwinger equation with the contrast parameter $\varsigma := \varepsilon_\mathrm{p}(\omega)-\varepsilon_\mathrm{m}$
\begin{align*}
    \mathrm{E}(\mathrm{x}) + \varsigma \; \mathbb{M}^{(\mathrm{k})}\big[\mathrm{E}\big](\mathrm{x}) - \omega^2\mu_\mathrm{m}\varsigma \; \mathbb{N}^{(\mathrm{k})}\big[\mathrm{E}\big](\mathrm{x}) = \mathrm{E}^\textbf{in}(\mathrm{x}).
\end{align*}    
Now, with integration by parts and as $\nabla\cdot \mathrm{E}=0$, we show that $\mathbb{M}^{(\mathrm{k})}\big[\mathrm{E}\big] = \nabla \mathcal{S}^{(\mathrm{k})}\big[\nu\cdot\mathrm{E}\big]$. Consequently, we obtain that
\begin{align*}
    \mathrm{E}(\mathrm{x}) + \varsigma \; \nabla \mathcal{S}^{(\mathrm{k})}\big[\nu\cdot\mathrm{E}\big](\mathrm{x}) - \omega^2\mu_\mathrm{m}\varsigma \; \mathbb{N}^{(\mathrm{k})}\big[\mathrm{E}\big](\mathrm{x}) = \mathrm{E}^\textbf{in}(\mathrm{x}).
\end{align*}   
By scaling the prior equation to the domain $\mathrm{B}$, we arrive at the following expression:
 \begin{align} \label{2.12}
    \Tilde{\mathrm{E}} + \varsigma \; \nabla \mathcal{S}^{(\mathrm{k})}_\mathrm{B}\big[\nu\cdot\Tilde{\mathrm{E}}\big] - \omega^2\mu_\mathrm{m}\varsigma\delta^2 \; \mathbb{N}_\mathrm{B}^{(\mathrm{k})}\big[\Tilde{\mathrm{E}} \big] = \Tilde{\mathrm{E}}^\textbf{in}.
\end{align}
To write the closed system of equations, we need to take the Dirichlet trace of the above equations in $\mathrm{B}$ and  use the jump relation of the single-layer operator to get
\begin{align}
\nonumber
    \Big[\big(1+\frac{\varsigma}{2}\big)\mathrm{I}-\varsigma\mathcal{K}_\mathrm{B}^*\Big](\nu\cdot\Tilde{\mathrm{E}}) = \omega^2\mu_\mathrm{m}\varsigma\delta^2 \; \nu\cdot\mathbb{N}_\mathrm{B}^{(\mathrm{k})}\big[\Tilde{\mathrm{E}} \big] + \nu \cdot \Tilde{\mathrm{E}}^\textbf{in}
    \\ 
    \Rightarrow \big(1+\frac{\varsigma}{2}\big)\nu\cdot\Tilde{\mathrm{E}} = \varsigma\mathcal{K}_\mathrm{B}^*(\nu\cdot\Tilde{\mathrm{E}}) +   \omega^2\mu_\mathrm{m}\varsigma\delta^2 \; \nu\cdot\mathbb{N}_\mathrm{B}^{(\mathrm{k})}\big[\Tilde{\mathrm{E}} \big] +  \nu \cdot \Tilde{\mathrm{E}}^\textbf{in}.
\end{align}
Afterward, we obtain by taking into account the $\mathbb{H}^{\frac{1}{2}}$-norm on both sides of the aforementioned equation
\begin{align}
        \Rightarrow \big(1+\frac{\varsigma}{2}\big)\Big\Vert \nu\cdot\Tilde{\mathrm{E}}\Big\Vert_{\mathbb{H}^{\frac{1}{2}}(\partial\mathrm{B})} = \varsigma \Big\Vert\mathcal{K}_\mathrm{B}^*(\nu\cdot\Tilde{\mathrm{E}})\Big\Vert_{\mathbb{H}^{\frac{1}{2}}(\partial\mathrm{B})}  + \omega^2\mu_\mathrm{m}\varsigma\delta^2\; \Big\Vert\nu\cdot\mathbb{N}_\mathrm{B}^{(\mathrm{k})}\big[\Tilde{\mathrm{E}} \big] \Big\Vert_{\mathbb{H}^{\frac{1}{2}}(\partial\mathrm{B})}+   \Big\Vert\nu \cdot \Tilde{\mathrm{E}}^\textbf{in}\Big\Vert_{\mathbb{H}^{\frac{1}{2}}(\partial\mathrm{B})}.
\end{align}
Then, from the continuity of the operator $\mathcal{K}_\mathrm{B}^*: \mathbb{H}^{-\frac{1}{2}}(\partial\mathrm{B})\to \mathbb{H}^{\frac{1}{2}}(\partial\mathrm{B})$ we obtain
\begin{align}\label{1.26}
        \Rightarrow \Big\Vert \nu\cdot\Tilde{\mathrm{E}}\Big\Vert_{\mathbb{H}^{\frac{1}{2}}(\partial\mathrm{B})} \lesssim \big|\frac{\varsigma}{1+\frac{\varsigma}{2}}\big| \Big\Vert \nu\cdot\Tilde{\mathrm{E}}\Big\Vert_{\mathbb{H}^{-\frac{1}{2}}(\partial\mathrm{B})}+\frac{\omega^2\mu_\mathrm{m}\varsigma\delta^2}{\big|1+\frac{\varsigma}{2}\big|}\; \Big\Vert\nu\cdot\mathbb{N}_\mathrm{B}^{(\mathrm{k})}\big[\Tilde{\mathrm{E}} \big] \Big\Vert_{\mathbb{H}^{\frac{1}{2}}(\partial\mathrm{B})}+  \frac{1}{\big|1+\frac{\varsigma}{2}\big|} \Big\Vert\nu \cdot \Tilde{\mathrm{E}} ^\textbf{in}\Big\Vert_{\mathbb{H}^{\frac{1}{2}}(\partial\mathrm{B})}.
\end{align}
Now, from (\ref{2.12}), we estimate the following using the estimate (\ref{aprori3D})
\begin{align}\label{1.27}
    \big\Vert \textbf{curl}\;\mathrm{E}\big\Vert_{\mathbb{L}^2(\mathrm{B})}\nonumber &\lesssim \omega^2\mu_\mathrm{m}\varsigma\delta^2 \big\Vert \textbf{curl}\;\mathbb{N}_\mathrm{B}^{(\mathrm{k})}\big[\Tilde{\mathrm{E}} \big]\big\Vert_{\mathbb{L}^2(\mathrm{B})} + \big\Vert \textbf{curl}\;\Tilde{\mathrm{E}}^\textbf{in}\big\Vert_{\mathbb{L}^2(\mathrm{B})}
    \\ &\lesssim \omega^2\mu_\mathrm{m}\varsigma\delta^2 \big\Vert \Tilde{\mathrm{E}} \big\Vert_{\mathbb{L}^2(\mathrm{B})} + \big\Vert \Tilde{\mathrm{E}}^\textbf{in}\big\Vert_{\mathbb{L}^2(\mathrm{B})} = \mathcal{O}\big(\omega^2\mu_\mathrm{m}\varsigma\delta^{2-\mathrm{h}}\big) + \mathcal{O}(1) \sim 1.
\end{align}
Therefore, we obtain the following estimate
\begin{align}\label{1.28}
    \big\Vert \nu\cdot\Tilde{\mathrm{E}}\big\Vert_{\mathbb{H}^{-\frac{1}{2}}(\partial\mathrm{B})} \nonumber &\lesssim \big\Vert \Tilde{\mathrm{E}}\big\Vert_{\mathbb{H}_{\textbf{curl}}(\mathrm{B})}
    \nonumber\\ &\lesssim \Big(\big\Vert \Tilde{\mathrm{E}}\big\Vert_{\mathbb{L}^2(\mathrm{B})}^2 + \big\Vert \textbf{curl}\;\Tilde{\mathrm{E}}\big\Vert_{\mathbb{L}^2(\mathrm{B})}^2\Big)^{\frac{1}{2}} \sim \delta^{-\mathrm{h}}.
\end{align}
Furthermore, inserting the estimates (\ref{1.27}) and (\ref{1.28}) in (\ref{1.26}), taking into account (\ref{lorenz-model}), the fact that $\lambda_\mathrm{n}^{(3)}$ \footnote{These eigenvalues have $\frac{1}{2}$ as an accumulation point but they are different from it.} is different from $\frac{1}{2}$, and due to the smoothness of the Newtonian operator, we obtain
\begin{align}
        \Rightarrow \big\Vert \nu\cdot\Tilde{\mathrm{E}}\big\Vert_{\mathbb{H}^{\frac{1}{2}}(\partial\mathrm{B})} \nonumber&\lesssim \big|\frac{\varsigma}{1+\frac{\varsigma}{2}}\big| \delta^{-\mathrm{h}} +\frac{\omega^2\mu_\mathrm{m}\varsigma\delta^2}{\big|1+\frac{\varsigma}{2}\big|}\; \big\Vert \Tilde{\mathrm{E}}\big\Vert_{\mathbb{L}^2(\mathrm{B})}+  \frac{1}{\big|1+\frac{\varsigma}{2}\big|} \big\Vert \Tilde{\mathrm{E}} ^\textbf{in}\big\Vert_{\mathbb{L}^{2}(\mathrm{B})}
        \\ \nonumber &\lesssim \big|\frac{\varsigma}{1+\frac{\varsigma}{2}}\big| \delta^{-\mathrm{h}} + \frac{\omega^2\mu_\mathrm{m}\varsigma\delta^{2-\mathrm{h}}}{\big|1+\frac{\varsigma}{2}\big|} + \frac{1}{\big|1+\frac{\varsigma}{2}\big|} \sim \delta^{-\mathrm{h}}.
\end{align}
Then, based on the Calder\'on Zygmund inequalities and the traces properties, we deduce
\begin{align}\label{2.18}
    \big\Vert \Tilde{\mathrm{E}}\big\Vert_{\mathbb{H}^{1}(\mathrm{B})} \lesssim \underbrace{\big\Vert \Tilde{\mathrm{E}}\big\Vert_{\mathbb{L}^{2}(\mathrm{B})}}_{\sim \; \delta^{-\mathrm{h}}} + \underbrace{\big\Vert \textbf{curl}\;\mathrm{E}\big\Vert_{\mathbb{L}^2(\mathrm{B})}}_{\sim \; 1} + \underbrace{\big\Vert \textbf{div}\;\mathrm{E}\big\Vert_{\mathbb{L}^2(\mathrm{B})}}_{= \; 0} + \big\Vert \nu\cdot\Tilde{\mathrm{E}}\big\Vert_{\mathbb{H}^{\frac{1}{2}}(\partial\mathrm{B})} \sim \delta^{-\mathrm{h}}.
\end{align}
We have the following estimate based on Gagliardo-Nirenberg inequality, estimate (\ref{aprori3D}) and using (\ref{2.18})
\begin{align}\nonumber
\big\Vert \Tilde{\mathrm{E}}\big\Vert_{\mathbb{L}^\mathrm{4}(\mathrm{B})} \nonumber &\lesssim \big\Vert \Tilde{\mathrm{E}}\big\Vert_{\mathbb{L}^\mathrm{2}(\mathrm{B})}^{\frac{1}{2}}\big\Vert \Tilde{\mathrm{E}}\big\Vert_{\mathbb{H}^1(\mathrm{B})}^\frac{1}{2}
\\ &\lesssim \delta^{-\frac{\mathrm{h}}{2}} \cdot \delta^{-\frac{\mathrm{h}}{2}} \sim \delta^{-\mathrm{h}}.
\end{align}
So, using the aforementioned estimate and scaling it back to $\Omega$, we arrive at
\begin{align}\nonumber
    \big\Vert \mathrm{E}\big\Vert_{\mathbb{L}^\mathrm{4}(\mathrm{\Omega})} \sim \delta^{\frac{3}{4}-\mathrm{h}}.
\end{align}
Consequently, we derive the desired a priori estimate
\begin{align}
    \Big\Vert |\mathrm{E}|^2\Big\Vert_{\mathbb{L}^\mathrm{2}(\mathrm{\Omega})} \sim \delta^{\frac{3}{2}-2\mathrm{h}}.
\end{align}

\subsection{Dielectric Case} 
Similar to the plasmonic situation, we consider the same equation as was derived in (\ref{1.26}), recalling the fact that for dielectric nanoparticle, we have $\varsigma\sim \delta^{-2},\; \delta\ll1$
\begin{align}\label{d3}
        \Big\Vert \nu\cdot\Tilde{\mathrm{E}}\Big\Vert_{\mathbb{H}^{\frac{1}{2}}(\partial\mathrm{B})} \lesssim \big|\frac{\varsigma}{1+\frac{\varsigma}{2}}\big| \Big\Vert \nu\cdot\Tilde{\mathrm{E}}\Big\Vert_{\mathbb{H}^{-\frac{1}{2}}(\partial\mathrm{B})}+\frac{\omega^2\mu_\mathrm{m}\varsigma\delta^2}{\big|1+\frac{\varsigma}{2}\big|}\; \Big\Vert\nu\cdot\mathbb{N}_\mathrm{B}^{(\mathrm{k})}\big[\Tilde{\mathrm{E}} \big] \Big\Vert_{\mathbb{H}^{\frac{1}{2}}(\partial\mathrm{B})}+  \frac{1}{\big|1+\frac{\varsigma}{2}\big|} \Big\Vert\nu \cdot \Tilde{\mathrm{E}} ^\textbf{in}\Big\Vert_{\mathbb{H}^{\frac{1}{2}}(\partial\mathrm{B})}.
\end{align}
Following that, we first estimate $\Big\Vert \nu\cdot\Tilde{\mathrm{E}}\Big\Vert_{\mathbb{H}^{-\frac{1}{2}}(\partial\mathrm{B})}$. To accomplish that, we perform the following estimates utilizing the estimate (\ref{aprdie}) and the continuity of the Newtonian potential. First, we do
\begin{align}\label{d4}
\nonumber
    \big\Vert \textbf{curl}\;\mathrm{E}\big\Vert_{\mathbb{L}^2(\mathrm{B})}\nonumber &\lesssim \omega^2\mu_\mathrm{m}\varsigma\delta^2 \big\Vert \textbf{curl}\;\mathbb{N}_\mathrm{B}^{(\mathrm{k})}\big[\Tilde{\mathrm{E}} \big]\big\Vert_{\mathbb{L}^2(\mathrm{B})} + \big\Vert \textbf{curl}\;\Tilde{\mathrm{E}}^\textbf{in}\big\Vert_{\mathbb{L}^2(\mathrm{B})}
    \\ &\lesssim \omega^2\mu_\mathrm{m}\varsigma\delta^2 \big\Vert \Tilde{\mathrm{E}} \big\Vert_{\mathbb{L}^2(\mathrm{B})} + \big\Vert \Tilde{\mathrm{E}}^\textbf{in}\big\Vert_{\mathbb{L}^2(\mathrm{B})} = \mathcal{O}(\omega^2\mu_\mathrm{m}\delta^{1-\mathrm{h}}) + \mathcal{O}(1) \sim 1.
\end{align}
Consequently, using a priori estimate (\ref{aprdie}), we derive
\begin{align}\label{d5}
    \big\Vert \nu\cdot\Tilde{\mathrm{E}}\big\Vert_{\mathbb{H}^{-\frac{1}{2}}(\partial\mathrm{B})} \nonumber &\lesssim \big\Vert \Tilde{\mathrm{E}}\big\Vert_{\mathbb{H}_{\textbf{curl}}(\mathrm{B})}
    \nonumber\\ &\lesssim \Big(\big\Vert \Tilde{\mathrm{E}}\big\Vert_{\mathbb{L}^2(\mathrm{B})}^2 + \big\Vert \textbf{curl}\;\Tilde{\mathrm{E}}\big\Vert_{\mathbb{L}^2(\mathrm{B})}^2\Big)^{\frac{1}{2}} \sim 1.
\end{align}
Moreover, we use the estimate (\ref{d5}) and continuity of the Newtonian operator in the equation (\ref{d3}) to obtain
\begin{align}
        \Rightarrow \big\Vert \nu\cdot\Tilde{\mathrm{E}}\big\Vert_{\mathbb{H}^{\frac{1}{2}}(\partial\mathrm{B})} 
        \nonumber&\lesssim \big|\frac{\varsigma}{1+\frac{\varsigma}{2}}\big| +\frac{\omega^2\mu_\mathrm{m}\varsigma\delta^2}{\big|1+\frac{\varsigma}{2}\big|}\; \big\Vert \Tilde{\mathrm{E}}\big\Vert_{\mathbb{L}^2(\mathrm{B})}+  \frac{1}{\big|1+\frac{\varsigma}{2}\big|} \big\Vert \Tilde{\mathrm{E}} ^\textbf{in}\big\Vert_{\mathbb{L}^{2}(\mathrm{B})}
        \\ \nonumber &\lesssim \big|\frac{\varsigma}{1+\frac{\varsigma}{2}}\big|  + \frac{\omega^2\mu_\mathrm{m}\delta^{1-\mathrm{h}}}{\big|1+\frac{\varsigma}{2}\big|} + \frac{1}{\big|1+\frac{\varsigma}{2}\big|} \sim 1.
\end{align}
We then arrive to the following estimate using  Calder\'on Zygmund inequalities and the traces properties
\begin{align}\label{2.19}\nonumber
    \big\Vert \Tilde{\mathrm{E}}\big\Vert_{\mathbb{H}^{1}(\mathrm{B})} \lesssim \underbrace{\big\Vert \Tilde{\mathrm{E}}\big\Vert_{\mathbb{L}^{2}(\mathrm{B})}}_{\sim \; \delta^{1-\mathrm{h}}} + \underbrace{\big\Vert \textbf{curl}\;\mathrm{E}\big\Vert_{\mathbb{L}^2(\mathrm{B})}}_{\sim \; 1} + \underbrace{\big\Vert \textbf{div}\;\mathrm{E}\big\Vert_{\mathbb{L}^2(\mathrm{B})}}_{= \; 0} + \underbrace{\big\Vert \nu\cdot\Tilde{\mathrm{E}}\big\Vert_{\mathbb{H}^{\frac{1}{2}}(\partial\mathrm{B})}}_{\sim \; 1} \sim 1.
\end{align}
We have the following estimate based on Gagliardo-Nirenberg inequality, estimate (\ref{aprdie}) and using (\ref{2.19})
\begin{align}\nonumber
\big\Vert \Tilde{\mathrm{E}}\big\Vert_{\mathbb{L}^\mathrm{4}(\mathrm{B})} \nonumber &\lesssim \big\Vert \Tilde{\mathrm{E}}\big\Vert_{\mathbb{L}^\mathrm{2}(\mathrm{B})}^{\frac{1}{2}}\big\Vert \Tilde{\mathrm{E}}\big\Vert_{\mathbb{H}^1(\mathrm{B})}^\frac{1}{2}
\\ &\lesssim \delta^{\frac{1}{2}-\frac{\mathrm{h}}{2}} \cdot 1 \sim \delta^{\frac{1}{2}-\frac{\mathrm{h}}{2}}.
\end{align}
Thus, after scaling back to $\Omega$, we derive from the aforementioned estimate
\begin{align}\nonumber
    \big\Vert \mathrm{E}\big\Vert_{\mathbb{L}^\mathrm{4}(\mathrm{\Omega})} \sim \delta^{\frac{5}{4}-\frac{\mathrm{h}}{2}}
\end{align}
and hence
\begin{align}
    \Big\Vert |\mathrm{E}|^2\Big\Vert_{\mathbb{L}^\mathrm{2}(\mathrm{\Omega})} \sim \delta^{\frac{5}{2}-\mathrm{h}}.
\end{align}

\section{Appendix}\label{appen}

\subsection{Proof of Lemma \ref{hypersingularexpression}}
Let us start with defining the Double layer heat operator $\mathcal{K}$ corresponding to the density $\Phi^{\textbf{e}}(\xi,\mathrm{t};\mathrm{z},\cdot)$ by
\begin{align}
\nonumber
    \nonumber&\mathcal{K}\big[\Phi^{\textbf{e}}(\xi,\mathrm{t};\mathrm{z},\cdot)\big](\mathrm{y},\tau) 
    \\ \nonumber&=\frac{1}{\alpha}\displaystyle \int_{0}^{t}\int_{\partial\Omega}\frac{\alpha (\mathrm{v}-\mathrm{y})\cdot\nu_\mathrm{y}}{2(\tau-\mathrm{s})} \Big(\dfrac{\alpha}{4\pi(\tau-\mathrm{s})}\Big)^\frac{3}{2}\textbf{exp}\big(\dfrac{\alpha|\mathrm{v}-\mathrm{y}|^2}{4(\tau-\mathrm{s})}\big)\Phi^{\textbf{e}}(\xi,\mathrm{t};\mathrm{z},\mathrm{s})d\sigma_{\mathrm{v}}d\mathrm{s} 
    \\ &=\displaystyle\int_{\partial\Omega} \dfrac{(\mathrm{y}-\mathrm{v})\cdot \nu_\mathrm{y}}{4\pi}\frac{1}{|\mathrm{y}-\mathrm{v}|^3}\bigg[\displaystyle\int_{0}^{\tau}\Big(\dfrac{\alpha}{4\pi(\tau-\mathrm{s})}\Big)^\frac{3}{2}\dfrac{2\pi |\mathrm{v}-\mathrm{y}|^3}{(\tau-\mathrm{s})}\textbf{exp}\big(-\dfrac{\alpha|\mathrm{v}-\mathrm{y}|^2}{4(\tau-\mathrm{s})}\big)\Phi^{\textbf{e}}(\xi,\mathrm{t};\mathrm{z},\mathrm{s})d\mathrm{s}\bigg]d\sigma_{\mathrm{v}},
\end{align}
where $\alpha := \alpha_\mathrm{p} = \frac{\rho_{\mathrm{p}}C_{\mathrm{p}}}{\gamma_{\mathrm{p}}}$ and we set 
\begin{equation}\label{defvarphi3D}
 \varphi(\mathrm{v},\mathrm{y},\mathrm{t}, \tau) := \int_{0}^{\tau}\Big(\dfrac{\alpha}{4\pi(\tau-\mathrm{s})}\Big)^\frac{3}{2}\dfrac{2\pi |\mathrm{v}-\mathrm{y}|^3}{(\tau-\mathrm{s})}\textbf{exp}\big(-\dfrac{\alpha|\mathrm{v}-\mathrm{y}|^2}{4(\tau-\mathrm{s})}\big)\Phi^{\textbf{e}}(\xi,\mathrm{t};\mathrm{z},\mathrm{s})d\mathrm{s}.
\end{equation}
We then use the change of variable $\mathrm{m} :=\frac{\sqrt{\alpha}|\mathrm{y}-\mathrm{v}|}{2\sqrt{\tau-\mathrm{s}}}$, then it follows that $
    \mathrm{s} = \tau - \frac{\alpha|\mathrm{y}-\mathrm{v}|^2}{4\mathrm{m}^2}$ and then $d\mathrm{s} = \frac{1}{2}\alpha|\mathrm{y}-\mathrm{v}|^2\mathrm{m}^{-3}.
$ Therefore, we deduce
\begin{align}
    \nonumber
     \varphi(\mathrm{v},\mathrm{y},\mathrm{t}, \tau) 
     &= \displaystyle\int_{\frac{\sqrt{\alpha}|\mathrm{y}-\mathrm{v}|}{2\sqrt{\tau}}}^{\infty}\textbf{exp}(-\mathrm{m}^2)\Phi^{\textbf{e}}\big(\xi,\mathrm{t}; \mathrm{z},\tau - \dfrac{\alpha|\mathrm{y}-\mathrm{v}|^2}{4\mathrm{m}^2}\big)\frac{2\pi}{(\tau-\mathrm{s})}\mathrm{m}^3\frac{1}{2}\alpha|\mathrm{y}-\mathrm{v}|^2\mathrm{m}^{-3}d\mathrm{m} 
     \\ \nonumber&=\frac{4}{\sqrt{\pi}}\displaystyle\int_{\frac{\sqrt{\alpha}|\mathrm{y}-\mathrm{v}|}{2\sqrt{\tau}}}^{\infty}\textbf{exp}(-\mathrm{m}^2)\Phi^{\textbf{e}}\big(\xi,\mathrm{t}; \mathrm{z},\tau - \dfrac{\alpha|\mathrm{y}-\mathrm{v}|^2}{4\mathrm{m}^2}\big)\frac{\alpha|\mathrm{y}-\mathrm{v}|^2}{4(\tau-\mathrm{s})}d\mathrm{m} 
     \\ \nonumber&= \frac{4}{\sqrt{\pi}}\displaystyle\int_{\frac{\sqrt{\alpha}|\mathrm{y}-\mathrm{v}|}{2\sqrt{\tau}}}^{\infty}\mathrm{m}^2\textbf{exp}(-\mathrm{m}^2)\Phi^{\textbf{e}}\big(\xi,\mathrm{t}; \mathrm{z},\tau - \dfrac{\alpha|\mathrm{y}-\mathrm{v}|^2}{4\mathrm{m}^2}\big)d\mathrm{m}.
\end{align}
Now, we can show that $\displaystyle \int_0^\infty\mathrm{m}^2\textbf{exp}(-\mathrm{m}^2)d\mathrm{m}= \frac{\sqrt{\pi}}{4}$. Therefore, with this result, we rewrite the above equation as follows:
\begin{align}\label{varphi3D}
    \varphi(\mathrm{v},\mathrm{y},\mathrm{t}, \tau) - \Phi^{\textbf{e}}(\xi,\mathrm{t};\mathrm{z}, \tau ) \nonumber &= \displaystyle\int_{\frac{\sqrt{\alpha}|\mathrm{y}-\mathrm{v}|}{2\sqrt{\tau}}}^{\infty}\mathrm{m}^2\textbf{exp}(-\mathrm{m}^2)\bigg[\Phi^{\textbf{e}}\big(\xi,\mathrm{t}; \mathrm{z},\tau - \dfrac{\alpha|\mathrm{y}-\mathrm{v}|^2}{4\mathrm{m}^2}\big) - \Phi^{\textbf{e}}\big(\xi,\mathrm{t};\mathrm{z}, \tau \big)\bigg]d\mathrm{m}
    \nonumber
    \\ &- \frac{4}{\sqrt{\pi}}\Phi^{\textbf{e}}(\xi,\mathrm{t};\mathrm{z}, \tau )\displaystyle\int_{0}^{\frac{\sqrt{\alpha}|\mathrm{y}-\mathrm{v}|}{2\sqrt{\tau}}}\mathrm{m}^2\textbf{exp}(-\mathrm{m}^2) d\mathrm{m}.
\end{align}
We also have
\begin{align}\label{esti13D}
    \Phi^{\textbf{e}}(\xi,\mathrm{t};\mathrm{z}, \tau ) = \displaystyle\int_{\mathrm{t}}^\tau \partial_{s}\Phi^{\textbf{e}}(\xi,\mathrm{t};\mathrm{z}, \mathrm{s} )d\mathrm{s} \nonumber &\lesssim \Vert 1\Vert_{\mathrm{H}^{\frac{1}{4}}(\mathrm{t},\tau)}\Vert \partial_{s}\Phi^{\textbf{e}}(\xi,\mathrm{t};\mathrm{z}, \cdot )\Vert_{\mathrm{H}^{-\frac{1}{4}}(\mathrm{t},\tau)}
    \\ &\lesssim \tau^{\frac{1}{2}} \Vert \partial_{s}\Phi^{\textbf{e}}(\xi,\mathrm{t};\mathrm{z}, \cdot )\Vert_{\mathrm{H}^{-\frac{1}{4}}(0,\tau)}.
\end{align}
In a similar way as before, we write the following expression
\begin{equation}
\nonumber
    \Phi^{\textbf{e}}\big(\xi,\mathrm{t}; \mathrm{z},\tau - \dfrac{\alpha|\mathrm{y}-\mathrm{v}|^2}{4\mathrm{m}^2}\big) - \Phi^{\textbf{e}}\big(\xi,\mathrm{t};\mathrm{z}, \tau \big) = \displaystyle\int_{\tau}^{\tau - \frac{\alpha|\mathrm{y}-\mathrm{v}|^2}{4\mathrm{m}^2}}\partial_{s}\Phi^{\textbf{e}}\big(\xi,\mathrm{t};\mathrm{z}, \mathrm{s}\big)d\mathrm{s}.
\end{equation}
\noindent
We observe that $\mathrm{m}\ge \frac{\sqrt{\alpha}|\mathrm{y}-\mathrm{v}|}{2\sqrt{\tau}}$ which implies $\tau-\frac{\alpha|\mathrm{y}-\mathrm{v}|^2}{4\mathrm{m}^2} \ge 0$.
Hence we get 
\begin{align}\label{esti3D}
    \Phi^{\textbf{e}}\big(\xi,\mathrm{t}; \mathrm{z},\tau - \dfrac{\alpha|\mathrm{y}-\mathrm{v}|^2}{4\mathrm{m}^2}\big) - \Phi^{\textbf{e}}\big(\xi,\mathrm{t};\mathrm{z}, \tau \big) = \mathcal{O} \big(\frac{\sqrt{\alpha}|\mathrm{y}-\mathrm{v}|}{2\mathrm{m}}\Vert \partial_{s}\Phi^{\textbf{e}}(\xi,\mathrm{t};\mathrm{z}, \cdot )\Vert_{\mathrm{H}^{-\frac{1}{4}}(0,\tau)}\big).
\end{align}
Then, inserting (\ref{esti13D}) and (\ref{esti3D}) in (\ref{varphi3D}), we obtain
\begin{align}\label{varphi23D}
\nonumber
   \varphi(\mathrm{v},\mathrm{y},\mathrm{t}, \tau) - \Phi^{\textbf{e}}(\xi,\mathrm{t};\mathrm{z}, \tau ) &= \mathcal{O}\bigg(\frac{4}{\sqrt{\pi}}\displaystyle \int_{\frac{\sqrt{\alpha}|\mathrm{y}-\mathrm{v}|}{2\sqrt{\tau}}}^{\infty}\mathrm{m}\bm{\mathrm{e}}^{-\mathrm{m}^2}d\mathrm{m}\frac{\sqrt{\alpha}|\mathrm{y}-\mathrm{v}|}{2}\Vert \partial_{s}\Phi^{\textbf{e}}(\xi,\mathrm{t};\mathrm{z}, \cdot )\Vert_{\mathrm{H}^{-\frac{1}{4}}(0,\tau)}\bigg)
   \\ &+ \mathcal{O} \bigg(\frac{4}{\sqrt{\pi}}\displaystyle\int_{0}^{\frac{\sqrt{\alpha}|\mathrm{y}-\mathrm{v}|}{2\sqrt{\tau}}}\mathrm{m}^2\bm{\mathrm{e}}^{-\mathrm{m}^2}d\mathrm{m}\;\tau^{\frac{1}{2}}\Vert \partial_{s}\Phi^{\textbf{e}}(\xi,\mathrm{t};\mathrm{z}, \cdot )\Vert_{\mathrm{H}^{-\frac{1}{4}}(0,\tau)}\bigg).
\end{align}
We also have $\displaystyle\int_{\frac{\sqrt{\alpha}|\mathrm{y}-\mathrm{v}|}{2\sqrt{\tau}}}^{\infty}\mathrm{m}\bm{\mathrm{e}}^{-\mathrm{m}^2}d\mathrm{m} = \frac{1}{2}\bm{\mathrm{e}}^{-\frac{\alpha|\mathrm{y}-\mathrm{v}|^2}{4\tau}}$ and \\
$\displaystyle\int_0^{\frac{\sqrt{\alpha}|\mathrm{y}-\mathrm{v}|}{2\sqrt{\tau}}}\mathrm{m}^2\bm{\mathrm{e}}^{-\mathrm{m}^2}d\mathrm{m} = \frac{\sqrt{\pi}}{4}\bm{\textbf{erf}}(\frac{\sqrt{\alpha}|\mathrm{y}-\mathrm{v}|}{2\sqrt{\tau}})-\frac{1}{2}\frac{\sqrt{\alpha}|\mathrm{y}-\mathrm{v}|}{2\sqrt{\tau}}\bm{\mathrm{e}}^{-\frac{\alpha|\mathrm{y}-\mathrm{v}|^2}{4\tau}}$,\; where $"\bm{\textbf{erf}}"$ is the error function. Consequently, with this integral identities, we derive from (\ref{varphi23D})
\begin{align}
\nonumber
   &\nonumber\varphi(\mathrm{v},\mathrm{y},\mathrm{t}, \tau) - \Phi^{\textbf{e}}(\xi,\mathrm{t};\mathrm{z}, \tau ) \\ \nonumber&= \mathcal{O}\bigg(\frac{1}{\sqrt{\pi}}\displaystyle \sqrt{\alpha}|\mathrm{y}-\mathrm{v}| \bm{\mathrm{e}}^{-\frac{\alpha|\mathrm{y}-\mathrm{v}|^2}{4\tau}}\Vert \partial_{s}\Phi^{\textbf{e}}(\xi,\mathrm{t};\mathrm{z}, \cdot )\Vert_{\mathrm{H}^{-\frac{1}{4}}(0,\tau)}\bigg)+ \mathcal{O} \bigg(\tau^{\frac{1}{2}}\bm{\textbf{erf}}\big(\frac{\sqrt{\alpha}|\mathrm{y}-\mathrm{v}|}{2\sqrt{\tau}}\big)\Vert \partial_{s}\Phi^{\textbf{e}}(\xi,\mathrm{t};\mathrm{z}, \cdot )\Vert_{\mathrm{H}^{-\frac{1}{4}}(0,\tau)}\bigg) \\ &- \mathcal{O}\bigg(\frac{1}{\sqrt{\pi}}\displaystyle \sqrt{\alpha}|\mathrm{y}-\mathrm{v}| \bm{\mathrm{e}}^{-\frac{\alpha|\mathrm{y}-\mathrm{v}|^2}{4\tau}}\Vert \partial_{s}\Phi^{\textbf{e}}(\xi,\mathrm{t};\mathrm{z}, \cdot )\Vert_{\mathrm{H}^{-\frac{1}{4}}(0,\tau)}\bigg) .
\end{align}
Therefore, we obtain
\begin{align}
    \varphi(\mathrm{v},\mathrm{y},\mathrm{t}, \tau) - \Phi^{\textbf{e}}(\xi,\mathrm{t};\mathrm{z}, \tau) = \mathcal{O} \bigg(\tau^{\frac{1}{2}}\bm{\textbf{erf}}\big(\frac{\sqrt{\alpha}|\mathrm{y}-\mathrm{v}|}{2\sqrt{\tau}}\big)\Vert \partial_{s}\Phi^{\textbf{e}}(\xi,\mathrm{t};\mathrm{z}, \cdot )\Vert_{\mathrm{H}^{-\frac{1}{4}}(0,\tau)}\bigg).
\end{align}
Furthermore, considering the regime $|\mathrm{y}-\mathrm{v}|\ll \mathrm{t}$, and as we have error function's Maclaurin series as: $\textbf{erf}(\mathrm{a}) = \frac{2}{\sqrt{\pi}}\displaystyle\sum_{\mathrm{j=0}}^\infty \frac{(-1)^\mathrm{j}\mathrm{a}^{2\mathrm{j}+1} }{\mathrm{j}!(2\mathrm{j}+1)},$ we deduce
\begin{align}
    \varphi(\mathrm{v},\mathrm{y},\mathrm{t}, \tau) - \Phi^{\textbf{e}}(\xi,\mathrm{t};\mathrm{z}, \tau) = \mathcal{O} \big(\sqrt{\alpha}|\mathrm{y}-\mathrm{v}|\big)\Vert \partial_{s}\Phi^{\textbf{e}}(\xi,\mathrm{t};\mathrm{z}, \cdot )\Vert_{\mathrm{H}^{-\frac{1}{4}}(0,\tau)}\big).
\end{align}
 The proof of Lemma \ref{hypersingularexpression} is completed.

\end{document}